\documentclass[11pt]{article}
\usepackage{amsmath,amsthm,amssymb}
\usepackage[usenames,dvipsnames]{xcolor}
\usepackage{graphicx}
\usepackage{cite}
\usepackage{comment}
\usepackage{oands}
\usepackage{tikz}
\usepackage{changepage}
\usepackage{bbm}
\usepackage{mathtools}
\usepackage[margin=1in]{geometry}
\usepackage[pagewise,mathlines]{lineno}
\usepackage{appendix}
\usepackage{stmaryrd} 
\usepackage{multicol}
\usepackage{microtype}
\usepackage[colorinlistoftodos]{todonotes}
\usepackage[inline]{enumitem}

\usepackage{array}
    \newcolumntype{P}[1]{>{\centering\arraybackslash}p{#1}}
    \newcolumntype{M}[1]{>{\centering\arraybackslash}m{#1}}

\usepackage[normalem]{ulem}

\usepackage[pdftitle={Loewner evolution driven by complex Brownian motion},
  pdfauthor={Ewain Gwynne and Joshua Pfeffer},
colorlinks=true,linkcolor=NavyBlue,urlcolor=RoyalBlue,citecolor=PineGreen,bookmarks=true,bookmarksopen=true,bookmarksopenlevel=2,unicode=true,linktocpage]{hyperref}

\theoremstyle{plain}
\newtheorem{thm}{Theorem}[section]

\newtheorem{lem}[thm]{Lemma}
\newtheorem{prop}[thm]{Proposition}

\def\@rst #1 #2other{#1}
\newcommand\MR[1]{\relax\ifhmode\unskip\spacefactor3000 \space\fi
\MRhref{\expandafter\@rst #1 other}{#1}}
\newcommand{\MRhref}[2]{\href{http://www.ams.org/mathscinet-getitem?mr=#1}{MR#2}}

\theoremstyle{definition}
\newtheorem{defn}[thm]{Definition}
\newtheorem{remark}[thm]{Remark}

\newtheorem{ques}[thm]{Question}

\numberwithin{equation}{section}

\newcommand{\dsb}{\begin{adjustwidth}{2.5em}{0pt}
\begin{footnotesize}}
\newcommand{\dse}{\end{footnotesize}
\end{adjustwidth}}

\newcommand{\ssb}{\begin{adjustwidth}{2.5em}{0pt}}
\newcommand{\sse}{\end{adjustwidth}}

\newcommand{\aryb}{\begin{eqnarray*}}
\newcommand{\arye}{\end{eqnarray*}}
\def\alb#1\ale{\begin{align*}#1\end{align*}}
\def\allb#1\alle{\begin{align}#1\end{align}}
\newcommand{\eqb}{\begin{equation}}
\newcommand{\eqe}{\end{equation}}
\newcommand{\eqbn}{\begin{equation*}}
\newcommand{\eqen}{\end{equation*}}

\newcommand{\BB}{\mathbbm}
\newcommand{\ol}{\overline}

\newcommand{\op}{\operatorname}
\newcommand{\la}{\langle}
\newcommand{\ra}{\rangle}

\newcommand{\frk}{\mathfrak}
\newcommand{\eqD}{\overset{d}{=}}
\newcommand{\ep}{\varepsilon}
\newcommand{\rta}{\rightarrow}

\newcommand{\wt}{\widetilde}
\newcommand{\wh}{\widehat} 
\newcommand{\mcl}{\mathcal}

\newcommand{\bdy}{\partial}

\newcommand{\Cov}{\operatorname{Cov}}

\newcommand{\SLE}{{\operatorname{SLE}}_\kappa}
\newcommand{\SLEG}{{\operatorname{SLE}}_\Sigma}

\newcommand{\SLEGt}{{\operatorname{SLE}}_{\wt\Sigma}}

\newcommand{\dist}{{\operatorname{dist}}}

\newcommand{\cc}{{\mathbf{c}}}

\newcommand{\zz}{{\op{\mathbf{z}}}}
\newcommand{\tz}{{\mathfrak t}}
\newcommand{\sz}{{\mathfrak s}}

\let\originalleft\left
\let\originalright\right
\renewcommand{\left}{\mathopen{}\mathclose\bgroup\originalleft}
\renewcommand{\right}{\aftergroup\egroup\originalright}

\title{Loewner evolution driven by complex Brownian motion
}
\author{ \begin{tabular}{c}{Ewain Gwynne}\\[-4pt]\small Chicago\end{tabular}\qquad  \begin{tabular}{c}{Joshua Pfeffer}\\[-4pt]\small Columbia\end{tabular}\\
\hspace{6pt}
(with simulations by Minjae Park)
}

\date{   }

\begin{document}

\maketitle 

\begin{abstract}
We study the Loewner evolution whose driving function is $W_t = B_t^1 + i B_t^2$, where $(B^1,B^2)$ is a pair of Brownian motions with a given covariance matrix. This model can be thought of as a generalization of Schramm-Loewner evolution (SLE) with complex parameter values. We show that our Loewner evolutions behave very differently from ordinary SLE. For example, if neither $B^1$ nor $B^2$ is identically equal to zero, then the set of points disconnected from $\infty$ by the Loewner hull has non-empty interior at each time. We also show that our model exhibits three phases analogous to the phases of SLE: a phase where the hulls have zero Lebesgue measure, a phase where points are swallowed but not hit by the hulls, and a phase where the hulls are space-filling. The phase boundaries are expressed in terms of the signs of explicit integrals. These boundaries have a simple closed form when the correlation of the two Brownian motions is zero.
\end{abstract}

\tableofcontents

\bigskip
 
\noindent\textbf{Acknowledgments.} We thank an anonymous referee for helpful comments on an earlier version of this article. We thank Tom Kennedy and Scott Sheffield for helpful discussions. We thank Steffen Rohde for sharing with us an old email and Mathematica file from his work on this topic with Schramm in 2006. E.G.\ was partially supported by a Clay research fellowship. M.P was partially supported by NSF Award: DMS 1712862. J.P. was
partially supported by a National Science Foundation Postdoctoral Research Fellowship under Grant
No. 2002159.

\section{Introduction}
\label{sec-intro}

\subsection{Context and motivation}

The (chordal) Schramm-Loewner evolution (SLE)~\cite{schramm0} is a one-parameter family of random fractal curves in the plane, parametrized by $\kappa \in (0,\infty)$, that has been studied extensively in the past two decades. SLE describes the scaling limit of many discrete models that arise in statistical mechanics, such as the uniform spanning tree peano curve~\cite{lsw-lerw-ust}, Gaussian free field level lines~\cite{ss-contour}, critical percolation interfaces~\cite{smirnov-cardy,camia-newman-sle6}, and critical Ising model interfaces~\cite{smirnov-ising}. 
Moreover, it has deep connections with the Gaussian free field and with Liouville quantum gravity~\cite{ss-contour,dubedat-coupling,ig1,shef-zipper,wedges}. 
We refer to~\cite{lawler-book,bn-sle-notes,werner-notes} for introductory texts on SLE. 

SLE is obtained by solving the Loewner differential equation (see~\eqref{eqn-forward-loewner} below) with driving function given by $\sqrt\kappa$ times a standard linear Brownian motion.  
In this paper, we will consider a natural generalization of $\SLE$: namely, Loewner evolution driven by a \emph{complex-valued} Brownian motion, whose real and imaginary parts have a given covariance matrix $\Sigma$. 
There are several reasons to study Loewner evolutions driven by a complex Brownian motion. 
\begin{itemize}
\item Such Loewner evolutions have interesting and surprising properties which are quite different from the properties of Loewner evolutions with real driving functions. For example, if both the real and imaginary parts of the driving Brownian motion are not identically equal to zero, then each point $z\in\mathbb C$ is almost surely disconnected from $\infty$ by the Loewner hulls strictly before the time when it is absorbed into the hulls. In particular, the hulls are not simple curves and, at least for some times, their complements are not connected. See Theorem~\ref{thm-disconnects}. 
\item As is the case for SLE driven by real-valued Brownian motion, Loewner evolution driven by complex Brownian motion is in some ways ``exactly solvable". That is, one can get exact formulas for various quantities associated with the Loewner evolution using stochastic calculus. This is especially true in the case when the real and imaginary parts of the driving Brownian motion are independent. See, e.g., Theorem~\ref{thm-phases}.
\item There have been a number of recent works which extend results about Liouville quantum gravity (LQG) and related objects to the case of complex parameter values. Works on complex Gaussian multiplicative chaos (see, e.g.,~\cite{rhodes-vargas-complex-gmt,jsw-imaginary-gmc,jsw-decompositions,lacoin-complex-gmc}) allow one to extend the definition of the LQG area measure to complex parameter values. Huang~\cite{huang-complex-insertion} has investigated Liouville conformal field theory when the field has complex insertions (log singularities). Recent papers by the present authors and J.\ Ding~\cite{dg-supercritical-lfpp,pfeffer-supercritical-lqg,dg-confluence,dg-uniqueness} have constructed a metric associated with Liouville quantum gravity in the supercritical phase, which corresponds to complex values of the parameter $\gamma$ with $|\gamma|=2$. Due to the deep connections between SLE and LQG in the real-parameter case, it is natural to guess that there may be some connection between Loewner evolution driven by complex Brownian motion and some of these notions of ``complex LQG". However, we do not establish any such connection in this paper.  
\end{itemize}

\subsection{Basic definitions}

Let us now define the main objects that we will study in this paper.

\begin{defn}[Loewner chain with complex driving function]
\label{defn-loewner}
Let $U : [0,\infty) \rta \BB{C}$ be a continuous function.  We define the \emph{Loewner chain} with \emph{driving function} $U$ as the collection of mappings $(g_t)_{t \geq 0}$ such that, for each $z \in \BB{C}$, $g_t(z)$ is the solution to the initial value problem
\eqb
\label{eqn-forward-loewner}
\dot g_t(z) = \frac{2}{g_t(z) - U_t} ,\quad g_0(z) =z ,
\eqe
where a dot denotes the derivative with respect to $t$. 
We define the corresponding \emph{centered Loewner chain} $(f_t)_{t \geq 0}$ by 
\eqb
f_t(z) = g_t(z) - U_t .
\eqe
\end{defn}

Definition~\ref{defn-loewner} differs from the definition of ordinary chordal Loewner chains (Definition~\ref{defn-chordal-loewner}) in two key ways: $U$ may be complex-valued, and we do not restrict to points $z$ in the upper half-plane.

As in the case of Loewner chains with a real driving function, the solution to~\eqref{eqn-forward-loewner} is defined up to some time $T_z \in [0,\infty]$ depending on $z$, which we call the \emph{absorbing time}. We have $\lim_{t\to T_z^-} |f_t(z)| = 0$ (Lemma \ref{lem-endpt-cont}).
In contrast to forward and reverse chordal Loewner chains, the  Loewner chain defined in Definition~\ref{defn-loewner} describes not just one but \emph{two} families of sets $(L_t)_{t \geq 0}$ and $(R_t)_{t \geq 0}$.

\begin{defn} 
We define the \emph{left hull} $L_t$ as the complement of the domain of $f_t$---i.e., $L_t = \{z \in \BB{C} : T_z \leq t\}$---and the \emph{right hull} $R_t$ as the complement of the range of $f_t$.
This makes it so that
\eqb
f_t : \BB C\setminus L_t \rta \BB C\setminus R_t .
\eqe
\end{defn}

We refer to $L_t$ and $R_t$ as the left and right ``hulls'' to be consistent with the language of chordal Loewner chains. But, unlike in the case of a real-valued driving function, the complements of $L_t$ and $R_t$ may not be connected (see~\cite{lu-complex-loewner} or Theorem~\ref{thm-phases}).
By definition, the family of left hulls $(L_t)_{t\geq 0}$ is increasing in $t$, i.e., $L_s \subset L_t$ for $s<t$. The family of right hulls is typically not monotone in $t$. 

Using exactly the same arguments as in the case of a real driving function (see, e.g.,~\cite[Theorem 4.5]{lawler-book}) one can check that $g_t$ is conformal (i.e., holomorphic and bijective) from $\BB C\setminus L_t$ to $\BB C\setminus (R_t + U_t)$. Furthermore, $\lim_{t\rta\infty} z(g_t(z) - z) = 2t$. 

Definition~\ref{defn-loewner} directly extends the notions of forward and reverse chordal Loewner chains (with the maps extended to the negative half plane by Schwarz reflection). Roughly speaking, forward chordal Loewner chains correspond to Loewner chains with \emph{real} driving functions, and reverse chordal Loewner chains correspond to Loewner chains with \emph{purely imaginary} driving functions. We state this correspondence precisely in Lemma~\ref{lem-forward-reverse-ext}.
Loewner evolution with a complex driving function can in some sense be thought of as ``running forward and reverse Loewner evolution simultaneously". 
 
Loewner chains with complex driving functions are much less well-studied than Loewner chains with real driving functions. 
To our knowledge, the earliest study of Loewner evolution with a complex driving function was undertaken in unpublished (and unfinished) work of Rohde and Schramm~\cite{rs-complex-sle}, which sought to understand SLE in the case when the parameter $\kappa$ is complex (this is a special case of the setting in the present paper).
More recently, there have been a couple of works concerning Loewner evolution with a deterministic complex driving function~\cite{tran-complex-loewner,lu-complex-loewner}, which were primarily interested in conditions under which the Loewner evolution is generated by a simple curve. It was pointed out to us by R.\ Ugolini that Loewner evolution with a complex driving function fits into the general framework of so-called \emph{evolution families} in $\BB C$~\cite{bcd-evolution-families1}. 

\begin{remark} \label{remark-radial}
Instead of the chordal Loewner evolution considered in this paper, one could consider radial Loewner evolution, which describes hulls growing from the boundary of the unit disk and targeted at the origin. In the radial case, the driving function is typically taken to take values in the unit circle. In a similar vein to the present paper, one could look at radial Loewner evolution with a driving function which can take values in all of $\BB C$, rather than just in the unit circle. We expect that this variant of radial Loewner evolution has similar features to chordal Loewner evolution with a complex driving function, but we will not investigate it in this paper.
\end{remark}
 
In this work, we will focus on the case when the driving function $U$ is a complex-valued Brownian motion. 

\begin{defn}
\label{defn-sle}
Let
\eqb \label{eqn-cov-matrix}
\Sigma = \begin{pmatrix} a&c\\c&b\end{pmatrix} .
\eqe
be a $2 \times 2$ symmetric positive semidefinite matrix. We define $\SLEG$ to be the Loewner chain whose driving function is a Brownian motion in $\BB{C}$ with infinitesimal covariance $\Sigma$. That is, $U_t = \sqrt a B_t + i \sqrt b \wt B_t$, where $B_t$ and $\wt B_t$ are real standard linear Brownian motions with $\Cov(B_t,\wt B_t) = ct/\sqrt{ab}$ for each $t\geq 0$. 
\end{defn}
 
The condition that $\Sigma$ is symmetric and positive semidefinite is equivalent to the conditions $a,b \geq 0$ and $|c| \leq \sqrt{ab}$.  We can rephrase the definition of $\SLEG$ in terms of $a,b,c$ as the collection of maps $(g_t)_{t \geq 0}$ satisfying
\eqb
\label{eqn-def-gt}
\dot g_t(z) = \frac{2}{g_t(z) - \sqrt{a} B_t - i \sqrt{b} \wt B_t}, \quad g_0(z) = z
\eqe
where $B_t$ and $\wt B_t$ are as in Definition~\ref{defn-sle}.  Equivalently, the centered $\SLEG$ maps $f_t$ satisfy, for each $z \in \BB{C}$, the stochastic differential equation
\begin{equation}
df_t(z) = \frac{2}{f_t(z)}\, dt  - \sqrt{a}\, d B_t - i \sqrt{b}\, d\wt B_t ,\quad f_0(z) = z .
\label{eqn-def-ft}
\end{equation}

The cases when $c = 0$ or $|c| =\sqrt{ab}$ are particularly interesting. The condition $c=0$ means that the real and imaginary parts of the driving Brownian motion are independent, and appears to be the most ``solvable" case (see, e.g., Theorem~\ref{thm-phases}). The condition $|c| = \sqrt{ab}$ means that the driving Brownian motion is a complex multiple of a real Brownian motion. The Loewner evolution in this case can be interpreted as SLE$_\kappa$ with $\kappa\in\BB C$ (see also Remark~\ref{remark-duality}). We show some simulations of $\SLEG$ with $c=0$ and with $c=\sqrt{ab}$ in Figures~\ref{fig-sim1} and~\ref{fig-sim2}. 

We note that replacing $c$ by $-c$ in~\eqref{eqn-cov-matrix} has the effect of replacing the driving function $U = B + i\wt B$ by its conjugate. By~\eqref{eqn-forward-loewner}, this has the effect of conjugating the left and right hulls. Most of the properties of the hulls that we are interested in are invariant under conjugation, so we will typically assume that $c\geq 0$.

Like ordinary SLE, $\SLEG$ satisfies a scale invariance property (Lemma~\ref{lem-scaling}) and a Markov property (Lemma~\ref{lem-markov}), both of which follow easily from the scale invariance and Markov property of the driving Brownian motion.  Moreover, it is easy to see from the reversibility of Brownian motion that $\SLEG$ satisfies the following duality property, which is an extension of the duality between forward and reverse SLE$_\kappa$. We will prove these properties in Section~\ref{sec-basics-sleg}.  

\begin{prop}[$\SLEG$ duality] \label{prop-duality}
Let $(f_t)_{t \geq 0}$ be a centered $\SLEG$ Loewner chain and $(\wt f_t)_{t \geq 0}$ a centered $\SLEGt$ Loewner chain, where  $\Sigma = \begin{pmatrix} a&c\\c&b\end{pmatrix}$ and  $\wt \Sigma = \begin{pmatrix} b&-c\\-c&a\end{pmatrix}$.  Then, for each fixed $t>0$, the maps $z \mapsto i f_t^{-1}(-i z)$ is equal in law to $\wt f_t$.
In particular, the corresponding left and right hulls satisfy
\eqb
\wt L_t \eqD i R_t \quad \text{and} \quad \wt R_t \eqD i L_t .
\eqe
\end{prop}

In many cases, once a certain property has been proven for the left hulls of $\SLEG$, Proposition~\ref{prop-duality} allows us to immediately extend this property to the right hulls. In light of this, throughout most of the paper we will focus on left hulls.

\begin{remark}[The special case when $c = \sqrt{ab}$]
\label{remark-duality}
If $c = \sqrt{ab}$, which is its maximum possible value given $a$ and $b$, then $B_t = \wt B_t$, and we can express~\eqref{eqn-def-ft} as
\eqb
\label{eqn-sle-kappa-def}
df_t(z) = \frac{2}{f_t(z)} \,  dt  - \sqrt{\kappa}  \, d B_t, \quad f_0(z) = z ,
\eqe
where $\kappa \in \BB{C}$ satisfies $\sqrt{\kappa} = \sqrt{a} + i \sqrt{b}$. In other words, our definition of $\SLEG$ gives, as a special case, an extension of the notion of $\SLE$ to complex values of $\kappa$. 

In this special case, Proposition~\ref{prop-duality} has the following intriguing formulation: if $(L_t)_{t\geq 0}$ are the left hulls for SLE$_\kappa$, $\kappa\in\BB C$, and $(\wt R_t)_{t\geq 0}$ are the right hulls for SLE$_{-\kappa}$, then for each fixed $t>0$ we have $\wt R_t \eqD i L_t$. The same is true with left and right hulls interchanged.
We can also phrase this duality relation in terms of the \emph{central charge} associated with SLE$_\kappa$, which is defined for $\kappa \in (0,\infty)$ by
\eqb \label{eqn-cc}
\cc(\kappa) = \frac{(6-\kappa) (3\kappa-8)}{2\kappa} .
\eqe
By analytically extending this relation between $\cc$ and $\kappa$ to all $\kappa \in \BB{C}$, we can associate a value of central charge $\cc$ to each $\kappa \in \BB{C}$.
One has $\cc(\kappa) + \cc(-\kappa) = 26$, so the aforementioned duality relation between SLE$_\kappa$ and SLE$_{-\kappa}$ gives a duality between SLE with central charge $\cc$ and SLE with central charge $26-\cc$. 

The central charge~\eqref{eqn-cc} is real if and only if $\kappa \in \BB R$ or $\kappa \in \BB C$ with $|\kappa| = 4$. When $|\kappa| = 4$, we have $\cc(\kappa) \in [1,25]$, which is the same as the range of matter central charge values for supercritical LQG in~\cite{dg-supercritical-lfpp,pfeffer-supercritical-lqg,dg-confluence,dg-uniqueness}. In light of this, it is natural to think that $\SLEG$ might have some special behavior in this case, i.e., when $a^2 + b^2 = 4$ and $c = \sqrt{ab}$. We have not yet observed any such special behavior.
\end{remark}

\begin{figure}

\begin{center}
\begin{tabular}{|c|c|c|}
\hline
$\Sigma$ & \textbf{Left hull} & \textbf{Right hull}\\  \hline
\includegraphics[width=0.1\textwidth]{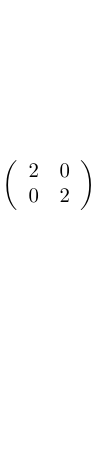} &\includegraphics[width=0.4\textwidth]{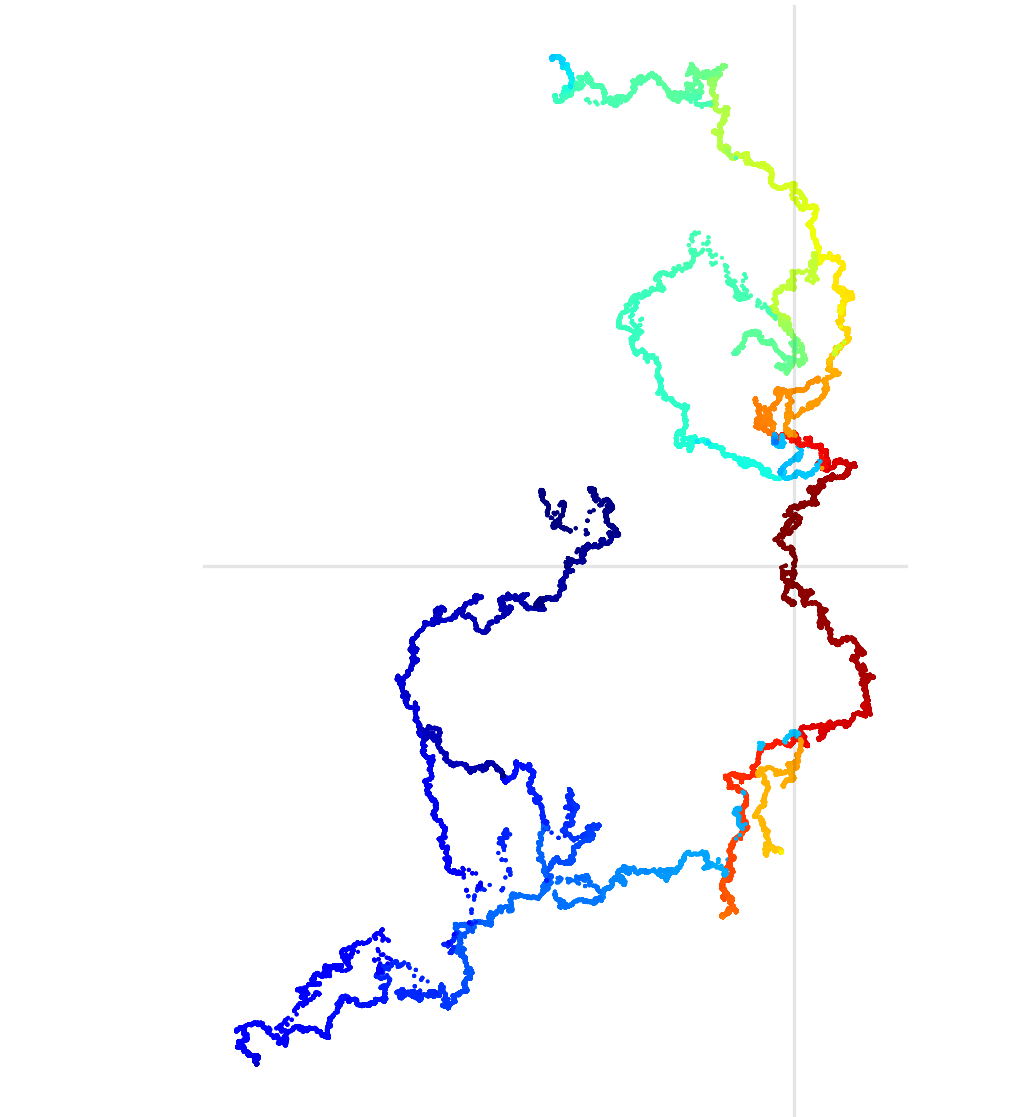}  & \includegraphics[width=0.4\textwidth]{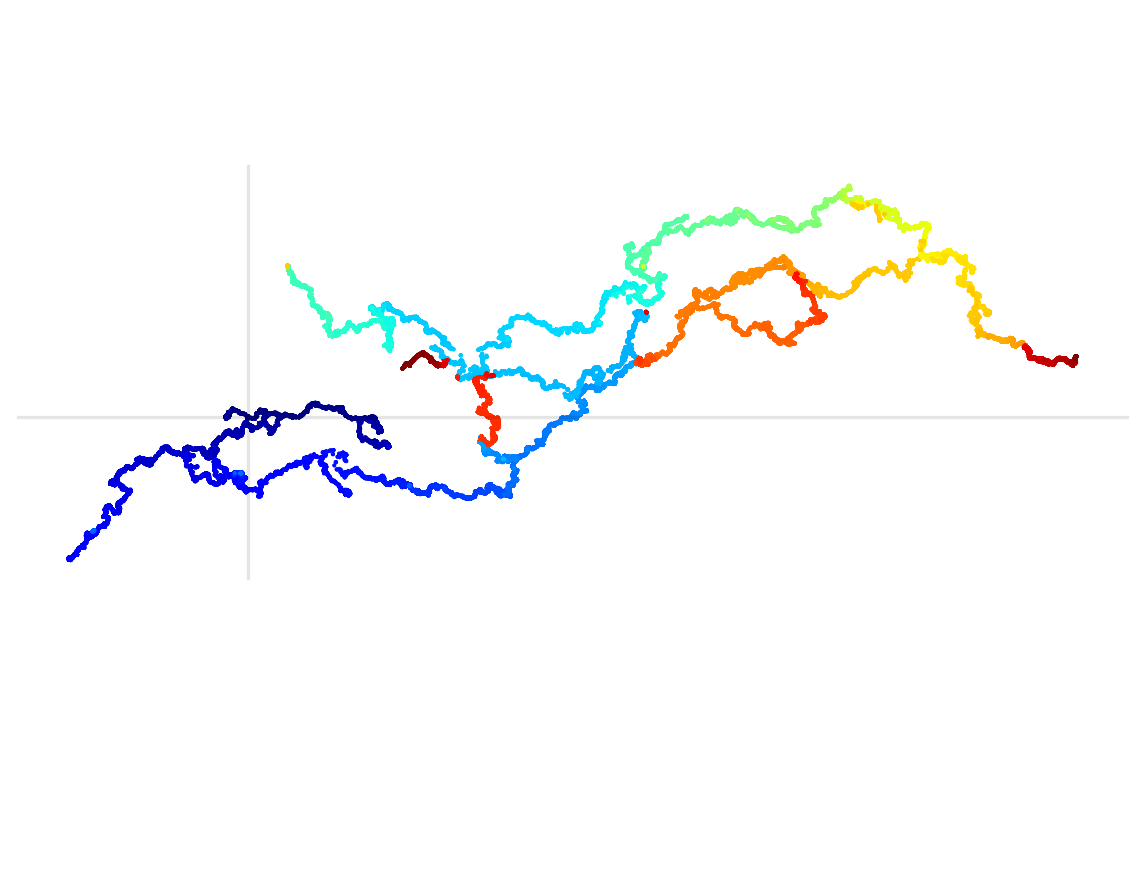} \\ \hline 
\includegraphics[width=0.11\textwidth]{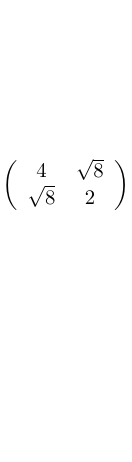} &\includegraphics[width=0.38\textwidth]{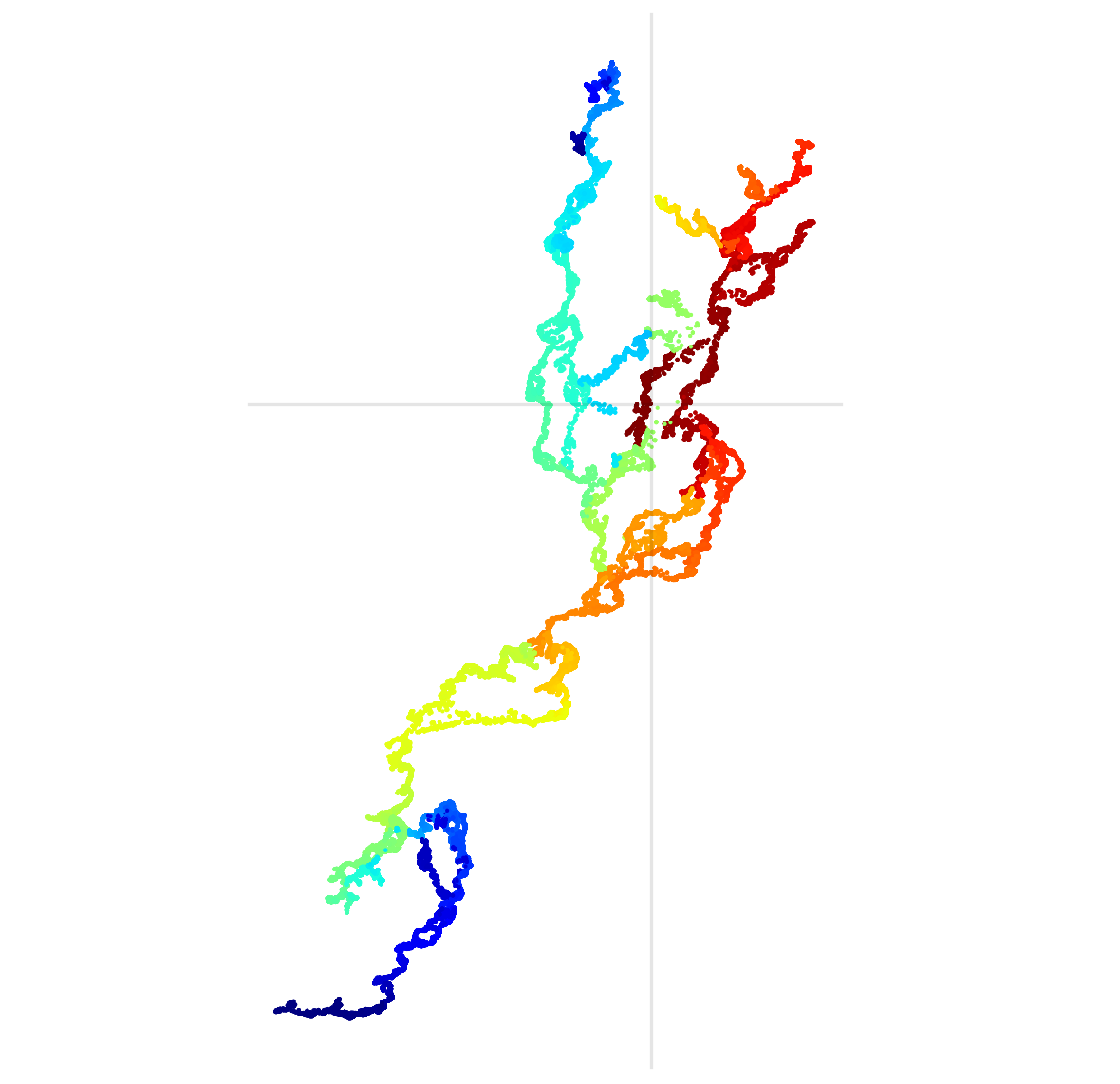}  & \includegraphics[width=0.4\textwidth]{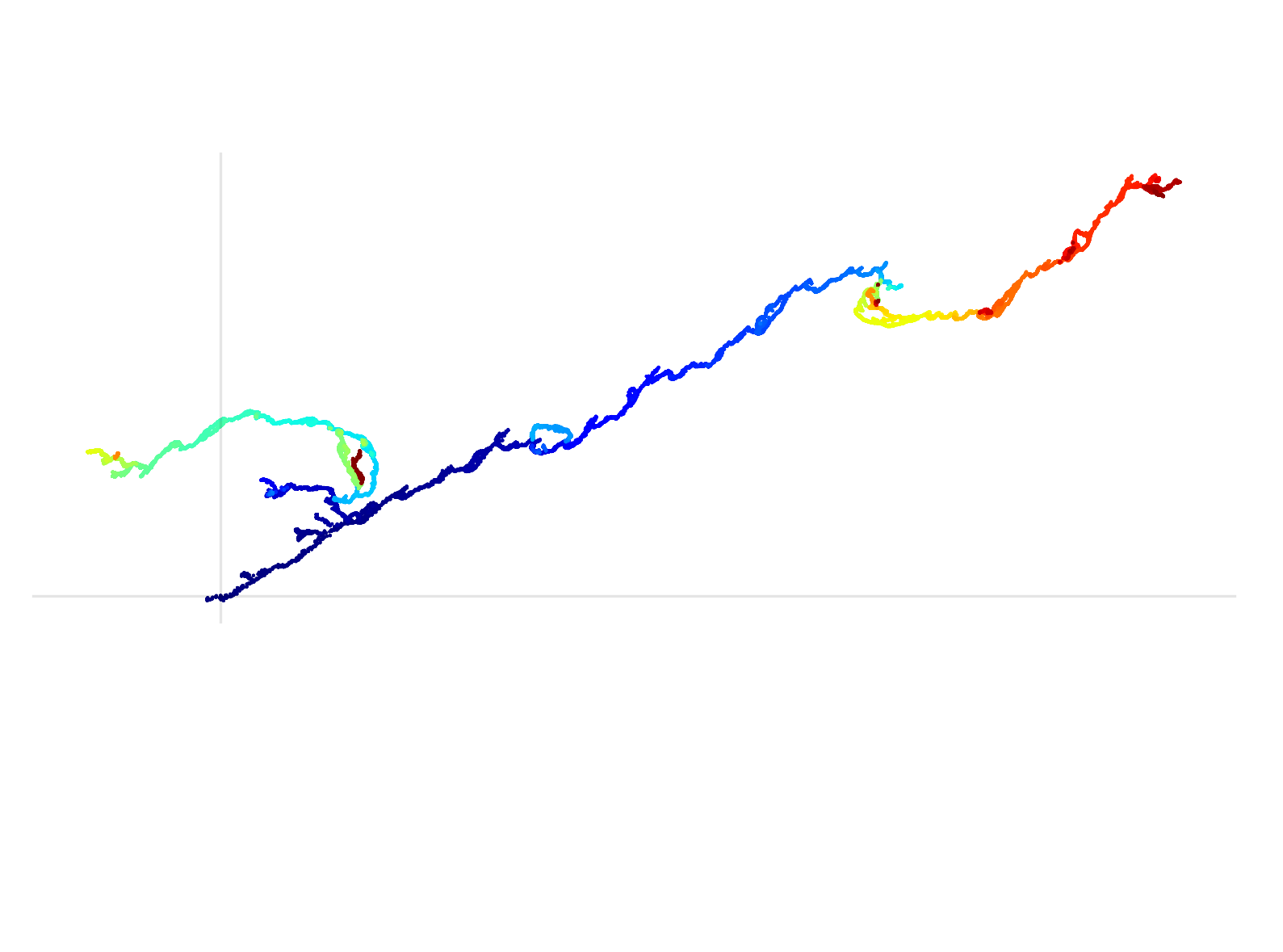} \\ \hline    
\end{tabular}

\includegraphics[width=0.35\textwidth]{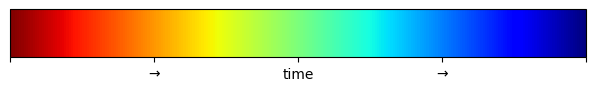}  

\end{center}

\caption{\label{fig-sim1} Simulations of left and right $\SLEG$ hulls made by Minjae Park. The first row corresponds a driving Brownian motion whose real and imaginary parts are independent (i.e. for which $c=0$). The second row corresponds to a driving function which is a complex multiple of a real Brownian motion (i.e. for which $c=\sqrt{ab}$).  Each figure shows a sample of $ 50,000$ points on the $\SLEG$ hull. The colors indicate the times at which the points are added to the left hull. The grey lines are the coordinate axes. We do not expect that there is a natural curve associated with the hull, so we do not interpolate the points (we do know, though, that the hulls are connected, see Lemma~\ref{lem-connected}). All of the hulls shown in the figure appear to be in the thin phase (for the hulls in the first row, we know this rigorously from Theorem~\ref{thm-phases}). Furthermore, the complements of the hulls are disconnected, as shown in Theorem~\ref{thm-disconnects}. See Section~\ref{sec-sim} for a discussion of how the simulations are made.  }
\end{figure}

\begin{figure}

\begin{center}

\begin{tabular}{|c|c|c|}
\hline
$\Sigma$ & \textbf{Left hull} & \textbf{Right hull}\\  \hline
\includegraphics[width=0.1\textwidth]{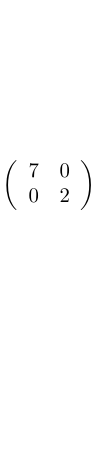} &\includegraphics[width=0.37\textwidth]{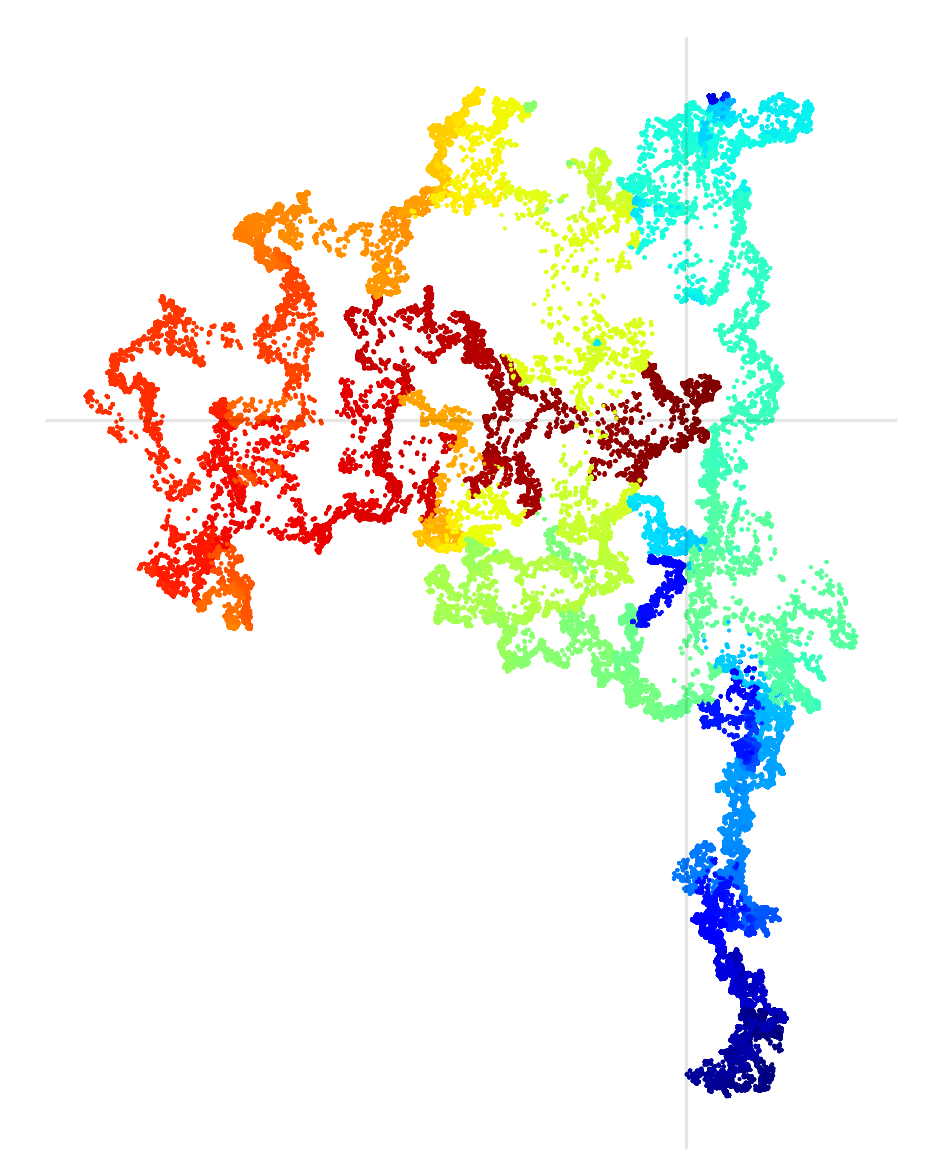}  & \includegraphics[width=0.4\textwidth]{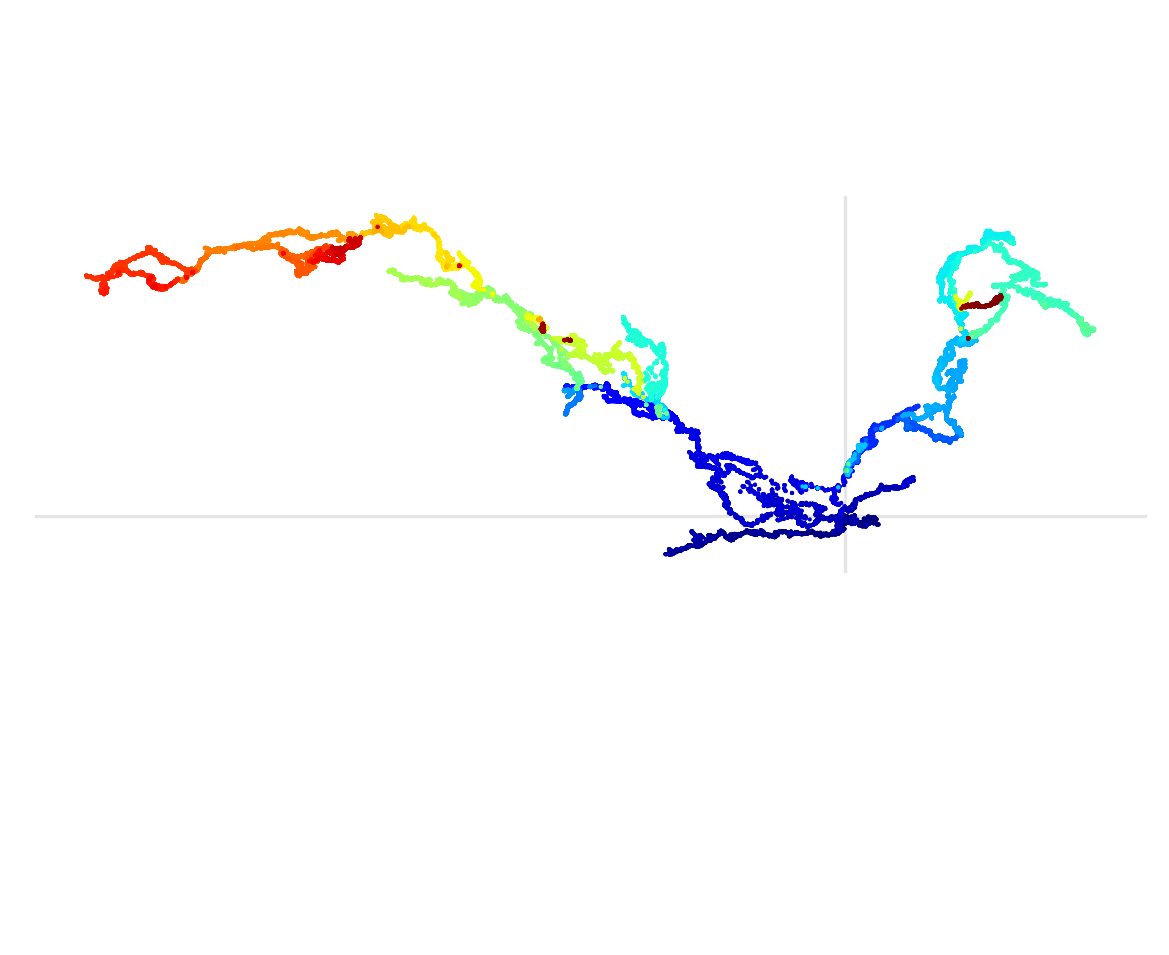} \\  \hline
\includegraphics[width=0.1\textwidth]{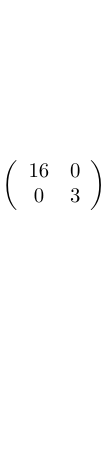} &\includegraphics[width=0.4\textwidth]{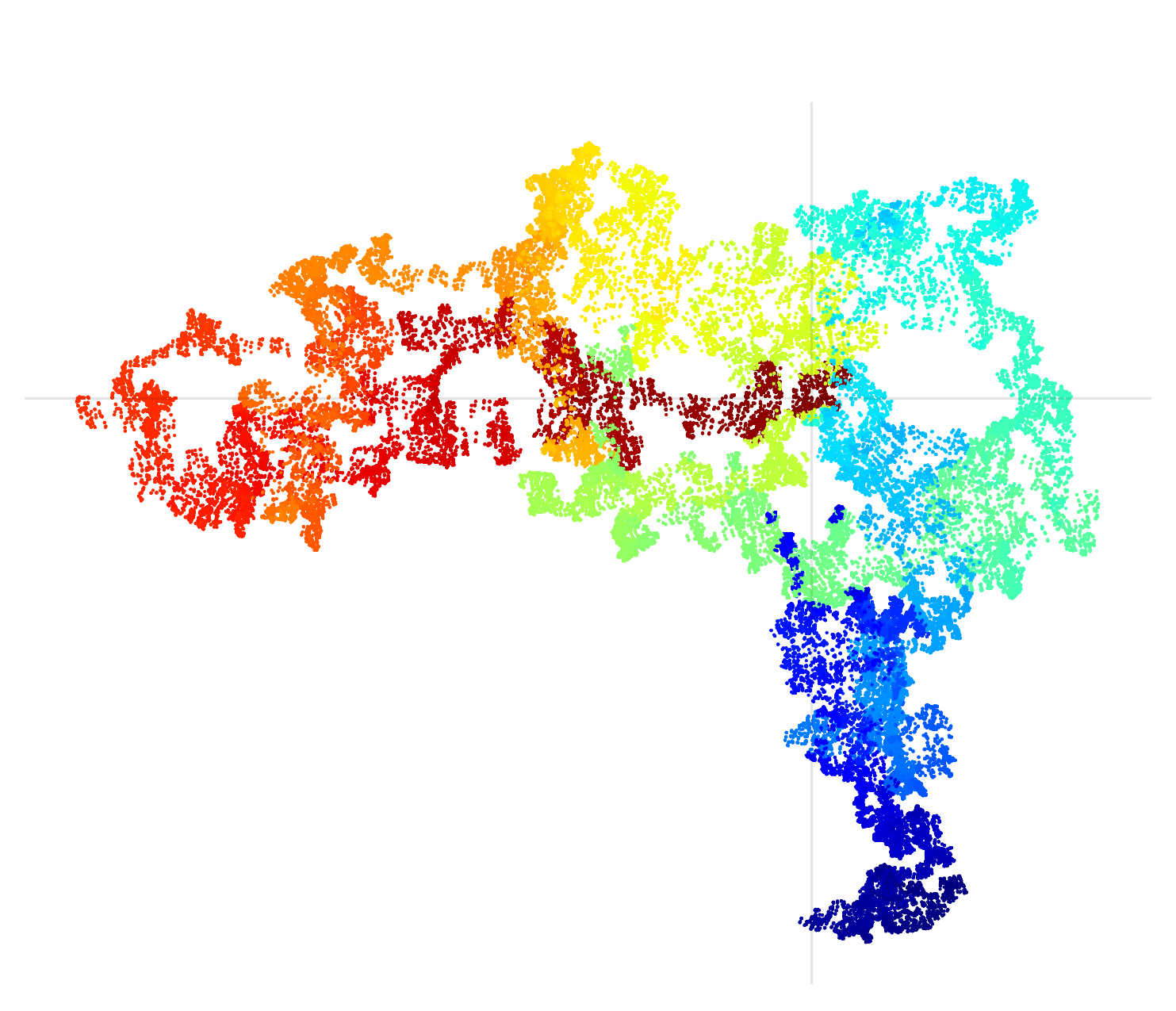}  & \includegraphics[width=0.4\textwidth]{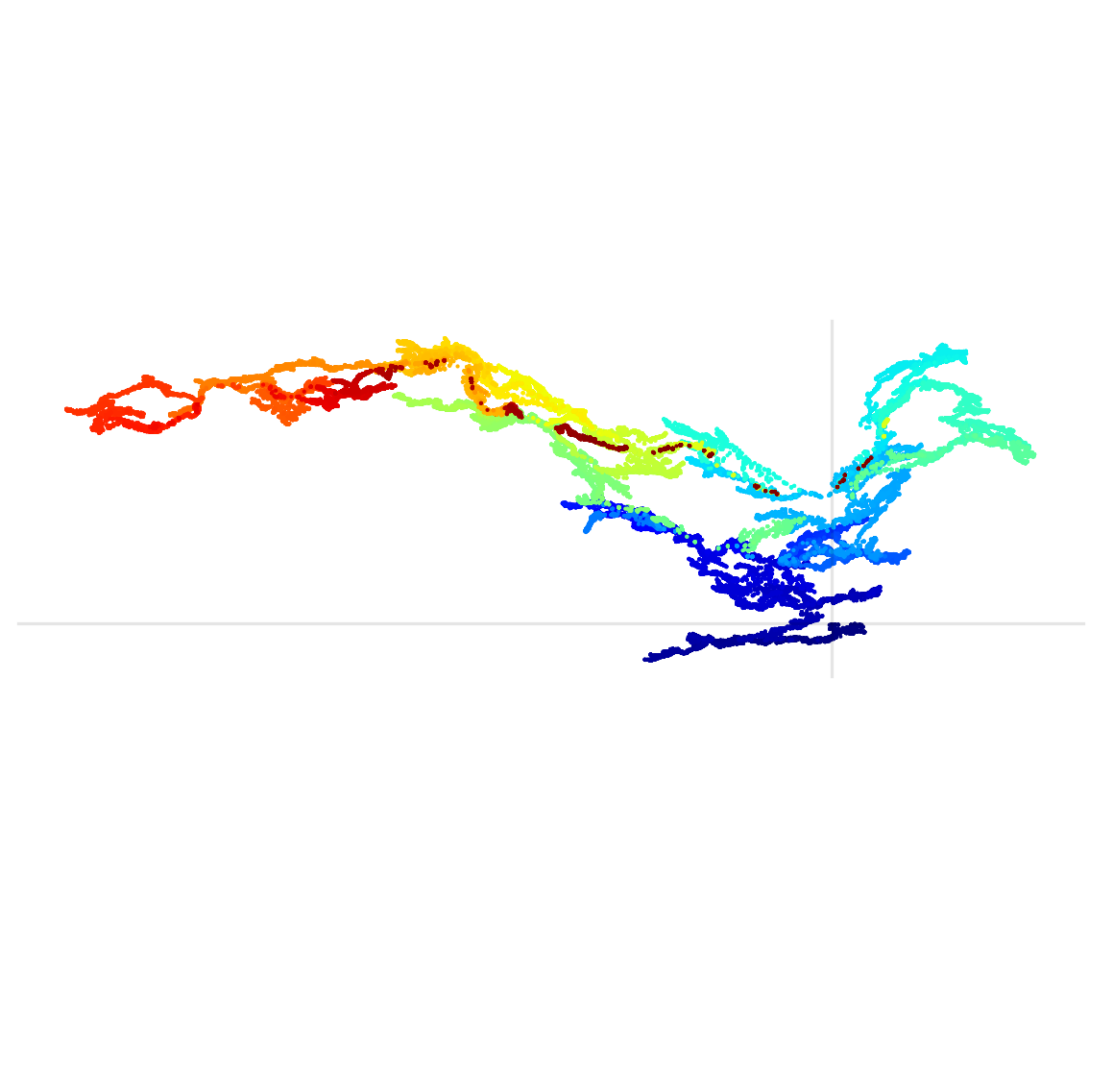} \\ \hline    
\end{tabular}

\includegraphics[width=0.35\textwidth]{figures/colorbar.png}  
 
\end{center}

\caption{\label{fig-sim2} More simulations of left and right $\SLEG$ hulls made by Minjae Park (see the caption to Figure~\ref{fig-sim1} for explanations). One should think of regions with a high density of points as being part of the hull. By Theorem~\ref{thm-phases}, the first (resp.\ second) row corresponds to a value of $\Sigma$ in the swallowing (resp.\ disconnecting) phase. The right hulls look like left hulls in the thin phase, rotated by $\pi/2$, due to Proposition~\ref{prop-duality}. Note that the left hulls still disconnect some regions from $\infty$ which are not part of the hull, which is consistent with Theorem~\ref{thm-disconnects}. }
\end{figure}

\subsection{Main results}

A classical result of Rohde and Schramm~\cite{schramm-sle} states that SLE driven by real-valued Brownian motion has three phases: the SLE$_\kappa$ hulls are simple curves for $\kappa\in (0,4]$, self-touching (but not self-crossing or space-filling) curves for $\kappa \in (4,8)$, and space-filling curves for $\kappa\geq 8$. 
Our first main result characterizes the phases of $\SLEG$ for a general covariance matrix $\Sigma$. We describe the phases in terms of the left hulls, but statements about the right hulls can be extracted via duality (Proposition~\ref{prop-duality}).

\begin{defn}  \label{defn-phases}
Let $\Sigma$ be a fixed symmetric positive semidefinite matrix, and let $(L_t)_{t \geq 0}$ be the sequence of left hulls corresponding to an $\SLEG$.  We say that $\SLEG$ is in the
\begin{itemize}
    \item \textbf{thin phase} if, for each $z \in \BB{C}$, a.s.\ $z \notin \ol{\bigcup_{t=0}^{\infty} L_t}$.  
    \item \textbf{swallowing phase} if, for each $z \in \BB{C}$, a.s.\ $z$ is \emph{swallowed}; i.e., $T_z := \inf\{t : z \in  L_t\} < \infty$ but $z$ is not contained in the closure of $\bigcup_{t<T_z} L_t$. 
    \item \textbf{hitting phase}  if,  for each $z \in \BB{C}$, a.s.\ $T_z < \infty$ and $z \in \ol{\bigcup_{t < T_z} L_t}$.
    \item \textbf{dense phase}  if,  for each $z \in \BB{C}$, a.s.\ $z \in \ol{\bigcup_{t = 0}^\infty L_t} \setminus \bigcup_{t=0}^\infty L_t$.
\end{itemize}
\end{defn}

The thin phase, swallowing phase, and hitting phase are analogous to the $\kappa \in (0,4]$, $\kappa \in (4,8)$, and $\kappa \geq 8$ phases of SLE$_\kappa$ with a real driving function, respectively. In the dense phase, a typical point $z$ is not contained in $L_t$ for any finite $t$, but the distance from $z$ to $L_t$ tends to zero as $t\rta\infty$. Note that this implies that $L_t$ is very far from being transient; i.e., for every open set $U$, the intersection $L_t\cap U$ increases for arbitrarily large times $t$. We expect that the dense phase is empty, but we do not rule this out (see also Question~\ref{ques-transient}). 

Our description of the phases of $\SLEG$ is in terms of the signs of two $\Sigma$-dependent definite integrals, which do not appear to have closed forms except in special cases. To describe these integrals, we need to introduce several functions which depend on $\Sigma$. We will explain where these functions come from just after the statement of Theorem~\ref{thm-phases}. 
Let $a,b,c$ be as in~\eqref{eqn-cov-matrix} and for $u\in\BB R$, let 
\allb \label{eqn-nu-D}
\mu(u) &:=  -\left(2-\frac{a}{2} + \frac{b}{2}\right) \sin{2u} - c \cos{2u}, \notag\\
\nu(u) &:= \left(2-\frac{a}{2} + \frac{b}{2}\right) \cos{2u} - c \sin{2u} \quad \text{and} \quad \notag\\
D(u) &:=  \frac{a}{2} \sin^2{u} + \frac{b}{2} \cos^2{u} - c \cos{u} \sin{u}.
\alle
We show in Section~\ref{sec-stationary} that there is a non-negative, $2\pi$-periodic solution $p$ to the ODE\footnote{
We expect that the non-negative, $2\pi$-periodic solution $p$ to~\eqref{eqn-p-ode} is unique, at least for a generic choice of the covariance matrix $\Sigma$. However, we do not need the uniqueness of the solution in this paper: all of our arguments work with any non-negative, $2\pi$-periodic solution which satisfies certain mild regularity conditions. This is because any such solution gives a stationary distribution for the argument process, see Lemma~\ref{lem-stationary}. So, we do not address uniqueness here. 
}
\eqb \label{eqn-p-ode}
-[\mu(u) p(u)]' + [D(u) p(u)]'' = 0 ,\quad \int_0^{2\pi} p(u) \,du = 1. 
\eqe

\begin{thm}[Phases of $\SLEG$] 
In the notation introduced just above, define the integrals
\begin{enumerate}[label=\textup{(\textbf{\Roman*})}] 
\item  \label{item-I}  $= \int_0^{2\pi} \nu(u)p(u) \, du$,
 \item \label{item-II} $= \int_0^{2\pi} (\nu(u) + 2 \cos(2u))p(u)\, du$.
\end{enumerate} 
We have the following:
\begin{itemize}
    \item If $~\ref{item-I}\geq 0,~\ref{item-II}>0$, then $\SLEG$ is in the thin phase.
    \item If $~\ref{item-I}<0,~\ref{item-II}>0$, then $\SLEG$ is in the swallowing phase 
    \item If $~\ref{item-I}<0$, $~\ref{item-II} < 0$, then $\SLEG$ is in the hitting phase.  
    \item If $~\ref{item-I}\geq 0$, $~\ref{item-II} <  0$, then $\SLEG$ is in the dense phase. 
\end{itemize}
If $c=0$ (i.e., the real and imaginary parts of the driving Brownian motion are independent),  then $\SLEG$ is in the thin phase if $a-b\leq 4$, the swallowing phase if $4 < a-b<8$, and the hitting phase if $a-b\geq 8$. 
\label{thm-phases}
\end{thm}

The theorem gives explicit formulas for the phases boundaries only when $c=0$; see Figure~\ref{fig-c0-phases} for an illustration. 
When $c\not=0$, one can approximate the phase boundaries by using numerical integration to estimate $~\ref{item-I}$ and $~\ref{item-II}$ for many different choices of $\Sigma$. 

In the case when $c\not=0$ and $~\ref{item-II}=0$, the theorem does not say whether for a fixed $z\in\BB C$, the limit $\lim_{t\rta T_z^-} \op{dist}(z,L_t)$ is a.s.\ zero or a.s.\ positive. When $c=0$, however, we are able to show that $~\ref{item-II} =0$ implies that $\lim_{t\rta T_z^-} \op{dist}(z,L_t) = 0$ (Lemma~\ref{lem-c0-bdy}). 
For every choice of $\Sigma$, regardless of the value of $~\ref{item-II}$, the sign of the integral $~\ref{item-I}$ determines whether $T_z = \infty$ a.s.\ or $T_z < \infty$ a.s.\ for each fixed $z\in\BB C$ (Proposition~\ref{prop-i}).


The proof of Theorem~\ref{thm-phases} is given in Section~\ref{sec-phases}. Let us briefly comment on the idea of the proof and the source of the functions appearing in~\eqref{eqn-nu-D} and~\eqref{eqn-p-ode}. Fix $z\in\BB C$ and for $\tz \geq 0$, let $\sigma(\tz)$ solve $\tz = \int_0^{\sigma(\tz)} \frac{1}{|f_s(z)|^2} \,ds$. One can show that a.s.\ $\sigma(\infty) = T_z$ (Remark~\ref{remark-sigma}).

Using It\^o calculus, one gets that for $\wh\theta_\tz  = \op{arg} f_{\sigma(\tz)}(z)$, 
\eqb
d\wh\theta_\tz  = \mu(\wh\theta_\tz) \,d\tz + \sqrt{2D(\wh\theta_\tz)} \,dW_\tz ,
\eqe
where $W$ is a standard linear Brownian motion (Lemma~\ref{lem-polar}). The ODE~\eqref{eqn-p-ode} for $p$ is precisely the Kolmorogov forward equation for this SDE, so $p(u)\,du$ is a periodic stationary measure for $\wh\theta_\tz$.  In other words, if $\wh\theta_0$ is sampled from the distribution $p(u) \BB 1_{[0,2\pi]}(u) \,du$ on $[0,2\pi]$, then the law of $e^{i\wh\theta_\tz}$ is stationary in $\tz$ (Lemma~\ref{lem-p}). Furthermore, the drift term in the SDE for $\log |f_{\sigma(\tz)}(z)|$ is equal to $\nu(\wh\theta_\tz)\,d\tz$. From this and an ergodic theory argument, we infer that when $\tz$ is large, $\log |f_{\sigma(\tz)}(z)| \approx \ref{item-I} \tz$. From this relation, we see that the sign of \ref{item-I} determines whether $\lim_{t \rta T_z^-} |f_t(z)|$ is equal to zero. This, in turn, determines whether $T_z < \infty$, i.e., whether $z\in \bigcup_{t>0} L_t$. Similar, but more involved, considerations show that the sign of the integral~\ref{item-II} determines whether $z\in \ol{\bigcup_{t < T_z} L_t}$.  

The explicit descriptions of the phase boundaries for $c=0$ come from the fact that the stationary density $p$ takes a particularly simple form in this case; see~\eqref{eqn-c0-p}. 

In the case of~\ref{item-I}, part of the above argument was carried out by Rohde and Schramm~\cite{rs-complex-sle} in the special case $c=\sqrt{ab}$.

\begin{figure}[ht!]
\centering
\includegraphics[width=0.5\textwidth]{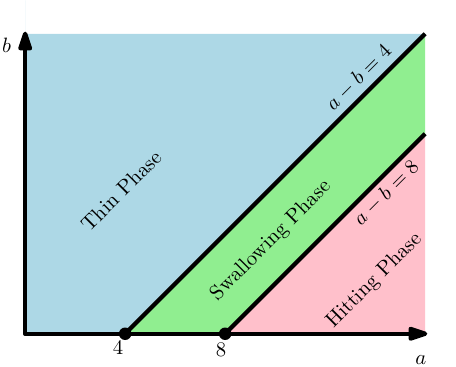}
\caption{A graph of the values of $a$ and $b$ for which $\SLEG$ with $c=0$ lies in each of the three phases. The boundary line $a-b=4$ belongs to the thin phase and the boundary line $a-b=8$ belongs to the hitting phase.} \label{fig-c0-phases}
\end{figure}

Our second main result asserts that, in contrast to the case of a real driving function, $\SLEG$ with $a,b>0$ and $a \neq b$ a.s.\ disconnects points from infinity \emph{before those points become part of the hull}. 

\begin{defn} \label{def-disconnect}
Let $S,S' \subset \BB{C}$.  We say $S $ disconnects $S'$ from infinity if $S'$ is disjoint from every unbounded connected component $\BB{C} \backslash S $. Note that we do not require that $S'\cap S = \emptyset$. 
\end{defn}

\begin{thm}[$\SLEG$ disconnects points]
\label{thm-disconnects} 
Let $(L_t)_{t\geq 0}$ be the left hulls associated to an $\SLEG$ with $a,b\neq 0$ and any $c\in [-\sqrt{ab},\sqrt{ab}]$. 
\begin{enumerate}[label=({\alph*})] 
\item \label{item-disconnect-z}
For each $z \in \BB{C}$, a.s.\ there exists $t>0$ such that $z \notin L_t$ and $L_t$ disconnects $z$ from infinity.  
\item \label{item-disconnect-t}
Almost surely, for each $t >0$ the set of points which are disconnected from $\infty$ by $L_t$ has non-empty interior. 
\end{enumerate}
\end{thm}

The result of Theorem~\ref{thm-disconnects} is in stark contrast to the case of forward or reverse Loewner evolution with real driving function, for which the complements of the hulls are always connected. Theorem~\ref{thm-disconnects} implies that even in the thin phase, $\SLEG$ hulls do not look at all like simple curves. The theorem also has non-trivial content in the hitting and swallowing phases: a.s.\ the left hulls disconnect $z$ from $\infty$ strictly \emph{before} the time $T_z = \inf\{t : z\in L_t\}$. 
The paper~\cite{lu-complex-loewner} gives some examples of deterministic complex-valued driving functions with the property that the complements of the hulls they generate are not connected. 

\begin{remark}[Generated by a curve?] \label{remark-curve}
It was shown by Rohde and Schramm~\cite{schramm-sle} that SLE$_\kappa$ for $\kappa > 0$ is generated by a curve, i.e., a.s.\ there is a continuous curve $\eta$ in the closed upper half-plane $\ol{\BB H}$ such that the SLE hull at each time $t$ is equal to the union of $\eta([0,t])$ and the set of points which it disconnects from $\infty$ in $\ol{\BB H}$. 
We expect that there is no ``reasonable" way of associating a curve with $\SLEG$ when $a,b\not=0$. The reason for this is that the proof of Theorem~\ref{thm-disconnects} shows that the left hull has to grow in two far-away places during a short interval of time. However, we do not give a precise statement to this effect, since unlike in the case of a real-valued driving function, it is not obvious how to define the concept of ``generated by a curve" for Loewner evolutions with a complex driving function. One reason for this is that the set $\BB C\setminus L_t$ is not necessarily connected, so one cannot simply extend the definition in the case of Loewner evolution with a real driving function by requiring that there is a curve such that $L_t$ is equal to the union of the curve and the set of points which it disconnects from $\infty$. 
Furthermore, $L_t$ is not necessarily strictly increasing in $t$ (i.e., one can have $L_s = L_t$ for $s<t$), so a potential ``generating curve" could spend a non-trivial interval of time tracing its past. 
An example of a driving function whose left hulls are not monotone increasing is $U_t = i\sqrt\kappa B_t$ for $\kappa  > 4$, which corresponds to reverse SLE$_\kappa$ in the phase where the right hulls are non-simple.\footnote{
To see this lack of strict monotonicity, let $U_t = i\sqrt\kappa B_t$ be the driving function and fix $t > 0$. By the duality of forward and reverse SLE~\cite{schramm-sle}, the right hull $R_t$ has the same law as the union of a (forward) SLE$_\kappa$ curve and its reflection across the real axis, rotated by ninety degrees. Hence, a.s.\ $R_t  $ contains a neighborhood of the origin.  
For $s \geq 0$ let $f_{t,s+t}$ be the cented Loewner maps driven by $s\mapsto U_{s+t} - U_t$. By the concatenation property~\eqref{eqn-conc}, a.s.\ $L_{s+t} = L_t \cup f_t^{-1}(L_{t,s+t} \setminus R_t)$. If $s > 0$ is small, then $L_{t,s+t}$ is contained in a small neighborhood of the origin, so since $R_t$ contains a small neighborhood of the origin we have $L_{s,s+t} \setminus R_t= \emptyset$ and hence $L_{s+t} = L_t$. This shows that the left hulls only grow at the atypical times $t$ when $R_t$ does not contain a neighborhood of the origin.
}
\end{remark}

\section{Context and basic properties}
\label{sec-basics}

In Section~\ref{sec-complex-loewner}, we  review some basic properties of general Loewner evolutions with complex driving functions. 
In Section~\ref{sec-cut-glue}, we record a simple description of the boundary behavior of complex Loewner chains in the case in which the left and right hulls are simple curves. (This is not directly relevant to the case of $\SLEG$, but is independently interesting and provides some useful intuition). 
In Section~\ref{sec-basics-sleg}, we   explain some consequences of the properties described in Section~\ref{sec-complex-loewner} in the special case of $\SLEG$. 
In Section~\ref{sec-polar}, we   write down SDEs satisfied by certain special observables associated with $\SLEG$, and in Section~\ref{sec-stationary} we   show that one of these SDEs---namely, the one for an appropriately time-changed version of $\op{arg} f_t(z)$---has an explicit stationary distribution. 
The computations in Sections~\ref{sec-polar} and~\ref{sec-stationary} are the key tools in the proofs of our main results.
In Section~\ref{sec-recurrence}, we prove a recurrence property of the argument process under the stationary distribution, which is needed for the proof of an ergodic-type property of this process in Section~\ref{sec-phases} (see in particular Lemma~\ref{lem-stationary-sde}).

\subsection{Basic properties of Loewner evolution with complex driving function}
\label{sec-complex-loewner}

\subsubsection{Simple operations on driving functions}

Let $U : [0,\infty)\rta \BB R$ be continuous, let $(f_t)_{t\geq 0} = (f_t^U)_{t\geq 0}$ be the centered Loewner chain driven by $U$.
Also let $(L_t)_{t\geq0 } = (L_t^U)_{t\geq 0}$ and $(R_t)_{t\geq 0} = (R_t^U)_{t\geq 0}$ be the associated left and right hulls.\footnote{We will usually exclude the superscript $U$, but it is convenient to include the driving function in the notation in the list of basic properties.}
The following properties are easily consequences of Definition~\ref{defn-loewner}; see~\cite[Section 2]{lu-complex-loewner} for proofs. Here we write $U(t) = U_t$ in order to make the notation more legible.

\begin{itemize}
\item \textbf{Scaling property.} Let $r > 0$. For each $t\geq 0$ and each $z\in\BB C$, 
\eqb \label{eqn-scale}
f_t^{r U(\cdot/r^2)} (z) = r f_{t/r^2}(z/r) ,\quad L_t^{r U(\cdot/r^2)}  = r L_{t/r^2} , \quad \text{and} \quad R_t^{r U(\cdot/r^2)}  = r R_{t/r^2} .   
\eqe 
\item \textbf{Reflection property.} Let $(x+iy)^* = x-iy$ denote the complex conjugate. For each $t\geq 0$, 
\eqb \label{eqn-reflect}
L_t^{U^*} =  ( L_t^U )^* \quad\text{and} \quad L_t^{-U} = -L_t^U .
\eqe 
The same is true with right hulls instead of left hulls.
\item \textbf{Concatenation property.} Fix $t\geq 0$ and for $s\geq 0$, let $f_{t,s+t}  = f_s^{U(\cdot+t) - U(t)}$ by the Loewner chain driven by $U(\cdot+t) - U(t)$. Similarly define the associate left/right hulls $L_{t,s+t}$ and $R_{t,s+t}$. For each $s,t\geq 0$,
\allb \label{eqn-conc}
f_{s+t}  &= f_{t,s+t}  \circ f_t |_{\BB C\setminus L_{s+t} } , 
\quad L_{s+t}  = L_t  \cup f_t^{-1}(L_{t,s+t} \setminus R_t ) \notag\\
 f_t (L_{s+t}  \setminus L_t ) &= L_{t,s+t}  , \quad \text{and} \quad
 R_{s+t}  = R_{t,s+t}  \cup f_{t,s+t} (R_t \setminus L_{t,s+t}) .
\alle
See Figure~\ref{fig-simple-conc} for an illustration. 
\item \textbf{Duality property.} For each $t\geq 0$,
\eqb \label{eqn-duality}
L_t^U = i R_t^{ i (U(t) -  U(t-\cdot))} \quad \text{and} \quad R_t^U = i L_t^{ i (U(t) -  U(t-\cdot))} .
\eqe 
\end{itemize}

\begin{figure}[ht!]
\centering
\includegraphics[width=0.7\textwidth]{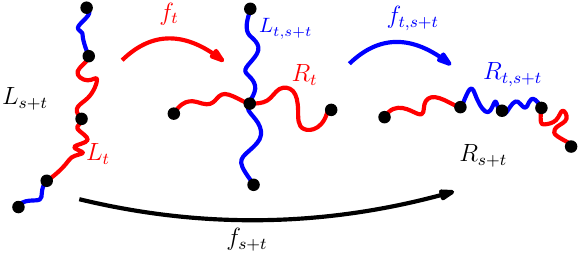}
\caption{Illustration of the concatenation property. For simplicity, we have shown the hulls as simple curves, but this is not always the case (and is typically not the case for $\SLEG$). See Figure \ref{fig-converse} for an illustration of the concatenation property in the case of hulls which are not simple curves.} \label{fig-simple-conc}
\end{figure}

\subsubsection{Relationship to Loewner evolution with a real driving function}

We will next discuss the relationship between Loewner evolution with a complex driving function and forward/reverse Loewner evolution with a real driving function. We first recall the definition of Loewner evolution in the case of a real driving function.

\begin{defn}[Chordal Loewner chains]
\label{defn-chordal-loewner}
Let $U : [0,\infty) \rta \BB{R}$ be a continuous function and let $\BB H \subset\BB C$ be the upper half-plane.  We define the forward (resp. reverse) \emph{chordal Loewner chain} with \emph{driving function} $U $ to be the collection of mappings $(g_t)_{t \geq 0}$ such that, for each $z \in \BB{H}$, $g_t(z)$ is the solution to the initial value problem  
\eqb
\label{eqn-chordal-loewner}
\dot g_t(z) =  \frac{2}{g_t(z) - U_t}   ,\quad g_0(z) =z  
\eqe
(resp.\
\eqb
\label{eqn-reverse-loewner}
\dot g_t(z) =  -\frac{2}{g_t(z) - U_t}   ,\quad g_0(z) =z  ). 
\eqe 
 We define the corresponding forward (resp. reverse) \emph{centered Loewner chain} $(f_t)_{t \geq 0}$ by $f_t(z) = g_t(z) - U_t$.
\end{defn}

The solution to~\eqref{eqn-chordal-loewner} is defined up to some time $T_z$ depending on $z$ (which we call the \emph{absorbing time}), while the solution to~\eqref{eqn-reverse-loewner} is defined for all $t$.  
For both forward and reverse chordal Loewner chains, Loewner's theorem asserts that the function $g_t$ is the \emph{unique} conformal mapping from its domain to its range satisfying $g_t(z) - z \rta 0$ as $z \rta \infty$. 

Loewner's theorem implies that we can describe each forward or reverse chordal Loewner chain as a sequence of sets in $\BB{H}$. 
In the forward case, $f_t$ maps $\BB{H} \backslash L_t \rta \BB{H}$, where  $L_t = \{z \in \BB{H} : T_z \leq t\}$.  In the reverse case,  $f_t$ maps $\BB{H}  \rta \BB{H} \backslash R_t$ for some set $R_t$.  Thus, we can identify the forward chordal Loewner chain with the sequence of sets $(L_t)_{t \geq 0}$, and the reverse chordal Loewner chain with the sequence of sets $(R_t)_{t \geq 0}$.  The sets $L_t$ and $R_t$ are both \emph{compact $\BB{H}$-hulls}, meaning that the sets are connected and their complements in $\BB{H}$ are connected.  They differ, however, in one crucial respect: the hulls $(L_t)_{t \geq 0}$ are increasing, but the hulls $(R_t)_{t \geq 0}$ are not. Intuitively, in each  infinitesimal increment of time, we add an infinitesimal piece to $L_t$ at a point along its boundary, which makes the set strictly larger.  In the case of $R_t$,  we add an infinitesimal piece to the hull $R_t$ at the origin; this does not make the set strictly larger, but instead ``pushes'' it infinitesimally further into the domain. 

In the special case when $U$ is real-valued (resp.\ purely imaginary-valued), the complex Loewner evolution reduces to forward (resp.\ reverse) chordal Loewner evolution. 
See Figure~\ref{fig-real-to-complex} for an illustration. 

\begin{figure}[ht!]
\centering
\includegraphics[width=0.7\textwidth]{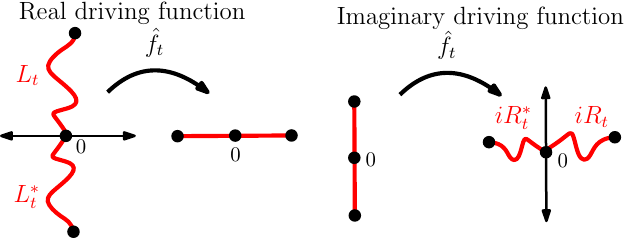}
\caption{Illustration of the left and right hulls for a centered Loewner evolution with real (left) and purely imaginary (right) driving function.} \label{fig-real-to-complex}
\end{figure}

\begin{lem} 
\label{lem-forward-reverse-ext}
As above, denote complex conjugation by a $*$. 
For a function $f$ defined on a subset of the upper half-plane $\BB{H}$, let $\wh f$ denote its analytic continuation via Schwarz reflection.\footnote{More precisely, let $\Omega \subset \BB{H}$ be the domain of $f$, and let $I$ be the set of $x \in \BB{R}$ such that $f$ extends continuously to $x$ with $f(x) \in \BB{R}$.  Then $\wh f$ is the function on $\Omega \cup \Omega^* \cup I$ defined by Schwarz reflection.}  Then the following is true.
\begin{enumerate}[label=({\alph*})]
\item
Let $(f_t)_{t \geq 0}$  be the centered forward chordal Loewner chain with driving function $U : [0,\infty)\rta \BB R$ and let $(L_t)_{t \geq 0}$ be the corresponding (left) hulls.  Then  $(\wh f_t)_{t \geq 0}$ is the centered complex Loewner chain with driving function $U $, the left hulls are $(L_t \cup   L_t^*)_{t \geq 0}$, and the right hulls are subsets of the real axis.  
\item
Let $(f_t)_{t \geq 0}$  be the centered reverse chordal Loewner chain with driving function $U : [0,\infty)\rta \BB R$ and let $( R_t )_{t \geq 0}$ be the corresponding (right) hulls.  Then $(i \wh f_t(-i\cdot))_{t \geq 0}$ is the centered complex Loewner chain with driving function $iU $, the left hulls are subsets of the imaginary axis, and the right hulls are $(i R_t \cup  i R_t^*)_{t \geq 0}$.
\end{enumerate}
\end{lem}

To prove Lemma~\ref{lem-forward-reverse-ext}, we use the following fact.

\begin{lem}
\label{lem-mult-by-i}
Let  $(g_t)_{t \geq 0}$ be the collection of mappings such that, for each $z \in \BB{C}$, $g_t(z)$ is the solution to the initial value problem~\eqref{eqn-reverse-loewner} for some continuous function  $U : [0,\infty)\rta \BB C$ (note that we allow a complex-valued driving function here). Let $f_t = g_t - U_t$.  
Then the collection of conformal maps $z\mapsto i f_t(-iz)$ for $t\geq 0$ is the centered Loewner chain with driving function $i U $.
\end{lem}
\begin{proof}
The lemma is an immediate consequence of Definition~\ref{defn-loewner}.
\end{proof}

We may define the maps $(g_t)_{t \geq 0}$ defined in~\eqref{eqn-reverse-loewner} as the \emph{reverse Loewner chain} with driving function $U$.  With this definition, we can also reverse the statement of Lemma~\ref{lem-mult-by-i}: if the maps $f_t$ are a centered Loewner chain with driving function $U$, then the maps $z\mapsto i f_t(-iz)$ are a centered reverse Loewner chain with driving function $i U $.

\begin{proof}[Proof of Lemma~\ref{lem-forward-reverse-ext}]
As noted in~\cite[Section 2]{lu-complex-loewner}, the first assertion is immediate from the reflection property under conjugation. The second assertion follows from combining the first with the result of Lemma~\ref{lem-mult-by-i}.
\end{proof}

\subsubsection{Regularity lemmas}

We now record some basic regularity properties for general Loewner evolutions with complex driving functions. 

\begin{lem} \label{lem-connected}
For any choice of continuous driving function $U : [0,\infty)\rta\BB C$, for each $t\geq 0$ the left and right hulls $L_t$ and $R_t$ are connected. 
\end{lem}
\begin{proof}
Let $\mcl L: [0,1] \rta \BB{C}$ be a loop that does not surround the origin and that does not intersect $L_t$. To prove $L_t$ is connected, it is enough to show that the region bounded by $\mcl L$ does not intersect $L_t$. The function on $[0,1] \times [0,t]$ that maps $(r,s) \mapsto f_s(\mcl L(r))$ is  continuous, and its range is contained in $\BB{C} \backslash \{0\}$.  Hence, a.s.\ $f_t \circ \mcl L$ is homotopic to $\mcl L$ in $\BB{C} \backslash \{0\}$, and therefore does not surround the origin.  This implies that the region bounded by $\mcl L$ does not intersect $L_t$.  Thus, $L_t$ is connected. The duality property~\eqref{eqn-duality} implies that $R_t$ is also connected.
\end{proof}

\begin{lem} \label{lem-endpt-cont}
Let $U : [0,\infty)\rta \BB C$ be continuous and let $(f_t)_{t\geq 0}$ be the centered Loewner chain driven by $U$.  Let $z\in\BB C$. If the absorbing time $T_z$ is finite, then 
\eqbn
\lim_{t\rta T_z^-} |f_t(z)| = 0 .
\eqen
\end{lem}
\begin{proof}
Standard existence results for solutions to ODEs imply that the solution to~\eqref{eqn-forward-loewner} exists until the smallest time $t$ for which $\liminf_{s\rta t^-} |f_s(z)|  =0$. Consequently, if $T_z < \infty$ then $\liminf_{t \rta T_z^-} |f_t(z)| = 0$. 
We need to replace the liminf by a limit. 

To this end, fix numbers $0 < \alpha < \beta$. We define an \emph{upcrossing} of $[\alpha,\beta]$ to be a pair of times $(t_1,t_2)$ with $t_1 < t_2 <T_z$ such that $|f_{t_1}(z)| = \alpha$, $|f_{t_2}(z)|= \beta$, and $|f_s(z)| \in (\alpha,\beta)$ for each $s\in (t_1,t_2)$. It suffices to show that there are at most finitely many upcrossings of $[\alpha,\beta]$ before time $T_z$. 

Since $U$ is continuous and $T_z$ is finite, $U$ is uniformly continuous on $[0,T_z]$.  Therefore, we can find a modulus of continuity $\rho : (0,\infty) \rta (0,\infty)$ with $\lim_{\delta\rta 0} \rho(\delta) = 0$ such that $|U_s -U_t| \leq \rho(|s-t|)$ whenever $s,t\in [0,T_z]$.  
By the Loewner equation~\eqref{eqn-forward-loewner}, if $(t_1,t_2)$ is an upcrossing of $[\alpha,\beta]$ with $t_1,t_2\in [0,T_z)$, 
\allb \label{eqn-endpt-cont-bd}
\beta-\alpha 
&\leq |f_{t_2}(z) - f_{t_1}(z)| \notag\\
&\leq \int_{t_1}^{t_2} \frac{2}{|f_s(z)|} \,ds + |U_{t_1} - U_{t_2}| \notag\\
&\leq 2\alpha^{-1} (t_2 - t_1)   + \rho(|t_2-t_1|) ,
\alle
where the last inequality comes from the definition of an upcrossing together with our choice of $\rho$. 

By~\eqref{eqn-endpt-cont-bd} there exists $\ep  = \ep(T_z ,\alpha,\beta) > 0$ such that $ t_2-t_1 \geq \ep$ for each upcrossing $(t_1,t_2)$ of $[\alpha,\beta]$ with $t_1,t_2 \in [0,T_z)$. 
The intervals corresponding to distinct upcrossings of $[\alpha,\beta]$ are disjoint, so the total number of upcrossings is at most $T_z/\ep$.
\end{proof}

\subsection{Cutting and gluing description}
\label{sec-cut-glue}

In this subsection we will describe the boundary behavior of the Loewner maps $f_t$ when the left and right hulls are simple curves. By Theorem~\ref{thm-disconnects} this property is \emph{not} satisfied in the case of $\SLEG$ with $a,b\not=0$. 
So, the result of this section is not directly relevant to the study of $\SLEG$. However, it is independently interesting, and it provides some useful intuition for what Loewner maps with a complex driving function look like.  

\begin{figure}[ht!]
\centering
\includegraphics[width=0.8\textwidth]{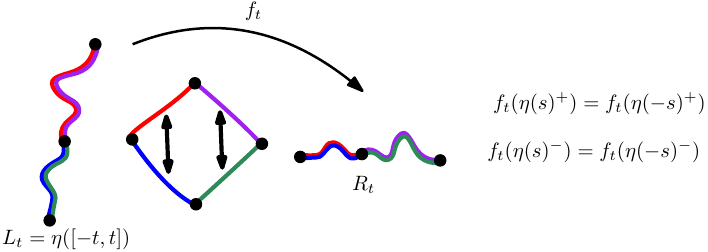}
\caption{Illustration of the statement of Lemma~\ref{lem-cut-glue}. If we assume that $L_t$ and $R_t$ are simple curves, then the mapping from prime ends of $\BB C\setminus L_t$ to prime ends of $\BB C\setminus R_t$ can be described as follows. Imagine we cut along $L_t$ with a pair of scissors to produce a polygon with four distinguished arcs, $\eta([0,t])^\pm$ and $\eta([-t,0])^\pm$ (indicated by the four colors). The map $f_t$ reverses the identification between pairs of arcs, i.e., it maps each of the pairs of prime ends $(\eta(t)^+ , \eta(-t)^+)$ and $(\eta(t)^-,\eta(-t)^-)$ to the same point. } \label{fig-cut-glue}
\end{figure}

\begin{lem} \label{lem-cut-glue}
Let $U : [0,\infty) \rta \BB C$ be a continuous function and let $(f_t)_{t\geq 0}$ be the centered Loewner chain driven by $U$.
Assume that there is a simple two-sided curve $\eta : \BB R\rta \BB C$ such that the left hull satisfies $L_t = \eta([-t,t])$ for each $t\geq 0$; and for each $t\geq 0$, the right hull is equal to the image of some simple curve.  
For each $t  > 0$ and each $z\in\eta([-t,t])$, let $z^-$ (resp.\ $z^+$) be the prime end of $\BB C\setminus \eta([-t,t])$ corresponding to the left (resp.\ right) side of $\eta([-t,t])$. Then for each $t\geq 0$ and each $s\in [0,t]$, 
\eqb \label{eqn-cut-glue} 
f_t(\eta(s)^+)= f_t(\eta(-s)^+) \quad \text{and} \quad f_t(\eta(s)^-)= f_t(\eta(-s)^-) .
\eqe
\end{lem}

See Figure~\ref{fig-cut-glue} for an illustration of the statement. 
It is shown in~\cite[Theorem 1.2]{tran-complex-loewner} that the ``simple curve" hypotheses of Lemma~\ref{lem-cut-glue} are satisfied if $U$ is $1/2$-H\"older continuous with small enough H\"older norm. 
Since $L_t$ is assumed to be a simple curve, the basic theory of prime ends (see, e.g., \cite[Theorem 2.15]{pom-book}) shows that that $f_t(z^-)$ and $f_t(z^+)$ are well-defined for each $z\in \eta([-t,t])$, so the formula~\eqref{eqn-cut-glue} makes sense.

\begin{proof}[Proof of Lemma~\ref{lem-cut-glue}]
The reader may want to consult Figure~\ref{fig-simple-conc} while reading the proof. (Note that in our setting we have $s$ in place of $t$ and $t-s$ in place of $s$). 
Fix $s \geq 0$. Let $(f_{s,s+r})_{r\geq 0}$ be the Loewner evolution driven by $(U_{r+s} - U_s)_{r\geq 0}$ and let $(L_{s,s+r})_{r\geq 0}$ and $(R_{s,s+r})_{r\geq 0}$ be the corresponding left and right hulls. By the concatenation property~\eqref{eqn-conc}, for $t\geq s$ we have $f_t = f_{s,t} \circ f_s|_{L_t}$ and 
\eqb \label{eqn-simple-conc}
\eta([-t,t]) = L_t = \eta([-s,s]) \cup f_s^{-1}(L_{s,t} \setminus R_s) .
\eqe

We first argue that 
\eqb \label{eqn-simple-tip}
f_s(\eta(s)) = f_s(\eta(-s)) = 0  .
\eqe
Since each of $\eta(s)$ and $\eta(-s)$ corresponds to a single prime end of $\BB C\setminus \eta([-s,s])$, 
\eqb \label{eqn-simple-lim}
f_s(\eta(s)) = \lim_{\BB C\setminus \eta([-s,s]) \ni z \rta \eta(s)} f_s(z) 
\eqe 
and similarly for $f_s(\eta(-s))$. 
By~\eqref{eqn-simple-conc} (with $t = s+\ep$), for each $\ep > 0$, 
\eqb \label{eqn-simple-increment}
f_s(\eta([-s-\ep ,s + \ep]) \setminus \eta([-s,s])) = L_{s,s+\ep}\setminus R_s .
\eqe 
Since the driving function $(U_{r+t} - U_t)_{r\geq 0}$ is continuous, it is easily seen from the Loewner equation that $\bigcap_{\ep>0} L_{s,s+\ep} = \{0\}$. 
Consequently, when $\ep$ is small the set~\eqref{eqn-simple-increment} is contained in a small neighborhood of the origin. 
By~\eqref{eqn-simple-lim} and since $\lim_{\ep \rta 0} \eta(s+\ep) = \eta(s)$, we get $f_s(\eta(s)) = 0$. Similarly, $f_s(\eta(-s)) = 0$. 

Now fix $t >  s$. Since $L_t = \eta([-s,s]) \cup f_s^{-1} (L_{s,t} \setminus R_s)$ (recall~\eqref{eqn-conc}) is a simple curve, the set $f_s^{-1}(L_{s,t} \setminus R_s)$ must be the disjoint union of two (possibly empty) simple curves whose closures intersect $\eta([-s,s])$ only at $\eta(-s)$ and $\eta(s)$, respectively. By applying $f_s$ and using~\eqref{eqn-simple-tip}, we get that $L_{s,t}$ is the union of two simple curves which intersect $R_s$ only at the origin, equivalently $L_{s,t}$ is a simple curve passing through the origin which intersects $R_s$ only at the origin. Since $R_s$ is a simple curve by hypothesis, the set $L_{s,t} \cup R_s$ is the union of two simple curves which intersect only at the origin. Hence the origin corresponds to four prime ends in $\BB C\setminus (L_{s,t} \cup R_s)$ which are the images under $f_s|_{\BB C\setminus L_t}$ of $\eta(s)^\pm$ and $\eta(-s)^\pm$. The prime ends which are the images of $ \eta(s)^+ $ and $ \eta(-s)^+ $ (resp.\ $ \eta(s)^- $ and $ \eta(-s)^- $) lie to the right (resp.\ left) of $L_{s,t}$. 

The origin corresponds to two prime ends in $\BB C\setminus L_{s,t}$, namely the left and right sides of $\BB C\setminus L_{s,t}$. Therefore, the two prime ends of $\BB C\setminus (L_{s,t} \cup R_s)$ lying to the right of 0 are mapped to the same point under $f_{s,t}$. By combining this with the concatenation property and the preceding paragraph, we get
\eqbn
f_t(\eta(s)^+) = f_{s,t}(f_s(\eta(s)^+) ) = f_{s,t}(0^+) = f_{s,t}(f_s(\eta(-s)^-)) = f_t(\eta(-s)^+) .
\eqen
The same argument shows that also $f_t(\eta(s)^-)= f_t(\eta(-s)^-)$.
\end{proof}

\subsection{Basic properties of $\SLEG$}
\label{sec-basics-sleg}

We now prove some basic properties of $\SLEG$ that we mentioned in Section~\ref{sec-intro}.  We first state and prove its scale invariance and Markov properties.

\begin{lem}[Scale invariance] \label{lem-scaling}
Let $r>0$ and let $(f_t)_{t \geq 0}$ be the centered Loewner chain for an $\SLEG$. Then the random collection of conformal mappings $(r^{-1} f_{r^2 t}(r\cdot))_{t\geq 0}$ has the distribution of the centered Loewner chain for an $\SLEG$.
\end{lem}

\begin{proof}
This follows immediately from the scaling property~\eqref{eqn-scale} and Brownian scaling. 
\end{proof}

\begin{lem}[Markov property]
Suppose that $(f_t)_{t \geq 0}$ is a centered $\SLEG$, and let $T$ be a stopping time w.r.t. the filtration $\mcl F_t$ associated to the driving Brownian motion $B_t$.  Then the centered Loewner chain $( f_{T,s+T} )_{s\geq 0}$ from~\eqref{eqn-conc} is independent of $\mcl F_T$ and has the distribution of a centered $\SLEG$.  
\label{lem-markov}
\end{lem}
 
\begin{proof}
This is immediate from the strong Markov property of Brownian motion. 
\end{proof}

Finally, we prove the duality property of $\SLEG$ that we stated as Proposition~\ref{prop-duality}.

\begin{proof}[Proof of Proposition~\ref{prop-duality}]
The result follows from combining Lemma~\ref{lem-mult-by-i} and the duality property~\eqref{eqn-duality} and applying the fact that, if $B_t$ is a complex Brownian motion and $T>0$, then the processes $B_t$ and $B_{T-t} - B_T$ on $[0,T]$ are equal in law.
\end{proof}

\subsection{Simulations of $\SLEG$}
\label{sec-sim}

The simulations of $\SLEG$ in Figures~\ref{fig-sim1} and~\ref{fig-sim2} are produced by the following method, which is a generalization of the method used to simulate SLE$_\kappa$ for $\kappa > 0$ in~\cite{kennedy-sle-sim}.

Let $n\in\BB N$ be large (we take $n = 25,000$). 
For $x \in \BB R$, let $\phi_t^x(z) = \sqrt{(z-x)^2  +4 t}$ be the forward centered Loewner chain with constant driving function $U_t =x$. For $y\in \BB R$, let $\phi_t^{i y}(z) = i \sqrt{( i z +  y)^2 - 4t}$ be the forward centered Loewner chain with constant driving function $U_t = i y$. We extend these functions to be defined on the complement of a line segment in $\BB C$ (rather than the complement of a line segment in a half-plane) by Schwarz reflection.
  
For $j = 1,\dots, n $, let $(X_j , Y_j)$ be i.i.d.\ two-dimensional Gaussian random variables with covariance matrix $ (n/2)^{-1/2} \Sigma$.  
For $k=1,\dots,n$, let
\eqb \label{eqn-sim-map}
\wh f_{2k / n} := \phi_{1/n}^{i Y_k} \circ \phi_{1/n}^{ X_k} \circ \phi_{1/n}^{ i Y_{k-1} } \circ \phi_{1/n}^{ X_{k-1}} \circ \dots \circ \phi_{1/n}^{ i Y_1} \circ \phi_{1/n}^{X_1}. 
\eqe
By the concatenation property~\eqref{eqn-conc}, $f_{2k / n}$ is the time $2 k / n$ Loewner map for the Loewner chain whose driving function is given by
\eqb \label{eqn-sim-driving}
\wh U_t = \sum_{\substack{j \in \BB N \\ j / n \leq t/2}} (X_j + i Y_j) .
\eqe  
This driving function converges in law as $n\rta\infty$ (with respect to the uniform topology) to $B_t^1 + i B_t^2$, where $(B^1,B^2)$ is a two-dimensional Brownian motion with covariance matrix $\Sigma$. Hence, when $n$ is large, $\wh f_{2k / n}$ should approximate the time-$(2k/n)$ map for $\SLEG$. 

Since the centered Loewner maps for $\SLEG$ satisfy $\lim_{t\rta T_z^-} |f_t(z)| = 0$ (Lemma~\ref{lem-endpt-cont}), for each $\ep  > 0$, each point of the left hull $L_2$ belongs to $f_t^{-1}(B_\ep(0))$ for some $t\leq 2$. Consequently, when $n$ is large and $\ep$ is small, the time 2 left hull for $\SLEG$ can be approximated by the set of points
\eqb \label{eqn-sim-set}
\bigcup_{w\in \{0,1,-1,i,-i\}} \left\{ \wh f_{2k/n}^{-1}( \ep w)  : k=1,\dots, n \right\} .
\eqe 
The images of the left hulls in Figures~\ref{fig-sim1} and~\ref{fig-sim2} are scatter plots of the set~\eqref{eqn-sim-set}, with points colored according to the time $2k/n$ at which they are added to the hull. Due to the duality property~\eqref{eqn-duality}, we can also approximate the right hulls corresponding to the driving function~\eqref{eqn-sim-driving} by replacing each $(X_j , Y_j)$ by $(Y_{n-j} , - X_{n-j})$, then applying the same method used to approximate the left hulls (but with the picture rotated by 90 degrees).

\subsection{Polar decomposition of $\SLEG$}
\label{sec-polar}

To describe the behavior of the centered $\SLEG$ maps $(f_t)_{t \geq 0}$, we often fix a point $z \in \BB{C}$ and analyze the evolution of $f_t(z)$ as $t$ ranges from $0$ to the absorbing time $T_z$.  To analyze the process $(f_t(z))_{t\in [0,T_z)}$, it is easier to parametrize the process not by capacity time, but by the following alternative time parametrization.

\begin{defn}[A time parametrization of $\SLEG$ run until it absorbs a fixed point]
\label{defn-sigma}
Let $(f_t)_{t \geq 0}$ be a centered $\SLEG$.
For $z \in \BB{C}$, we define the \emph{$\sigma$-time parametrization} $\sigma = \sigma_z$ of $(f_t)_{t \in (0,T_z)}$
by
\eqb \label{eqn-defn-sigma}
\tz = \int_0^{\sigma(\tz)} \frac{1}{|f_s(z)|^2} ds.
\eqe
\end{defn}

It can be shown that the function $\sigma$ maps the entire interval $(0,\infty)$ onto $(0,T_z)$; we justify this in Remark~\ref{remark-sigma} from Lemma~\ref{lem-polar}. 

The reason that the $\sigma$-time parametrization is so useful is that the polar coordinates of $f_t(z)$ in this parametrization can be expressed as SDEs involving \emph{only} the argument of $f_t(z)$ and a pair of real standard Brownian motions.  

\begin{defn} \label{def-arg}
For $z \in \BB{C}$, the \emph{argument process} associated to a centered $\SLEG$ $(f_t)_{t \geq 0}$ and the point $z$ is defined as the unique continuous process $t\mapsto \theta_t(z)$ for $t\in [0,T_z)$ such that $f_t(z) / |f_t(z)| = e^{i \theta_t(z)}$ for each $t\in [0,T_z)$ and $ \theta_0(z) \in [0,2\pi)$.  We often write $\theta_t(z)$ simply as $\theta_t$ with the value of $z$ implicit.
\end{defn}

\begin{lem}[Polar decomposition of $\SLEG$]
\label{lem-polar}
Let $(f_t)_{t \geq 0}$ be a centered $\SLEG$ and let $B_t,\wt B_t$ be the correlated standard linear Brownian motions as in~\eqref{eqn-def-ft}.
For $u\in\BB R$, also let
\eqb \label{eqn-def-mu} 
\mu(u) =  -\left(2-\frac{a}{2} + \frac{b}{2}\right) \sin{2u} - c \cos{2u} \quad\text{and} \quad 
\nu(u) = \left(2-\frac{a}{2} + \frac{b}{2}\right) \cos{2u} - c \sin{2u} ,
\eqe
as in~\eqref{eqn-nu-D}. 
Fix $z \in \BB{C}$.  With $\sigma = \sigma_z$ as in Definition~\ref{defn-sigma}, 
the mappings  
\eqbn
\wh \theta_\tz(z) := \theta_{\sigma(\tz)}(z) \quad \text{and} \quad \wh f_\tz(z) := f_{\sigma(\tz)}(z)
\eqen
satisfy the stochastic differential equations
\eqb
d\wh\theta_\tz = \mu(\wh\theta_\tz) \, d\tz - \sqrt{a}\sin{\wh\theta_\tz}  \, dW_\tz + \sqrt{b}\cos{\wh\theta_\tz}  \, d\wt W_\tz,
\label{eqn-theta-t}
\eqe
\eqb
d\log|\wh f_\tz(z)| = \nu(\wh\theta_\tz) \, d\tz + \sqrt{a}\cos{\wh\theta_\tz}\, dW_\tz + \sqrt{b}\sin{\wh\theta_\tz}\,  d\wt W_\tz,
\label{eqn-abs-zt}
\eqe
and
\eqb
d\log|\wh f_{t}'(z)| = -2 \cos{2 \wh{\theta}_\tz } \, d\tz,
\label{eqn-abs-zt'}
\eqe
where $W_\tz$ and $\wt W_\tz$ are correlated standard linear Brownian motions with $(W_\tz,\wt W_\tz)_{\tz\geq 0} \eqD (B_\tz, \wt B_\tz)_{\tz\geq 0} $ and a prime denotes the derivative with respect to $z$. 
\end{lem}

\begin{remark}
Lemma~\ref{lem-polar} implies that the function $\sigma$ of Definition~\ref{defn-sigma} is defined on the entire range $(0,\infty)$.  To see why, we first observe that, since the coefficients in the SDEs~\eqref{eqn-abs-zt} and~\eqref{eqn-abs-zt'} are bounded, both $|\wh f_\tz(z)|$ and $|\wh f_\tz'(z)|$ are a.s.\ bounded away from $0$ and $\infty$ on compact subsets of $[0,\infty)$.  
In particular, this means that if the domain of $\sigma$ were some finite interval $[0,T]$, then $\liminf_{t \rta T^-} |\wh f_t(z)|$ would be strictly positive. On the other hand, Lemma~\ref{lem-endpt-cont} implies that on the event $T_z < \infty$, we have $\lim_{t \rta T_z^-} |f_t(z)| = 0$. \label{remark-sigma}
\end{remark}

\begin{proof}[Proof of Lemma~\ref{lem-polar}]
By the SDE~\eqref{eqn-def-ft} for $f_t(z)$ and It\^{o}'s lemma for complex semimartingales~\cite[Chapter II, Theorem 35]{protter}, 
\begin{align*}
d\log f_t(z) 
&= \frac{1}{f_t(z)} \, df_t(z) - \frac{1}{2 (f_t(z))^2} \, d\la f_t(z) , f_t(z) \ra  \notag\\
&=\frac{1}{(f_t(z))^2}\left(2 - \frac{a}{2} + \frac{b}{2} - i c \right) \, dt + \frac{1}{f_t(z)}\left(\sqrt{a} \, dB_t + i \sqrt{b} \, d\wt B_t \right)\\
&= \frac{\cos{2 \theta_t} - i \sin{2 \theta_t}}{|f_t(z)|^2} \left(2 - \frac{a}{2} + \frac{b}{2} - i c \right) \, dt + \frac{\cos{ \theta_t} - i \sin{\theta_t} }{|f_t(z)|} \left(\sqrt{a} \, dB_t + i \sqrt{b} \, d\wt B_t \right) .
\end{align*}
Taking the real and imaginary parts of the latter expression yields
\eqb
d\log|  f_t(z)| = \frac{1}{|f_t(z)|^2} \nu( \theta_t) \, dt + \frac{1}{|f_t(z)|} \left( \sqrt{a}\cos{ \theta_t} \, dB_t + \sqrt{b}\sin{ \theta_t} \, d\wt B_t \right)
\label{eqn-abs-zt-op}
\eqe
and
\eqb
d \theta_t = \frac{1}{|f_t(z)|^2} \mu( \theta_t) \, dt - \frac{1}{|f_t(z)|} \left(\sqrt{a}\sin{ \theta_t} \, dB_t + \sqrt{b}\cos{ \theta_t} \, d\wt B_t \right),
\label{eqn-theta-t-op}
\eqe
respectively.
Hence, with
\[
W_\tz = \int_0^{\sigma(\tz)} \frac{1}{|f_s(z)| } dB_s \qquad \text{and} \qquad \wt W_\tz = \int_0^{\sigma(\tz)} \frac{1}{|f_s(z)| } d\wt B_s,
\]
the time-reparametrized processes $\log|\wh f_\tz(z)| $ and $\wh \theta_\tz(z)$  satisfy the SDEs~\eqref{eqn-abs-zt} and~\eqref{eqn-theta-t}, respectively. 
We have $(W_\tz , \wt W_\tz)_{\tz\geq 0} \eqD (B_\tz , \wt B_\tz)_{\tz\geq 0}$ since both processes are continuous martingales with the same quadratic variations and quadratic cross-variations. 
Finally, to obtain~\eqref{eqn-abs-zt'}, we differentiate the SDE for $f_t(z)$ w.r.t.\ $z$ and divide by $f_t'(z)$ to obtain $d\log f_t'(z) = (-2/(f_t(z))^2) \,  dt$, and then we time-change by $\sigma$ and take real parts. 
\end{proof}

\subsection{The stationary argument process}
\label{sec-stationary}
 
To analyze the evolution of the processes in Lemma~\ref{lem-polar}, we consider a \emph{random} choice $\zz$ of $z$ (sampled independently from the driving process $U$) for which the time-changed process $\wh f_\tz(\zz) / |\wh f_\tz(\zz)| = e^{i  \wh \theta_\tz(\zz)}$ is stationary in $\tz$. In this subsection, we will construct the distribution of $\zz$. 

\begin{lem}
\label{lem-p}
Let $\mu$ be as in~\eqref{eqn-def-mu} and let
\eqb
\label{eqn-def-D}
D(u) =  \frac{a}{2} \sin^2{u} + \frac{b}{2} \cos^2{u} - c \cos{u} \sin{u} 
\eqe
be as in~\eqref{eqn-nu-D}. There exists a non-negative, $\pi$-periodic solution to the differential equation
\eqb
-[\mu(u) p(u)]' +  \left[ D(u) p(u) \right]'' = 0, \quad \int_0^{2\pi} p(u) \,du = 1  
\label{eqn-p-ODE}
\eqe
which is strictly positive except possibly on a discrete set. 
Furthermore, the solution $p$ can be taken to have the following property: let $\zz$ be a random point on the unit circle, sampled independently from the driving Brownian motion, whose angle has distribution $p(u) \BB 1_{[0,2\pi]}(u) \,du$. Also let $\sigma = \sigma_\zz$ be the time change as in Definition~\ref{defn-sigma}. Then $e^{i \theta_{\sigma(\tz)}(\zz)} = f_{\sigma(\tz)}(\zz) / |f_{\sigma(\tz)}(\zz)|$  is stationary in $\tz$.
\end{lem}

The proof of Lemma \ref{lem-p} will be given at the end of this subsection. Note that it is natural to consider a $\pi$-periodic solution to~\eqref{eqn-p-ODE}, rather than a $2\pi$-periodic solution, since the functions $\mu$ and $D$ are $\pi$-periodic. This is related to the fact that the driving function $U$ is a complex Brownian motion, so $-U$ has the same distribution as $U$.

Lemma~\ref{lem-p} is a powerful tool that we use throughout this work to study $\SLEG$. By Lemma~\ref{lem-polar}, we can express all three processes $\wh \theta_\tz(\zz)$, $\log|\wh f_{\tz}(\zz)|$, and $\log|\wh f_{\tz}'(\zz)|$ as solutions to stochastic differential equations whose coefficients are $2\pi$-periodic functions of $\wh \theta_\tz(\zz)$. This means that we can use the stationarity of $e^{i \wh\theta_\tz(\zz)}$ to analyze the asymptotic behavior of these three processes.  If we can apply our understanding of this behavior to show that some property of $(\wh f_\tz(z))_{\tz \geq 0}$ a.s.\ holds with $z \sim \zz$, then we will often be able to deduce that the same property a.s.\ holds for each fixed point $z$ on the unit circle.  We can then extend the result to general $z \in \BB{C}$ using scale invariance (Lemma~\ref{lem-scaling}).  

The ODE~\eqref{eqn-p-ODE} comes from the Kolmogorov forward equation for the SDE satisfied by $\wh\theta_\tz$, as we now explain. 

\begin{lem} \label{lem-stationary}
Let $p:\BB{R} \rta [0,\infty)$ be a continuous, non-negative, $ \pi$-periodic function with $\int_0^{2\pi} p(u) \, du = 1$. 
Assume that there exists $x\in\BB R$ such that on $\BB R\setminus (x+\pi\BB Z)$, the function $p$ is twice continuously differentiable and satisfies the ODE~\eqref{eqn-p-ODE}; and the function $u \mapsto D(u) p(u)$ (where $D$ is as in~\eqref{eqn-def-D}) is continuously differentiable on all of $\BB R$. 
Let $\zz$ be a random point on the unit circle, sampled independently from the driving Brownian motion, and whose angle has distribution $p(u) \BB 1_{[0,2\pi]}(u) \,du$.  Then the law of $e^{i \theta_{\sigma(\tz)}(\zz)} = f_{\sigma(\tz)}(\zz) / |f_{\sigma(\tz)}(\zz)|$  is stationary in $\tz$.
\end{lem}

In the case when the coefficients of the covariance matrix $\Sigma$ satisfy $|c| < \sqrt{ab}$, we will construct a solution to~\eqref{eqn-p-ODE} which is twice continuously differentiable (in fact, smooth) on all of $\BB R$, so the condition involving $x$ in Lemma~\ref{lem-stationary} will hold vacuously. 
In the case when $c =\sqrt{ab}$, the solution $p$ construct is smooth only away from the points $x + \pi \BB{Z}$ for $x = \arctan(\sqrt{b/a})$, but satisfies the regularity conditions  of Lemma~\ref{lem-stationary} for this value of $x$. 

\begin{proof}[Proof of Lemma~\ref{lem-stationary}]
Let
\eqbn
M_\tz := - \sqrt{a} \int_0^\tz \sin{\wh\theta_\sz}  \, dW_\sz + \sqrt{b} \int_0^\tz \cos{\wh\theta_\sz}  \, d\wt W_\sz 
\eqen
be the martingale part from the SDE~\eqref{eqn-theta-t} for $\wh\theta_\tz$. 
Then $M_\tz$ is a continuous martingale with quadratic variation $d\la M , M \ra_\tz = 2 D(\wh\theta_\tz) \,d\tz$.
Therefore, $dM_\tz = \sqrt{2 D(\wh\theta_\tz)} \,dW_\tz'$, where $W_\tz'$ is a standard linear Brownian motion. 
Plugging this into~\eqref{eqn-theta-t} gives
\eqb \label{eqn-theta-t-bm}
d\wh\theta_\tz = \mu(\wh\theta_\tz) \, d\tz  + \sqrt{2 D(\wh\theta_\tz)} \,dW_\tz' .
\eqe
Just below, we will use a version of the Kolmogorov forward (Fokker-Planck) equation for the density of the stationary distribution of the SDE~\eqref{eqn-theta-t-bm}, namely Lemma~\ref{lem-kolmogorov}, to get that the law of $e^{2i \theta_{\sigma(\tz)}(\zz)}$ is stationary. We claim that this implies that the law of $e^{ i \theta_{\sigma(\tz)}(\zz)}$ is also stationary. Indeed, since $\arg \zz$ is sampled from $p(u) \BB 1_{[0,2\pi]}(u) \,du$ and $p$ is $\pi$-periodic, we get that $-\zz \eqD \zz$. By combining this with the reflection symmetry of Brownian motion and the behavior of the Loewner evolution under multiplying the driving function by $-1$~\eqref{eqn-reflect}, we get that $(f_t(\zz))_{t\geq 0} \eqD (f_t(-\zz))_{t\geq 0}\eqD (-f_t( \zz))_{t\geq 0}$. Hence, $-e^{i \theta_{\sigma(\tz)}(\zz)} \eqD e^{i \theta_{\sigma(\tz)}(\zz)} $, so the stationarity of $e^{2i \theta_{\sigma(\tz)}(\zz)}$ implies the stationarity of $e^{ i \theta_{\sigma(\tz)}(\zz)}$. 

Let us now use Lemma~\ref{lem-kolmogorov} to check that the law of $e^{2i \theta_{\sigma(\tz)}(\zz)}$ is stationary. 
In the case when $|c| \not= \sqrt{a b}$, the function $D$ of~\eqref{eqn-def-D} is always strictly positive, so the coefficients $\mu$ and $\sqrt{2D}$ of the SDE~\eqref{eqn-theta-t-bm} are smooth and $\pi$-periodic and the diffusion coefficient $\sqrt{2 D }$ never vanishes. Hence, the stationarity of the law of $ e^{2i \theta_{\sigma(\tz)}(\zz)} $ is immediate from Lemma~\ref{lem-kolmogorov} in this case. In the case when $c = \pm \sqrt{a b}$, we have $2 D(u) = (\sqrt{a } \sin (u) \mp \sqrt{b} \cos(u))^2 $, so the SDE~\eqref{eqn-theta-t-bm} can equivalently be written in the form
\eqb \label{eqn-theta-t-bm'}
d\wh\theta_\tz = \mu(\wh\theta_\tz) \, d\tz  +   \beta(\wh\theta_\tz) \,dW_\tz'' , \quad \text{where} \quad \beta(u) := \sqrt{a } \sin (u) \mp \sqrt{b} \cos(u) 
\eqe
for a standard linear Brownian motion $W''$. 
The coefficients of~\eqref{eqn-theta-t-bm'} are smooth. Moreover, the diffusion coefficient $\beta$ vanishes only on $\arctan(\pm \sqrt{b/a}) + \pi \BB Z$, and for each $u$ in this set we have $\mu(u) \beta'(u) = \mp 4 \sqrt{\frac{ab}{a+b}} \not=0$. Hence the SDE~\eqref{eqn-theta-t-bm'} satisfies the hypotheses of Lemma~\ref{lem-kolmogorov}, and we obtain the stationarity of the law of $ e^{2i \theta_{\sigma(\tz)}(\zz)} $ in the case when $c = \pm\sqrt{ab}$.
\end{proof}

\begin{proof}[Proof of Lemma~\ref{lem-p}]
We need to construct a non-negative, $\pi$-periodic solution to the differential equation~\eqref{eqn-p-ODE} which satisfies the hypotheses of Lemma~\ref{lem-stationary}. 
First, we claim that it suffices to prove the existence of $p$ when $c \geq 0$. To justify this claim, let $\wt \mu$ and $\wt D$ be the functions obtained by switching the sign of $c$ in~\eqref{eqn-def-mu} and~\eqref{eqn-def-D}.  Then $\wt \mu(u) = -\mu(-u)$ and $\wt D(u) = D(-u)$.  Thus, if $p$ satisfies the hypotheses of Lemma~\ref{lem-stationary}, then $\wt p(u) = p(-u)$ satisfies the hypotheses of Lemma~\ref{lem-stationary} with $\mu$ and $D$ replaced by $\wt \mu$ and $\wt D$. Thus, we may assume $c \geq 0$. We consider separately the cases  $c < \sqrt{ab}$ and $c = \sqrt{ab}$; the key difference between these cases is that, when $c = \sqrt{ab}$, the function $D$ is not always strictly positive.
\medskip

\noindent \emph{Case 1: $c < \sqrt{ab}$.}
Since $D(u) > 0$ for all $u\in\BB R$, the functions $p_0,p_1:\BB{R} \rta [0,\infty)$ given by 
\begin{equation}
p_0(u) =  \exp\left(\int_0^u \frac{\mu(s) - D'(s)}{D(s)} \, ds \right) 
\label{eqn-defn-p0}
\end{equation}
and
\begin{equation}
p_1(u) =  \frac{1}{D(u)} \exp\left(\int_0^u \frac{\mu(s)}{D(s)}\, ds \right)
\int_0^u \exp\left(-\int_0^s \frac{\mu(v)}{D(v)} dv \right)\, ds 
\label{eqn-defn-p1}
\end{equation}
are smooth and satisfy the first-order ODEs 
\begin{equation}
-\mu(u) p_j(u) +  \left[ D(u) p_j(u) \right]' = j, \quad \forall j \in \{0,1\} .
\label{eqn-p0-p1-ODE}
\end{equation}

By~\eqref{eqn-p0-p1-ODE}, the function $p_0 + r p_1$ satisfies~\eqref{eqn-p-ODE} for every $r \in \BB{R}$.  Let 
\[
p = k(p_0 + r_* p_1) , 
\]
where
\[
r_* = \frac{p_0(0) - p_0(\pi)}{p_1(\pi)}.
\]
and $k > 0$ is a constant which will be chosen later in order to make $\int_0^{2\pi} p(u) du = 1$.

By the choice of $r_*$ and since $p_1(0) = 0$, we have $p(0) = p(\pi)$.  Now, since $\mu$ and $D$ are $\pi$-periodic, the functions $p$ and $p(\cdot + \pi)$ satisfy the same first-order ODE:
\eqbn
-\mu(u) p(u) +  \left[ D(u) p(u) \right]' = 
-\mu(u) p(u + \pi) +  \left[ D(u) p(u + \pi) \right]' = k r_*.
\eqen
Therefore, the fact that $p(0) = p(\pi)$, together with the uniqueness part of the Picard-Lindel\"of theorem, implies that $p$ is $\pi$-periodic.  

We next check that $p$ is nonnegative. Since $p_0$ and $p_1$ are clearly non-negative, to prove this, it is enough to show that $r_* \geq 0$.  Since $p_1(\pi) > 0$ and $p_0(0) = 1$, we just have to show that $p_0(\pi) < 1$---or, equivalently, that 
\[ 
\int_0^\pi \frac{\mu(u) - D'(u)}{D(u)} \,du \leq 0.
\]  
We have
\[
\int_0^\pi \frac{D'(u)}{D(u)} \, du = \log{D(\pi)} - \log{D(0)} = 0.
\]
We now verify that $\int_0^\pi \frac{\mu(u)}{D(u)} \,du \leq 0$ by 
a direct computation.  The antiderivative of $\mu(u)/D(u)$ is given by $\left((a-b)^2+4 c^2\right)^{-1}$ times the expression
\begin{align*} 
&\left( (a-b)^2 -4 (a-b) +4 c^2\right) \log ((b-a) \cos (2 u)-2 c \sin (2 u)+a+b) \\
&\qquad +16 c u - \frac{8 c (a+b)}{\sqrt{ab-c^2}} \arctan\left(\frac{a \tan
   (u)-c}{\sqrt{ab-c^2}}\right).
   \end{align*}
Evaluating this antiderivative at the endpoints of the interval $[0,\pi]$, we deduce that
\[
\int_0^\pi \frac{\mu(u)}{D(u)}  \,du  = 
8\pi c \frac{2 \sqrt{a b-c^2}- (a+b)}{\sqrt{a b-c^2} \left((a-b)^2+4 c^2\right)} \leq 0.
\] 

Next we check that $p$ is non-zero except possibly on a discrete set. By~\eqref{eqn-defn-p0}, this is clearly the case if $r_* = 0$. If $r_*\not=0$, then by~\eqref{eqn-p0-p1-ODE}, we have that $ -\mu(u) p (u) +  \left[ D(u) p (u) \right]'$ is a non-zero constant, so again $p$ is non-zero except possibly on a discrete set. 

Finally, we note that the fact that $p = k(p_0 + r_* p_1)$ is strictly positive except on a discrete set implies that $\int_0^{2\pi} (p_0 + r_* p_1)(u) \,du > 0$, so we can choose $k > 0$ so that $\int_0^{2\pi} p(u) \,du = 1$. 
\medskip

\noindent \emph{Case 2: $c = \sqrt{ab}$.}
Choose $x \in [0,\pi)$ such that $\tan(x) = \sqrt{b/a}$, so that $D^{-1}(0) = x + \pi \BB Z$. 
Let 
\eqb \label{eqn-critical-p}
p (u) = \frac{\int_x^u \exp\left(-\int_{x+\pi/2}^s \frac{\mu(v)}{D(v)} dv \right) ds}{D(u) \exp\left(-\int_{x+\pi/2}^u \frac{\mu(s)}{D(s)} ds \right)} , 
\eqe
i.e., $p$ is defined analogously to~\eqref{eqn-defn-p1} but with different bounds of integration. By~\eqref{eqn-def-mu}, $\mu(x) = \mu(x+\pi) = -4\sqrt{ab}/(a+b) < 0$, from which we infer that $\int_{x+\pi/2}^u \frac{\mu(s)}{D(s)} ds$ tends to $\infty$ as $u\rta x$ and to $-\infty$ as $u\rta x+\pi$. In particular, the outer integral in the numerator in~\eqref{eqn-critical-p} is finite for $u\in (x,x+\pi)$. 
Hence, $p$ is well-defined, positive, and smooth on the domain $(x,x+\pi)$ and on this domain satisfies 
\eqb \label{eqn-critical-p-ode}
-\mu(u) p (u) +  \left[ D(u) p(u) \right]' = 1  .
\eqe
By applying L'H\^opital's rule at the values $u=x$ and $u=x+\pi$, we see that 
\eqb \label{eqn-lhop-down}
\lim_{u \downarrow x} p (u) = \frac{1}{D'(x) - \mu(x)} = \frac{-1}{\mu(x)}  = \frac{a+b}{4\sqrt{ab}}  
\eqe
and  
\eqb \label{eqn-lhop-up}
\lim_{u \uparrow x+\pi} p (u) = \frac{1}{D'(x+\pi) - \mu(x+\pi)} =  \frac{-1}{\mu(x+\pi)}  = \frac{a+b}{4\sqrt{ab}}   .
\eqe  
We deduce that $p $ extends to a continuous periodic function on all of $\BB{R}$ that is smooth away from the points $x+n\pi$ for $n \in \BB{Z}$.  Since $\mu,D$ are $\pi$-periodic, this extended function $p :\BB{R} \rta [0,\infty)$ satisfies~\eqref{eqn-critical-p-ode} and therefore~\eqref{eqn-p-ODE} away from the points $x+n\pi$. By~\eqref{eqn-p0-p1-ODE}, we have $[D(u)p (u)]' = 1 + \mu(u) p (u)$ on $\BB R\setminus (x+\pi \BB Z)$. By~\eqref{eqn-lhop-down} and~\eqref{eqn-lhop-up}, $[D(u)p (u)]'$ tends to 0 as $u\downarrow x$ or as $u\uparrow x + \pi$. Since $D(u) p (u)$ is $\pi$-periodic, it follows that $D(u) p (u)$ is continuously differentiable with vanishing derivative on $x + \pi\BB Z$. 
\end{proof}

\begin{remark} \label{remark-p0} 
In the case $c=0$, the periodic solution $p$ in Lemma~\ref{lem-p} is a constant times $p_0$, since for $c=0$ the integrand inside the exponential in~\eqref{eqn-defn-p0} is odd and $\pi$-periodic.   
\end{remark}

\subsection{Recurrence of the stationary argument process}
\label{sec-recurrence}

In this subsection we prove that the stationary argument process of Lemma~\ref{lem-p} is recurrent in the following sense. 
This will be needed to check the hypotheses of Lemma~\ref{lem-stationary-sde} below. 

\begin{lem} \label{lem-recurrence}
Let $p : \BB R \to [0,\infty)$ be as in Lemma~\ref{lem-p}.  
Let $\zz \in \bdy\BB D$ be a random point on the unit circle, sampled independently from the driving Brownian motion, whose angle has the distribution $p(u) \BB 1_{[0,2\pi]}(u) \,du$.
Let $\wh\theta_{\tz} (\zz) = \theta_{\sigma_{\zz}(\tz)}(\zz)$ be the stationary argument process as in Definitions~\ref{defn-sigma} and~\ref{def-arg}.
\begin{itemize}
\item If $|c| < \sqrt{a b}$, then for each $w\in \bdy\BB D$, a.s.\ there exist arbitrarily large times $\tz$ for which $e^{2 i \wh\theta_{\tz}(\zz)} = w$.
\item If $|c| = \sqrt{a b}$, then a.s.\ there exist arbitrarily large times $\tz$ for which $e^{2 i \wh\theta_{\tz}(\zz)} = e^{2 i x}$, where $x\in [0,\pi)$ is such that $\tan(x) = \sqrt{b/a}$. 
\end{itemize}
\end{lem} 
\begin{proof}
Assume without loss of generality that $c\geq 0$. If $c = \sqrt{a b}$, let $x  \in [0,\pi)$ such that $\tan(x) = \sqrt{b/a}$, so that $D^{-1}(0) = x + \pi \BB Z$. If $c < \sqrt{ab}$, instead let $x \in [0,\pi)$ be arbitrary. To prove both assertions of the lemma, it suffices to show that a.s.\ there exist arbitrarily large times $\tz$ for which $e^{2 i \wh\theta_{\tz}(\zz)} = e^{2 i x}$. 
  
Let $\BB P^y$ denote the law of the process $\wh\theta_{\tz}$ (i.e., the solution to the SDE~\eqref{eqn-theta-t}) started from $\wh\theta_0 = y$. We claim that for each $\ep  >0$, there exists $T , \delta>0$ such that 
\eqb \label{eqn-angle-hit-unif}
\BB P^y\left[  \exists \tz \in [0,T] \: \text{such that} \: \wh\theta_{\tz} = x \right]  \geq \delta \quad\forall y \in [x,x+\pi-\ep] .
\eqe  
We will prove~\eqref{eqn-angle-hit-unif} at the end of the proof using stochastic calculus. First we explain how~\eqref{eqn-angle-hit-unif} implies the lemma statement.

For $\tz \geq 0$, let $[\wh\theta_{\tz}(\zz)]$ be the unique number in $[x,x+\pi)$ for which
\eqbn
\wh\theta_{\tz}(\zz)  - [\wh\theta_{\tz}(\zz)] \in \pi \BB Z. 
\eqen
Fix $\ep > 0$ and let $T,\delta$ be as in~\eqref{eqn-angle-hit-unif}. Let $\tau_0^\ep = 0$ and inductively define 
\eqbn
\tau_k^\ep := \inf\left\{ \tz \geq \tau_{k-1}^\ep + T : [\wh\theta_{\tz}(\zz) ] \in [x,x+\pi-\ep] \right\}  .
\eqen
Note that $\tau_k^\ep$ could in principle be infinite. 
By the strong Markov property and the $\pi$-periodicity of the coefficients of the SDE~\eqref{eqn-theta-t-bm} for $\wh\theta_{\tz}$, we get that the conditional law of the increment $\{[\wh\theta_{\tz}(\zz)] : \tz \geq \tau_k^\ep \}$ given $\{ \wh\theta_{\tz}(\zz)  : \tz \leq \tau_k^\ep\}$ is the same as the law of the process $[\wh\theta_{\tz}]$ started from the point $[\wh\theta_{\tau_k^\ep}] \in [x,x+\pi-\ep]$ at time $\tz=0$. 
By this and~\eqref{eqn-angle-hit-unif}, the conditional probability given the past that $[\wh\theta_{\tz}(\zz)]$ hits $x$ during the time interval $[\tau_k^\ep,\tau_{k+1}^\ep]$ is always at least $\delta$. Hence, on the event that $\tau_k^\ep < \infty$ for every $k\geq 1$, a.s.\ $[\wh\theta_{\tz}(\zz)]$ hits $x$ infinitely often. 

Hence, it suffices to show that a.s.\ there exists $\ep > 0$ such that $\tau_k^\ep < \infty$ for every $k\geq 1$. 
Indeed, if such an $\ep$ does not exist, then for every $\ep > 0$ it holds for each large enough time $\tz$ that $[\wh\theta_{\tz}(\zz) ] \in (x+\pi-\ep,x+\pi)$. This implies that $\lim_{\tz\to\infty} e^{2i\wh\theta_{\tz}(\zz)} = e^{2i (x+\pi)} = e^{2i x}$.  
The law of $e^{2 i \wh\theta_{\tz}(\zz)}$ is stationary in $\tz$ and this law has a density with respect to Lebesgue measure, so is non-atomic. Since $x$ is deterministic, this implies that
\eqbn
\BB P\left[\lim_{\tz\to\infty} e^{2 i \wh\theta_{\tz}(\zz) } = e^{2 i x } \right]  = 0 .
\eqen

It remains to prove~\eqref{eqn-angle-hit-unif}. Consider the time-changed process $X_r  = \wh{\theta}_{\omega(r)}$ with $\omega(r)$ defined by 
\eqb \label{eqn-omega-time}
r = 2 \int_0^{\omega(r)} D(\wh{\theta}_{\tz} ) d\tz, 
\eqe 
defined for $r\in [0,R]$ where 
\eqbn
R := \inf\left\{r > 0 : X_r \in \{x,x+\pi\} \right\}  .
\eqen  
By time-changing the SDE~\eqref{eqn-theta-t-bm}, we get that for $r \in [0,R]$, the process $X_r$ satisfies the SDE
\eqb \label{eqn-omega-sde}
dX_r = \frac{\mu(X_r)}{2D(X_r)} dr + dZ_r,
\eqe
where 
\eqbn
Z_r =   \int_0^{\omega(r)}   \sqrt{2 D(\wh\theta_\tz)}  \,dW'_{\tz} ,
\eqen
which is a continuous martingale with quadratic variation $dr$, i.e., a standard linear Brownian motion. 

We will now show that with positive probability, $R < \infty$ and $X_R = x$ using~\cite[Lemma 1.24]{lawler-book} (which in turn is based on a scale function argument). 
Let $\psi : (x,x+\pi) \to \BB R$ be defined by 
\eqbn
\psi(v) = \int_{x+\pi/2}^v \exp\left( - \int_{x+\pi/2}^s \frac{\mu(t)}{ D(t)} \,dt  \right) \,ds .
\eqen 
We have $\lim_{t\to x} \mu(t) /   D(t) = \mu(x) / D(x)$ (which is a finite number, since $D(x)\not=0$) if $c < \sqrt{a b}$ and $\lim_{t\to x} \mu(t) /   D(t) = -\infty$ if $c =\sqrt{a b}$. In either case,
\eqbn
\liminf_{s \to x} \int_{x+\pi/2}^s \frac{\mu(t)}{ D(t)} \,dt > -\infty .
\eqen
Therefore,
\eqbn
\limsup_{s \to x} \exp\left( - \int_{x+\pi/2}^s \frac{\mu(t)}{ D(t)} \,dt  \right) < \infty 
\eqen
and so $\liminf_{v \to x} \psi(v) > -\infty$. By~\cite[Lemma 1.24]{lawler-book}, we get that with positive probability $R < \infty$ and $X_R  = x$.  

Since $X$ is a time change of $\wh\theta$, we see that for each $y\in (x,x+\pi)$, it holds with positive probability that the process $\wh\theta_{\tz}$ started from $\wh\theta_0 = y$ hits $x$ in finite time. In particular, there exists some finite $T >0$ (depending on $y$) such that 
\eqb \label{eqn-angle-hit-fixed}
\BB P^y\left[  \exists \tz \in [0,T] \: \text{such that} \: \wh\theta_{\tz} = x \right]  > 0 .
\eqe  
Since $\wh\theta_{\tz}$ is a continuous process, the strong Markov property shows that the probability in~\eqref{eqn-angle-hit-fixed} is a non-increasing function of $y$ on $(x,x+\pi)$. Consequently, for every $\ep > 0$ there exists $T , \delta >0$ such that~\eqref{eqn-angle-hit-unif} holds.
\end{proof}

\section{Phases of $\SLEG$}
\label{sec-phases}

In this section we prove Theorem~\ref{thm-phases}, which characterizes the phases of $\SLEG$ in terms of the values of two definite integrals, which are functions of the entries $a,b,c$ of $\Sigma$. 

Throughout this section, we let $\zz$ be a random point on the unit circle whose argument is sampled from the stationary distribution $p(u) \BB 1_{[0,2\pi]}(u)\,du$, independent from the driving Brownian motion, as in Lemma~\ref{lem-p}. We also let $\sigma= \sigma_\zz$ be the time change as in Definition~\ref{defn-sigma} and we define the time-changed processes $\wh f_{\tz} = f_{\sigma_{\zz}(\tz)}$ and $\wh\theta_{\tz} = \theta_{\sigma_{\zz}(\tz)}$ as in Lemma \ref{lem-polar}.

We define the pair of definite integrals~\ref{item-I} and~\ref{item-II} as in Theorem~\ref{thm-phases}. Since $p \BB 1_{[0,2\pi]}$ is the density for the law of $\arg \zz$, the law of $e^{i\wh\theta_\tz(\zz)}$ is stationary, and the function $\nu$ of~\eqref{eqn-def-mu} is $\pi$-periodic, we have 
\eqb \label{eqn-integral-nu}
\ref{item-I}  = \BB{E}(\nu(\wh \theta_\tz(\zz))) \quad\text{and} \quad
\ref{item-II} = \BB{E}(\nu(\wh \theta_\tz(\zz)) + 2 \cos(2\wh \theta_\tz(\zz))) ,\quad\forall \tz \geq 0. 
\eqe  

We will deduce Theorem~\ref{thm-phases} from the following two propositions, which describe the relationship between each of the quantities $~\ref{item-I},~\ref{item-II}$  and the geometric behavior of the $\SLEG$ left hulls.  The first proposition relates~\ref{item-I} to the absorbing time $T_z$.

\begin{prop} 
\label{prop-i} Let~\ref{item-I} be as in Theorem~\ref{thm-phases}.
\begin{enumerate}[label=(\alph*)]
\item If $~\ref{item-I} \geq 0$, then for each $z \in \BB{C} \setminus \{0\} $, a.s.\ $T_z = \infty$.  
\item If $~\ref{item-I} < 0$, then for each $z \in \BB{C} \setminus \{0\} $, a.s.\ $T_z < \infty$.  
\end{enumerate}
\end{prop}

The second proposition relates~\ref{item-II}  to the limiting distance between a point $z \in \BB{C}$ and the left hull $L_t$ as $t \rta T_z$.

\begin{prop} 
\label{prop-ii} Let~\ref{item-II} be as in Theorem~\ref{thm-phases}.
\begin{enumerate}[label=(\alph*)]
    \item \label{item-ii-a}
If $~\ref{item-II} < 0$, then for each $z \in \BB{C} \setminus \{0\} $, a.s.\ $\lim_{t\rta T_z^-} \dist(z,L_t) = 0$.  
\item \label{item-ii-b}
If $~\ref{item-II}>0$, then for each $z \in \BB{C} \setminus \{0\}$, a.s.\ $\lim_{t \rta T_{z}^-} \dist(z,L_t) > 0$.
\end{enumerate}
\end{prop}

We emphasize that Proposition~\ref{prop-ii} applies regardless of whether $T_z < \infty$ or $T_z=\infty$. Unlike in the case of Proposition~\ref{prop-i}, Proposition~\ref{prop-ii} does not say what happens when $~\ref{item-II} = 0$.
The proof of Proposition~\ref{prop-i} is relatively short; the proof of Proposition~\ref{prop-ii} (especially the second assertion) is much more involved.

\subsection{Preliminary lemmas}

Here we record some general lemmas that we use in our proofs of Propositions~\ref{prop-i} and~\ref{prop-ii}. To prove both these propositions, we first consider the case in which $z$ is the random point $\zz$ sampled from the stationary distribution. Then, we extend the result to general $z\in\BB C\setminus \{0\}$ by applying the following lemma.

\begin{lem} \label{lem-all-z}
Let $\{\mcl E(z)\}_{z\in\BB C\setminus \{0\}}$ be a collection of events for the $\SLEG$ process $(f_t)_{t\geq 0}$ with the following properties.
\begin{itemize}
\item \textbf{Scale invariance:} a.s., for each $r >0$, $\mcl E(z)$ occurs if and only if $\mcl E(r z)$ occurs for the re-scaled Loewner chain $t\mapsto r f_{t/r^2}(\cdot/r)$. 
\item \textbf{Future dependence:} a.s., for each $s > 0$, $\mcl E(z) \cap \{s < T_z \}$ occurs if and only if $s <T_z$ and $\mcl E(f_s(z))$ occurs for the Loewner chain $t\mapsto f_{s,s+t}$ from~\eqref{eqn-conc}. 
\end{itemize}
If $\BB P(\mcl E(\zz)) = 1$ (recall that $\arg \zz$ is sampled from $p(u) \BB 1_{[0,2\pi]}(u)$), then $\BB P(\mcl E(z)) = 1$ for every $z\in\BB C\setminus \{0\}$.
\end{lem}

Examples of events $\mcl E(z)$ for which the hypotheses of Lemma~\ref{lem-all-z} apply are $\{T_z < \infty\}$ and $\{\lim_{t\rta T_z^-} \op{dist}(z,L_t) = 0\}$. 

\begin{proof}[Proof of Lemma~\ref{lem-all-z}]
Since the density $p \BB 1_{[0,2\pi]}$ of the law of $\arg \zz$ is non-zero on $[0,2\pi]$ except possibly on a discrete set, the hypothesis that $\BB P(\mcl E(z)) = 1$ implies that the set
\eqb \label{eqn-nohit-set}
\Lambda := \left\{z\in \BB C : |z| = 1 \: \text{and} \: \BB P(\mcl E(z)) =1 \right\}
\eqe
is the whole unit circle except possibly for a set of zero Lebesgue measure. 

Now consider an arbitrary point $z\in\BB C\setminus \{0\}$ and a stopping time $\tau$ for $(f_t)_{t\geq 0}$ with $\tau < T_z$ a.s.\ (e.g., $\tau$ could be one of the times $\sigma_z(\tz)$ from Definition~\ref{defn-sigma}). By our hypotheses for $\mcl E(z)$, a.s.\ the event $\mcl E(z) \cap \{\tau < T_z\} = \mcl E(z)$ occurs if and only if the event $\mcl E(f_\tau(z) / |f_\tau(z)|) = \mcl E(e^{i \theta_\tau(z)} )$ occurs for the Loewner chain 
\eqbn
t \mapsto |f_\tau(z)|^{-1} f_{\tau , \tau + |f_\tau(z)|^{-2} t}(|f_\tau(z)| \cdot) .
\eqen 
By scale invariance and the strong Markov property (Lemmas~\ref{lem-scaling} and~\ref{lem-markov}), the conditional law of this latter Loewner chain given $(f_t)_{t\leq \tau}$ is the same as the law of $(f_t)_{t\geq 0}$. Hence, $\BB P(\mcl E(z)) = \BB P(\mcl E(z'))$ for $z' = e^{i \theta_\tau(z)}$.
  
Since $t\mapsto \theta_t(z) $ is non-constant and continuous and the set $\Lambda$ of~\eqref{eqn-nohit-set} is dense in the unit circle, for any $q > 0$ we can find a stopping time $\tau$ such that $\tau  < T_z$ a.s.\ and $\BB P(e^{i  \theta_\tau(z)} \in \Lambda) \geq q$. Therefore, $\BB P(\mcl E(z)) \geq q$, so since $q$ is arbitrary, $\BB P(\mcl E(z)) = 1$. 
\end{proof}

Due to~\eqref{eqn-integral-nu} and the SDEs~\eqref{eqn-abs-zt} and~\eqref{eqn-abs-zt'}, we can express the expectations of $\log|\wh f_{\tz}(\zz)|$ and $\log(|\wh f_\tz'(\zz)| / |\wh f_\tz(\zz)|)$ in terms of the integrals $~\ref{item-I}$ and $~\ref{item-II}$. 
We will use the following general result for It\^o processes in order to transfer from statements about the expectations of these processes to statements about their asymptotic behavior. 

\begin{lem}
\label{lem-stationary-sde}
Let $m\in\BB N$ and let $(B^1,\dots,B^m)$ be an $m$-dimensional Brownian motion (with an arbitrary covariance matrix). 
Also let $X : [0,\infty) \rta \BB C $ be a process which is adapted to the filtration for $(B^1,\dots,B^m)$, whose law is stationary in time, and which satisfies an SDE of the form
\eqbn
dX_t = g^0(X_t)\,dt + \sum_{j=1}^m g^j(X_t) \,dB_t^j 
\eqen 
for some measurable functions $g^0,\dots,g^m : \BB C\to\BB C$.  
Assume that there exists a deterministic $\ol x \in\BB C $ such that a.s.\ there exist arbitrarily large times $t$ for which $X_t =  \ol x $.
Let $f^0, f^1,\dots,f^m : \BB C\rta\BB R$ be measurable functions such that $\BB E\left( \int_0^1 (f^j(X_s ))^2 ds \right) < \infty$ for each $j=0,1,\dots,m$. 
Let $Y:[0,\infty) \rta \BB{R}$ be a stochastic process satisfying the SDE 
\[
dY_t = f^0(X_t ) \, dt + \sum_{j=1}^m f^j(X_t ) \, dB_t^j .
\] 
Almost surely, $Y_t/t \rta  \BB{E}(f^0(X_0^0))$ as $t \rta \infty$.
Moreover, if $\BB{E}(f^0(X_0 )) = 0$, then a.s.\ for each $\ep> 0$ there are arbitrarily large times $t$ for which $|Y_t  |  < \ep$.
\end{lem}

In our applications of Lemma~\ref{lem-stationary-sde}, we will take $X$ to be equal to the stationary process $\tz \mapsto e^{2 i\wh\theta_{\tz}(\zz)}$. The needed recurrence condition for $X$ follows from Lemma~\ref{lem-recurrence}. 
We prove Lemma~\ref{lem-stationary-sde} by applying the following elementary result from ergodic theory.

\begin{lem}
\label{lem-recurrent}
Let $\xi_1,\xi_2,\ldots$ be a stationary sequence taking values in $\BB{R}$ with $\mcl I$ the associated invariant $\sigma$-algebra. If $\BB{E}(\xi_1|\mcl I) = 0$, then a.s.\ for each $\ep > 0$ there exist infinitely many values of $n$ for which $\left| \sum_{k=1}^n \xi_k   \right|  <\ep$.
\end{lem}

\begin{proof}
The proof is essentially identical to the proof of~\cite[Theorem 7.3.2]{durrett}, which gives an analogous statement for random variables taking values in $\BB Z$.
\end{proof}

\begin{proof}[Proof of Lemma~\ref{lem-stationary-sde}]
The process $Y$ has stationary increments, so by the Birkhoff ergodic theorem and a standard continuity argument to deal with non-integer times, a.s.\ 
$Y_t/t \rta \BB{E}\left( Y_1 - Y_0\,|\, \mcl I\right)$ as $t \rta \infty$, 
where $\mcl I$ is the invariant $\sigma$-algebra associated to the stationary sequence of increments of the process $Y_t$.  

We now argue that $\lim_{t\to\infty} Y_t/t$ is a.s.\ equal to a deterministic constant, and so a.s.\ $Y_t/t \rta \BB{E}(Y_1-Y_0)$. To see this, let $\ol x$ be as in the lemma statement, let $\tau_1$ be the smallest time $t$ for which $X_t  = \ol x$, and for $k\geq 2$ inductively let $\tau_k$ be the smallest $t\geq \tau_{k-1} + 1$ for which $X_t = \ol x$. By hypothesis, a.s.\ $\tau_k<\infty$ for each $k\geq 1$. 
By the strong Markov property and the SDEs for $X$ and $Y$, we obtain that for each $k\geq 1$, the conditional law of $\{Y_{t+\tau_k} - Y_{\tau_k}\}_{t \geq 0}$ given $(B^1,\dots,B^m)|_{[0,\tau_k]}$ does not depend on $(B^1,\dots,B^m)|_{[0,\tau_k]}$, i.e., $\{Y_{t+\tau_k} - Y_{\tau_k}\}_{t \geq 0}$ is independent from $(B^1,\dots,B^m)|_{[0,\tau_k]}$. 
Hence
\eqbn
\lim_{t\to\infty} \frac{Y_t }{t} = \lim_{t\to\infty} \frac{Y_{t+\tau_k} - Y_{\tau_k}}{t} 
\eqen
is independent from $(B^1,\dots,B^m)|_{[0,\tau_k]}$. Since $\tau_k \geq k$, we get that the limit is measurable with respect to the tail $\sigma$-algebra of $(B^1,\dots,B^m)$, so is a.s.\ equal to a deterministic constant. 

Since $\BB E\left( \int_0^1 (f^j(X_t^j))^2 dt \right) < \infty$, the integral $\int_0^1 f^j(X_t^j) dB_t^j$ has mean zero. 
By this and the stationarity of  $X_t$, we have $\BB{E}(Y_1-Y_0) = \BB{E}(f^0(X_0 ))$.  This proves that a.s.\ $Y_t/t \rta \BB{E}(f^0(X_0 ))$. 

Finally, assume that $\BB{E}(f^0(X_0 )) = 0$.  
By applying Lemma~\ref{lem-recurrent} to the stationary sequence $\xi_n = Y_n - Y_{n-1}$, we conclude that a.s.\ for each $\ep> 0$ there are arbitrarily large times $t$ for which $|Y_t|  < \ep$. 
\end{proof}

\subsection{Proof of Proposition~\ref{prop-i} }

Proposition~\ref{prop-i} turns out to be an easy consequence of Lemmas~\ref{lem-all-z} and~\ref{lem-stationary-sde} combined with the following lemma.

\begin{lem}
\label{lem-swallow-distance}
Let $z \in \BB{C} $, let $\sigma=\sigma_z$ be as in Definition~\ref{defn-sigma}, and denote time changed processes by a hat. A.s.\ if $T_z < \infty$ then $\log |\wh f_\tz(z)| \rta -\infty$ as $\tz \rta \infty$.
\end{lem}

\begin{proof}
By Lemma~\ref{lem-endpt-cont}, on the event $\{T_z < \infty\}$, we have $\lim_{t \rta T_z^-} \log |f_t(z)| = -\infty$.
As we noted in Remark~\ref{remark-sigma}, we have $\sigma(\infty) = T_z$, so this implies that $\lim_{\tz \rta \infty} \log |\wh f_\tz(z)| = -\infty$. 
\end{proof}

\begin{proof}[Proof of Proposition~\ref{prop-i}]
By the SDE~\eqref{eqn-abs-zt} for $\log |\wh f_\tz(\zz)|$ and~\eqref{eqn-integral-nu}, for each $\tz>0$, the integral $~\ref{item-I} = \int_0^{2\pi} \nu(u) p(u) \, du = \BB E(\nu(\wh\theta_\tz(\zz))$ is equal to $(1/\tz) \BB{E} \log |\wh f_\tz(\zz)|$. We will apply Lemma~\ref{lem-stationary-sde} with $Y_\tz = \log |\wh f_\tz(\zz)|$, $m=2$, and $X_\tz = e^{2 i \wh\theta_\tz(\zz)}$. We recall that the needed recurrence condition for $X$ follows from Lemma~\ref{lem-recurrence}, and that the needed SDEs for $X$ and $Y$ come from Lemma~\ref{lem-polar} and the fact that the coefficients of the SDEs in this lemma are $\pi$-periodic.

\begin{enumerate}[label=(\alph*)]
\item 
If $~\ref{item-I} \geq 0$, then by Lemma~\ref{lem-stationary-sde} a.s.\ $\log |\wh f_\tz(\zz)|$ does not converge to $-\infty$ as $\tz \rta \infty$.
By Lemma~\ref{lem-swallow-distance}, this implies that a.s.\ $T_{\zz} = \infty$. By Lemma~\ref{lem-all-z}, a.s.\ $T_z = \infty$ for each fixed $z\in\BB C\setminus \{0\}$. 
\item  
If $~\ref{item-I} = - \alpha  < 0$, then a.s.\ $\log |\wh f_\tz(\zz)|/\tz \rta - \alpha < 0$ as $\tz \rta \infty$.  Going back to the original time parametrization, we get that for all $t$ in some nonempty interval $(t_*,T_{\zz})$, a.s.\
\eqb \label{eqn-neg-I1-bd}
\log |f_{t}(\zz)| < - \frac{\alpha}{2} \int_0^{t} \frac{1}{|f_s(\zz)|^2} ds .
\eqe
Let $F(t) = \int_0^{t} 1/|f_s(\zz)|^2 ds$, so that $F'(t) = |f_t(\zz)|^{-2}$. 
By re-arranging~\eqref{eqn-neg-I1-bd} we get that a.s.\ $F'(t) e^{-\alpha F(t)} > 1$.  Integrating by substitution yields $e^{-\alpha F(t)} - e^{-\alpha F(t_*)}< - \alpha(t-t_*)$   for all $t \in (t_*,T_{\zz})$.   The latter cannot hold when
$t >  t_* + \frac{1}{\alpha}e^{-\alpha F(t_*)} $, so we must have $T_{\zz} \leq t_* + \frac{1}{\alpha}e^{-\alpha F(t_*)} < \infty$ almost surely.
By Lemma~\ref{lem-all-z}, a.s.\ $T_z < \infty$ for each $z \in \BB C\setminus \{0\}$.
\end{enumerate}
\end{proof}

\subsection{Proof of Proposition~\ref{prop-ii}}
\label{sec-prop-ii}

This section is devoted to proving Proposition~\ref{prop-ii}.
Throughout this section, we use the following notation.
If $z \in \BB{C}$ and $f,g:\BB{C} \rta \BB{R}$, we write $f(z) \asymp g(z)$  if $k^{-1} \leq f/g \leq k$ for some constant $k$, and $f(z) \lesssim g(z)$  if $f/g \leq k$ for some constant $k$. 
 
To prove Proposition~\ref{prop-ii}, we begin by relating~\ref{item-II} to the behavior of the random variable $|\wh{f_{\tz}'}(\zz)|/|\wh{f_{\tz}}(\zz)|$.  

\begin{lem}
We have
\eqbn
\BB{E}\left( \log \frac{|\wh{f_{\tz}'}(\zz)|}{|\wh{f_{\tz}}(\zz)|} \right)  = - \tz \times \ref{item-II}  ,\quad\forall \tz \geq 0. 
\eqen
If $~\ref{item-II} < 0$, then a.s.\ $\lim_{\tz \rta \infty} |\wh{f_{\tz}'}(\zz)| / |\wh{f_{\tz}}(\zz)| = +\infty$.  
\label{lem-integral-ii}
\end{lem}

\begin{proof}
By Lemma~\ref{lem-polar},  
\eqb
d\log\frac{|\wh f_{\tz}' (z)|}{|\wh f_{\tz} (z)|} = -(\nu(\wh{\theta_\tz}) + 2 \cos{2 \wh{\theta_\tz}}) \, d\tz -  \sqrt{a}\cos{\wh\theta_\tz}\, dW_\tz - \sqrt{b}\sin{\wh\theta_\tz}\,  d\wt W_\tz ,
\label{eqn-sde-ratio}
\eqe
where $W$ and $\wt W$ are Brownian motions.
Since $e^{i\wh\theta_t(\zz)}$ is stationary in $\zz$, it follows that
\eqbn
\BB{E}\left( \log \frac{|\wh{f_{\tz}'}(\zz)|}{|\wh{f_{\tz}}(\zz)|} \right) = -\tz\BB{E}(\nu(\zz) + 2 \cos{(2 \zz)}) = - \tz \times \ref{item-II}  .
\eqen
This gives the first part of the lemma statement, and the second part follows from applying Lemma~\ref{lem-stationary-sde}.
\end{proof}

The last lemma is enough to deduce the first part of Proposition~\ref{prop-ii}.

\begin{proof}[Proof of Proposition~{\hyperref[item-ii-b]{\ref*{prop-ii}\ref*{item-ii-a}}}]
By Lemma~\ref{lem-integral-ii}, if $~\ref{item-II} < 0$ then a.s.\ $\lim_{\tz \rta \infty} |\wh{f_{\tz}'}(\zz)|/|\wh{f_{\tz}}(\zz)| = +\infty$, i.e., 
\[
\lim_{t\rta T_z^-} |f_t'(\zz)| / |f_t(\zz)| = \infty.
\]
On the other hand, by the Koebe quarter theorem (see, e.g.,~\cite[Corollary 3.18]{lawler-book}) followed by the fact that $0\in R_t$, 
\eqb \label{eqn-hit-koebe}
|f_t'(\zz)| \asymp \frac{\dist(f_t(\zz),R_t)}{\dist(\zz,L_t)} \leq \frac{|f_t(\zz)|}{|\dist(\zz,L_t)|} ,
\eqe
with universal implicit constants in $\asymp$. 
Therefore, a.s.\ $\dist(\zz,L_t) \rta 0$ as $t \rta T_{\zz}$. By Lemma~\ref{lem-all-z}, for each fixed $z\in\BB C\setminus \{0\}$, a.s.\ $\dist(z,L_t) \rta 0$ as $t \rta T_{z}$.
\end{proof}

The proof of Proposition~\hyperref[item-ii-b]{\ref*{prop-ii}\ref*{item-ii-b}} is considerably more involved.  
The main difficulty is that $\dist(f_t(\zz),R_t)$ could be much smaller than $ |f_t(\zz)|$ when $t$ is close to $T_z$. This would mean that the inequality in~\eqref{eqn-hit-koebe} is far from sharp, i.e., $|f_t'(\zz)| / |f_t(\zz)|$ is much smaller than $1/\op{dist}(\zz,L_t)$. We will not rule this out directly. Rather, our proof will be based on a more intricate argument which does not require a direct comparison of $|f_t'(\zz)| / |f_t(\zz)|$ and $1/\op{dist}(\zz,L_t)$.  
The main step in the proof is the following lemma. 

\begin{lem} 
If $~\ref{item-II} > 0$, then the probability that $\lim_{t \rta T_{\zz}^-} \dist(\zz,L_t) = 0$ is strictly less than $1$.
\label{lem-converse-w-o-upgrade}
\end{lem}

Once we prove Lemma~\ref{lem-converse-w-o-upgrade}, we can easily deduce Proposition~\hyperref[item-ii-b]{\ref*{prop-ii}\ref*{item-ii-b}} by a tail triviality argument.
We  now give an outline of the proof of Lemma~\ref{lem-converse-w-o-upgrade}. Our aim is to prove the contrapositive, i.e., we will assume that $\lim_{t\rta T_{\zz}^-} \dist(\zz,L_t) = 0$ a.s.\ and deduce that $~\ref{item-II} \leq 0$. 

To this end, we define the ``future" left hulls $(L_{t,s+t})_{s\geq 0}$ at time $t\geq 0$ as in the concatenation property~\eqref{eqn-conc}. 
Also let $\sigma = \sigma_{\zz}$ be the usual time change (Definition~\ref{defn-sigma}). 
Using estimates for conformal maps, we will show that the assumption that $\lim_{t\rta T_{\zz}^-} \dist(\zz,L_t) = 0$ a.s.\ implies that if $0 < \tz < \sz$ and $\sz$ is large, then with high probability,
\eqb
\log \frac{\dist(f_{\sigma(\tz)}(\zz), L_{\sigma(\tz),\sigma(\sz + \tz)})}{|f_{\sigma(\tz)}(\zz)|} - \log(\dist(\zz,L_{\sigma(\sz+\tz)})) \leq \log \frac{|f_{\sigma(\tz)}'(\zz)|}{|f_{\sigma(\tz)}(\zz)|}  + C ,  
\label{eqn-dist-3-outline}
\eqe
where $C>0$ is a universal constant. The key idea in the proof of~\eqref{eqn-dist-3-outline} is as follows. Since $\sigma(\sz+\tz) \rta T_z^-$ as $\sz\rta\infty$, the assumption that $\lim_{t\rta T_{\zz}^-} \dist(\zz,L_t) = 0$ implies that $\lim_{\sz \rta\infty} \dist(\zz,L_{\sigma(\tz),\sigma(\sz+\tz)}) = 0$. Consequently, if $\sz$ is much larger than $\tz$, then $f_{\sigma(\tz)}(\zz)$ is closer to $L_{\sigma(\tz),\sigma(\sz + \tz)}$ than to $R_{\sigma(\tz)}$. This together with distortion estimates leads to~\eqref{eqn-dist-3-outline}. 
See the final part of the proof of Lemma~\ref{lem-converse-w-o-upgrade} for details. 

By Lemma~\ref{lem-integral-ii}, the expectation of the right side of~\eqref{eqn-dist-3-outline} is equal to $C -\tz \times \ref{item-II}$. We will show that the expectation of the left side of~\eqref{eqn-dist-3-outline} is non-negative, which implies that $\ref{item-II} \leq 0$ upon taking $\tz$ to be sufficiently large. To accomplish this, we will use the stationarity of $\tz\mapsto f_{\sigma(\tz)}(\zz) / |f_{\sigma(\tz)}(\zz)|$ and a scaling argument to get that 
\eqb \label{eqn-stationary-outline}
\frac{\dist(f_{\sigma(\tz)}(\zz), L_{\sigma(\tz),\sigma(\sz + \tz)})}{|f_{\sigma(\tz)}(\zz)|} \eqD \dist(\zz,L_{\sigma(\sz )}) ;
\eqe 
see Lemma~\ref{lem-computation}. We will also show that the expectations of both terms on the left side of~\eqref{eqn-dist-3-outline} are finite (Lemma~\ref{lem-log-finite-exp}). 
Since $\op{dist}(\zz , L_{\sigma(\sz + \tz)}) \leq \op{dist}(\zz , L_{\sigma(\sz)})$, the relation~\eqref{eqn-stationary-outline} shows that the expectation of the left side of~\eqref{eqn-dist-3-outline} is non-negative, as required. 

We will need two lemmas for the proof of Lemma~\ref{lem-converse-w-o-upgrade}, which can be read in any order.

By our choice of $\zz$, we know that the law of $e^{i \theta_{\sigma(\tz)}(\zz)} = f_{\sigma(\tz)}(\zz) / |f_{\sigma(\tz)}(\zz)|$ is stationary in $\tz$. We need the following extension of this stationarity property which also gives information about the left hulls $L_{t,s+t}$ associated to the conformal maps $f_{t,s+t}$ which appear in the concatenation property~\eqref{eqn-conc}. 

\begin{lem} 
With $\sigma = \sigma_{\zz}$ as in Definition~\ref{defn-sigma}, 
\[
\left( \frac{ f_{\sigma(\tz)}(\zz) }{|f_{\sigma(\tz)}(\zz)| } , \frac{ L_{ \sigma(\tz), \sigma(\sz+\tz)} }{ |f_{\sigma(\tz)}(\zz)| } \right) \eqD (\zz, L_{\sigma(\sz)})
\]
for each fixed $\sz ,\tz \geq 0$.
\label{lem-computation}
\end{lem}

\begin{proof}
The proof is an elementary, but somewhat tedious, scaling calculation.
For each $u \geq 0$, let 
\eqb \label{eqn-computation-tilde-f}
\wt f_u(w) := |f_{\sigma(\tz)}(\zz)|^{-1} f_{\sigma(\tz), |f_{\sigma(\tz)}(\zz)|^2 u + \sigma(\tz)}(|f_{\sigma(\tz)}(\zz)|w).
\eqe
By scale invariance and the strong Markov property (Lemmas~\ref{lem-scaling} and~\ref{lem-markov}), $(\wt f_u)_{u \geq 0} \eqD (f_u)_{u \geq 0}$ and $(\wt f_u)_{u\geq 0}$ is independent from the Loewner evolution run until time $\sigma(\tz)$. The left hull of $\wt f_u$ is given by
\[
\wt L_u := |f_{\sigma(\tz)}(\zz)|^{-1} L_{\sigma(\tz), |f_{\sigma(\tz)}(\zz)|^2 u + \sigma(\tz)}.
\]

By stationarity (Lemma~\ref{lem-p}), $\zz \eqD f_{\sigma(\tz)}(\zz)/|f_{\sigma(\tz)}(\zz)|$. Since $f_{\sigma(\tz)}(\zz)/ |f_{\sigma(\tz)}(\zz)|$ is independent from $(\wt f_u)_{u\geq 0}$, if we define the time parametrization $\sz \mapsto \wt \sigma(\sz)$ so that
\[
\sz = \int_0^{\wt \sigma(\sz)} \frac{1}{\left|\wt f_u(f_{\sigma(\tz)}(\zz) /  |f_{\sigma(\tz)}(\zz)| ) \right|^2} \, du,
\]
then 
\eqb
\left( \frac{f_{\sigma(\tz)}(\zz)}{|f_{\sigma(\tz)}(\zz)|} , \wt L_{\wt \sigma(\sz)} \right)_{\sz \geq 0} \eqD (\zz, L_{\sigma(\sz)})_{\sz \geq 0}. \label{eqn-wt-L-dist}
\eqe
We have  
\alb
\sz &= \int_0^{\wt \sigma(\sz)} \frac{1}{\left|\wt f_u(f_{\sigma(\tz)}(\zz) /  |f_{\sigma(\tz)}(\zz)| ) \right|^2} \, du \\
&= \int_0^{\wt \sigma(\sz)} \frac{|f_{\sigma(\tz)}(\zz)|^2}{\left| f_{\sigma(\tz), |f_{\sigma(\tz)}(\zz)|^2 u + \sigma(\tz)}(f_{\sigma(\tz)}(\zz)) \right|^2} \, du \quad \text{(by~\eqref{eqn-computation-tilde-f})} \\
&= \int_0^{\wt \sigma(\sz)} \frac{|f_{\sigma(\tz)}(\zz)|^2}{\left|f_{ |f_{\sigma(\tz)}(\zz)|^2 u + \sigma(\tz)}(\zz)\right|^2} \, du \quad \text{(concatenation property)} \\
&= \int_{\sigma(\tz)}^{|f_{\sigma(\tz)}(\zz)|^2 \wt \sigma(\sz) + \sigma(\tz)} \frac{1}{|f_v(\zz)|^2} \,  dv \quad \text{(substitute $v =|f_{\sigma(\tz)}(\zz)|^2 u + \sigma(\tz)$)} \\
&= - \tz +  \int_{0}^{|f_{\sigma(\tz)}(\zz)|^2 \wt \sigma(\sz) + \sigma(\tz)} \frac{1}{|f_v(\zz)|^2} \, dv \quad \text{(definition of $\sigma(\tz)$)} .
\ale
By the definition of $\sigma(\sz+\tz)$, it follows that $|f_{\sigma(\tz)}(\zz)|^2 \wt \sigma(\sz) + \sigma(\tz) = \sigma(\sz+\tz)$; i.e., $\wt \sigma(\sz) = |f_{\sigma(\tz)}(\zz)|^{-2} (\sigma(\sz+\tz) - \sigma(\tz))$. This implies that $\wt L_{\wt \sigma(\sz)} = |f_{\sigma(\tz)}(z)|^{-1} L_{\sigma(\tz), \sigma(\sz+\tz)}$, and plugging this into~\eqref{eqn-wt-L-dist} yields the desired equality in law.
\end{proof}

We also need the following estimate for the distance between $\zz$ and the left hull at a given time in the $\sigma$-time parametrization.

\begin{lem} 
With $\sigma=\sigma_{\zz}$, we have, for each $\tz > 0$, 
\eqbn
\BB{E}\left( \log \frac{ 1}{\dist(\zz,L_{\sigma(\tz)})} \right)  < \infty .
\eqen 
\label{lem-log-finite-exp}
\end{lem}

We note that $|\zz|=1$ and $0 \in L_{\sigma(\tz)}$, so $\dist(\zz,L_{\sigma(\tz)}) \leq 1$, i.e., $\log (1/\dist(\zz,L_{\sigma(\tz)})) \geq 0$.

\begin{figure}[ht!]
\centering
\includegraphics[width=0.8\textwidth]{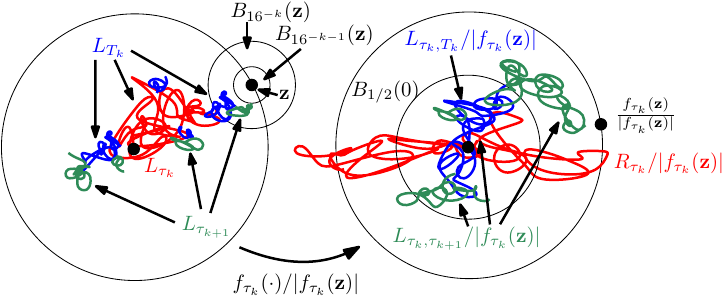}
\caption{Illustration of the objects involved in the proof of Lemma~\ref{lem-log-finite-exp}. 
} \label{fig-log-finite-exp}
\end{figure}

\begin{proof}[Proof of Lemma~\ref{lem-log-finite-exp}]
Fix $\tz > 0$. We will show that there is a constant $C_{\tz} > 0$ (depending on $\tz$ and $\Sigma$) such that 
\eqb \label{eqn-log-finite-show}
\BB{P}\left(\dist(\zz,L_{\sigma(\tz)}) \leq 16^{-k}\right) \leq C_{\tz} 2^{-k+1} , \qquad \forall k \in\BB N ,
\eqe
which immediately implies that $\BB{E} \log(1/\dist(\zz,L_{\sigma(\tz)}) ) < \infty$.

We will prove~\eqref{eqn-log-finite-show} by a multi-scale argument. See Figure~\ref{fig-log-finite-exp} for an illustration. 
For each $k \in \BB{N}$, let $\tau_k$ be the first time that $L_t$ hits the ball of radius $16^{-k}$ centered at $\zz$. 
By the standard distortion estimate~\cite[Corollary 3.23]{lawler-book}, 
\eqb
|f_{\tau_k}(w) - f_{\tau_k}(\zz)| \leq 8|w-\zz| \frac{\dist(f_{\tau_k}(\zz), R_{\tau_k})}{\dist(\zz, L_{\tau_k})} \qquad \forall w \in B_{16^{-k-1}}(\zz)  .
\label{eqn-finite-exp-1}
\eqe
Since $0 \in R_{\tau_k}$ we have $ \dist(f_{\tau_k}(\zz), R_{\tau_k}) \leq  |f_{\tau_k}(\zz)| $ and by the definition of $\tau_k$ we have $\dist(\zz, L_{\tau_k}) \geq 16^{-k}$.  Plugging these bounds into~\eqref{eqn-finite-exp-1}, we deduce that 
\eqb
\frac{f_{\tau_k}(B_{16^{-k-1}}(\zz))}{|f_{\tau_k}(\zz)|} \subset B_{1/2}\left(\frac{f_{\tau_k}(\zz)}{|f_{\tau_k}(\zz)|}\right).
\label{eqn-finite-exp-2}
\eqe

Let $T_k$ be the first time $t$ after $\tau_k$ such that $L_{\tau_k,t}/|f_{\tau_k}(z)|$ is not contained in the ball $B_{1/2}(0)$. By~\eqref{eqn-finite-exp-2}, a.s.\ $T_k \leq \tau_{k+1}$. Observe that the random variable \[  (T_k-\tau_k)/|f_{\tau_k}(\zz)|^2\] is the first time that the left hulls of the Loewner chain
\eqb
\wt f_t(w) :=  f_{\tau_k, |f_{\tau_k}(\zz)|^2 t + \tau_k}(|f_{\tau_k}(\zz)|w), \qquad t \geq 0
\label{eqn-finite-exp-3}
\eqe
exit $B_{1/2}(0)$. 

By scale invariance and the strong Markov property (Lemmas~\ref{lem-scaling} and~\ref{lem-markov}), applied as in the beginning of the proof of Lemma~\ref{lem-computation}, $(\wt f_t)_{t\geq 0}$ is equal in law to $(f_t)_{t \geq 0}$ and is independent from the $\sigma$-algebra $\mcl F_{\tau_k}$ generated by the original Loewner evolution run until time $\tau_k$.  Therefore, the integral
\eqb
\int_0^{(T_k-\tau_k)/|f_{\tau_k}(\zz)|^2} \frac{1}{\left|\wt f_u(f_{\tau_k}(\zz)/|f_{\tau_k}(\zz)|)\right|^2} du
\label{eqn-finite-exp-4}
\eqe
is independent from $\mcl F_{\tau_k}$ and its law does not depend on $k$.  Moreover,~\eqref{eqn-finite-exp-4} simplifies to
\alb
&\int_0^{(T_k-\tau_k)/|f_{\tau_k}(\zz)|^2}
\frac{|f_{\tau_k}(\zz)|^2}{|f_{\tau_k, |f_{\tau_k}(\zz)|^2 u + \tau_k}(f_{\tau_k}(\zz))|^2} du \quad \text{(by~\eqref{eqn-finite-exp-3})} \\
&= \int_0^{(T_k-\tau_k)/|f_{\tau_k}(\zz)|^2} \frac{|f_{\tau_k}(\zz)|^2}{|f_{ |f_{\tau_k}(\zz)|^2 u + \tau_k}(\zz)|^2} du \quad \text{(concatenation property)} \\
&=\int_{\tau_k}^{T_k} \frac{1}{|f_v(\zz)|^2} dv \quad \text{(substitute $v = |f_{\tau_k}(\zz)|^2 u + \tau_k$)} .
\ale

Therefore, the integral $\int_{\tau_k}^{T_k} \frac{1}{|f_v(\zz)|^2} dv$ is independent from $\mcl F_{\tau_k}$ and its law does not depend on $k$.  Hence, we can choose a constant $q = q(\Sigma)  >0$ such that 
\eqbn
\BB{P}\left( \int_{\tau_k}^{T_k} \frac{1}{|f_v(\zz)|^2} dv \geq q \right) \geq \frac{1}{2} \qquad \forall k \in\BB N .
\eqen
Furthermore, the events $\{ \int_{\tau_k}^{T_k} \frac{1}{|f_v(\zz)|^2} dv \geq q\}$ for different values of $k\in\BB N$ are independent. 
By the formula for the binomial distribution, for each fixed $N\in\BB N$ there exists a combinatorial constant $C_N > 0$ (depending only on $N$) such that
\eqbn
\BB P\left( \#\left\{j \in \{1,\dots, k\} : \int_{\tau_j}^{T_j} \frac{1}{|f_v(\zz)|^2} dv \geq q \right\} \leq N \right) \leq C_N 2^{-k}  ,\quad\forall k \in \BB N .
\eqen
Therefore,  
\eqbn
\BB{P}\left( \int_{0}^{ \tau_k } \frac{1}{|f_v(\zz)|^2} dv \leq q N \right) \leq C_N 2^{-k },  \qquad \forall k \in\BB N ,
\eqen
which, by the definition~\eqref{eqn-defn-sigma} of $\sigma(\tz)$ implies that
\eqb \label{eqn-log-finite-end} 
\BB{P}\left( \tau_k \leq \sigma(q N ) \right)  \leq C_N 2^{-k } , \qquad \forall k \in\BB N.
\eqe
For a given $\tz > 0$, we now choose $N = N(\tz , \Sigma) \in \BB N$ such that $q N \geq \tz$. 
Recalling the definition of $\tau_k$, we see that~\eqref{eqn-log-finite-end} for this choice of $C_N$ implies~\eqref{eqn-log-finite-show}. 
\end{proof}

We are now ready for the core part of the proof of Proposition~\hyperref[item-ii-b]{\ref*{prop-ii}\ref*{item-ii-b}}. 

\begin{proof}[Proof of Lemma~\ref{lem-converse-w-o-upgrade}] 
We will prove the contrapositive, i.e., we will show that 
\eqb \label{eqn-converse-show0}
\lim_{u \rta T_{\zz}^-} \dist(\zz,L_u) = 0, \: a.s. \quad \Rightarrow \quad \ref{item-II} \leq 0 .
\eqe
Write $\sigma = \sigma_\zz$. 
In light of Lemma~\ref{lem-integral-ii}, we seek to lower-bound $ \BB{E}\left(  \log \frac{|f_{\sigma(\tz)}'(\zz)|}{|f_{\sigma(\tz)}(\zz)|} \right) $. 
See Figure~\ref{fig-converse} for an illustration of the proof. 
\medskip

\begin{figure}[ht!]
\centering
\includegraphics[width=0.8\textwidth]{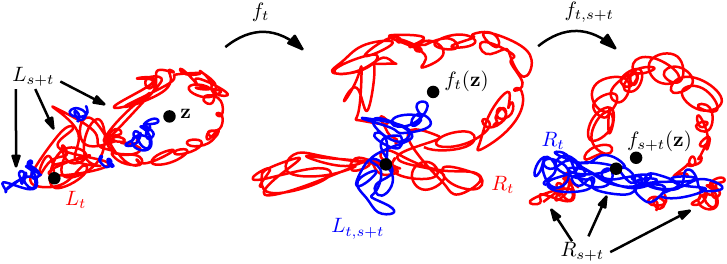}
\caption{Illustration of the proof of Lemma~\ref{lem-converse-w-o-upgrade}. We set $t = \sigma(\tz)$ and $s = \sigma(\sz  + \tz) - \sigma(\sz)$, so that $\sz \rta\infty$ ($\tz$ fixed) is equivalent to $t+s \rta T_z$. On the event $\mcl E_{\sz}$ defined in~\eqref{eqn-dist-event}, in the middle picture $f_t(\zz)$ is closer to the blue hull $L_{t,s+t}$ than to the red hull $R_t$. This together with repeated applications of the Koebe quarter theorem allows us to deduce~\eqref{eqn-dist-3}. We note that in the figure, the hull $L_t$ disconnects $\zz$ from $\infty$. Theorem~\ref{thm-disconnects} tells us that this is typically the case for some $t < T_z$. 
} \label{fig-converse}
\end{figure}

\noindent\textit{Step 1: setup.}
Fix times $\tz , \sz > 0$. We will eventually send $\sz \rta\infty$ and then $\tz\rta\infty$. 
We define 
\eqbn
t = \sigma(\tz) \quad \text{and} \quad s = \sigma(\sz + \tz )-\sigma(\tz) .
\eqen
We have defined $s$ and $t$ this way to reduce clutter in the calculations that follow, since we will be working in the $\sigma$-time parametrization throughout the proof but want to avoid writing the symbol $\sigma$ each time. Note that $s + t \rta T_z$ as $\sz \rta\infty$ (Remark~\ref{remark-sigma}). 

Let $f_{t,s+t}, L_{t,s+t}$ and $R_{t,s+t}$ be the ``future" Loewner maps and hulls as in the concatenation property~\eqref{eqn-conc}. 
Under the assumption that $\lim_{u \rta T_{\zz}^-} \dist(\zz,L_u) = 0$ a.s., we will show that there is an event $\mcl E_{\sz}$ such that a.s.\ $\BB 1_{\mcl E_{\sz}} \rta 1$ as $\sz \rta\infty$ and we have the bound
\eqb
\log \frac{\dist(f_t(\zz), L_{t,s+t})}{|f_t(\zz)|} - \log(\dist(\zz,L_{s+t})) \leq \log \frac{|f_t'(\zz)|}{|f_t(\zz)|} \mathbf{1}_{\mcl{E}_{\sz}} + \log \frac{1}{\dist(\zz,L_t)} \mathbf{1}_{\mcl{E}_{\sz}^c} + C ,
\label{eqn-dist-3}
\eqe
where $C$ is a universal constant. 
\medskip

\noindent\textit{Step 2: proof of~\eqref{eqn-converse-show0} assuming~\eqref{eqn-dist-3}.}
We first analyze the behavior of the right side of~\eqref{eqn-dist-3} as $\sz \rta\infty$. 
The SDE~\eqref{eqn-sde-ratio} for $\log(|f_{\sigma(\cdot)}'(\zz)| / |f_{\sigma(\cdot)}(\zz)|)$ (and the fact that we have set $t=\sigma(\tz)$) shows that 
\eqbn
\BB E\left( \left|\log \frac{|f_{t}'(\zz)| }{ |f_{t}(\zz)|} \right| \right) < \infty .
\eqen 
Since $\BB 1_{\mcl E_{\sz}} \rta 1$, we can apply the dominated convergence theorem to the first term on the right-hand side of~\eqref{eqn-dist-3} to obtain
\eqbn
\BB{E}\left(  \log \frac{|f_t'(\zz)|}{|f_t(\zz)|} \mathbf{1}_{\mcl{E}_{\sz}} \right) \rightarrow \BB{E}\left(  \log \frac{|f_t'(\zz)|}{|f_t(\zz)|} \right) \qquad \text{as $\sz \rta \infty$}.
\eqen 
Also, Lemma~\ref{lem-log-finite-exp} implies that we may apply dominated convergence theorem to the positive part of the second term on the right-hand side of~\eqref{eqn-dist-3} to get 
\eqbn
\BB{E}\left(  \left( \log \frac{1}{\dist(\zz,L_t)} \right)   \mathbf{1}_{\mcl{E}_{\sz}^c}  \right) \rta 0 \qquad \text{as $\sz \rta \infty$}.
\eqen
Thus, the limit of the right-hand side of~\eqref{eqn-dist-3}  as $\sz \rta \infty$ is equal to $\BB{E}\left(  \log \frac{|f_t'(\zz)|}{|f_t(\zz)|} \right) + C$. 

On the other hand, by Lemma~\ref{lem-computation}, 
\eqbn
  \frac{\dist(f_t(\zz), L_{t,s+t})}{|f_t(\zz)|} \eqD \dist(\zz,L_s). 
\eqen
Since $\op{dist}(\zz , L_{s+t}) \leq \op{dist}(\zz , L_s)$ and $\log(1 / \op{dist}(\zz , L_{s+t}))$ has finite expectation by Lemma~\ref{lem-log-finite-exp}, this immediately implies that the left-hand side of~\eqref{eqn-dist-3} has non-negative expectation for any choice of $\sz >0$. By combining this with the previous paragraph, we get that $\BB{E}\left(  \log \frac{|f_t'(\zz)|}{|f_t(\zz)|} \right) + C \geq 0$. By Lemma~\ref{lem-integral-ii}, this implies that $C - \tz \ref{item-II} \geq 0$. Since $C$ does not depend on $\tz$, making $\tz$ large implies that $~\ref{item-II} \leq 0$, as required. 
\medskip

\noindent\textit{Step 3: many applications of the Koebe quarter theorem.}
It remains to find an event $\mcl E_{\sz}$ for which~\eqref{eqn-dist-3} holds. To construct such an event, we start with some basic calculations based on the the Koebe quarter theorem. Recall that this theorem implies that if $K_1,K_2$ are compact and $\phi : \BB C\setminus K_1 \rta \BB C\setminus K_2$ is a conformal map, then 
\eqbn
|\phi'(z)| \op{dist}(z,K_1) \asymp \op{dist}(\phi(z),K_2) ,\quad\forall z\in \BB C\setminus K_1 ,
\eqen
with universal implicit constants; see, e.g.,~\cite[Corollary 3.18]{lawler-book}. 
  
By our definitions of $s$ and $t$, we have $s+t = \sigma(\frk s + \frk t)  <T_\zz$, so $\zz\notin L_{s+t}$ and hence $f_t(\zz) \notin R_t \cup L_{t,s+t}$. 
We can view $f_{t,s+t}$ as a mapping on either the domain $\BB{C} \backslash L_{t,s+t}$ or on the smaller domain $\BB{C} \backslash (L_{t,s+t} \cup R_{t})$.  By applying the Koebe quarter theorem to the mapping $f_{t,s+t}$ on these two domains and using the concatenation property~\eqref{eqn-conc}, we obtain
\eqb
|f_{t,s+t}'(f_t(\zz))| 
\asymp \frac{\dist(f_{t,s+t}(f_t(\zz)),  \BB C\setminus f_{t,s+t}(\BB C\setminus L_{t,s+t})  )}{\dist(f_t(\zz),L_{t,s+t})} 
= \frac{\dist(f_{s+t}(\zz),R_{t,s+t})}{\dist(f_t(\zz),L_{t,s+t})}  
\label{eqn-two-formulas}
\eqe
and
\eqb 
|f_{t,s+t}'(f_{t}(\zz))| 
\asymp \frac{\dist(f_{t,s+t}(f_t(\zz)),  \BB C\setminus f_{t,s+t}(\BB C\setminus [L_{t,s+t} \cup R_t] )  )}{\dist(f_t(\zz),L_{t,s+t} \cup R_t)} 
= \frac{\dist(f_{s+t}(\zz), R_{s+t})}{\dist(f_{t}(\zz),L_{t,s+t} \cup R_{t})},
\label{eqn-two-formulas'}
\eqe
with universal implicit constants. By combining~\eqref{eqn-two-formulas} and~\eqref{eqn-two-formulas'}, we get
\eqb \label{eqn-two-dist-ratio}
 \frac{\dist(f_{s+t}(\zz),R_{t,s+t})}{\dist(f_t(\zz),L_{t,s+t})} \asymp \frac{\dist(f_{s+t}(\zz), R_{s+t})}{\dist(f_{t}(\zz),L_{t,s+t} \cup R_{t})}  .
\eqe
 
By the concatenation property, $f_t|_{\BB C\setminus L_{s+t}}$ maps $\BB C\setminus L_{s+t}$ to $ \BB C\setminus (L_{t,s+t} \cup R_t)$. 
By the Koebe quarter theorem applied to this conformal map, followed by~\eqref{eqn-two-dist-ratio},
\eqb \label{eqn-koebe-t}
|f_t'(\zz)| \op{dist}(\zz,L_{s+t}) 
\asymp \op{dist}\left(f_t(\zz) , L_{t,s+t} \cup R_t \right) 
\asymp  \frac{\dist(f_t(\zz),L_{t,s+t})}{\dist(f_{s+t}(\zz),R_{t,s+t})} \dist(f_{s+t}(\zz), R_{s+t}) .
\eqe
Re-arranging~\eqref{eqn-koebe-t} gives
\eqb
\frac{\dist(f_t(\zz),L_{t,s+t})}{|f_t(\zz)| \dist(\zz,L_{s+t})} \asymp \frac{|f_t'(\zz)|}{|f_t(\zz)|} \frac{\dist(f_{s+t}(\zz),R_{t,s+t})}{\dist(f_{s+t}(\zz),R_{s+t})}.
\label{eqn-dist-1}
\eqe 
\medskip

\noindent\textit{Step 4: definition of $\mcl E_{\sz}$.}
Let
\eqb \label{eqn-dist-event}
\mcl{E}_{\sz}   := \left\{ \dist(f_t(\zz),L_{t,s+t}) \leq  \dist(f_t(\zz),R_t) \right\} .
\eqe 
Since $s + t \rta T_{\zz}$ as $\sz\rta\infty$ (Remark~\ref{remark-sigma}) and $t = \sigma(\tz)$ is fixed independently of $\sz$, the assumption $\lim_{u \rta  T_{\zz}^-} \dist(\zz,L_u) = 0 $ implies that $\BB 1_{\mcl E_{\sz}} \rta 1$. 

We will now check the bound~\eqref{eqn-dist-3}. We first note that by~\eqref{eqn-dist-event}, 
\eqb \label{eqn-dist-cases}
\op{dist}\left( f_t(\zz) , L_{t,s+t} \cup R_t \right)
= \op{dist}\left( f_t(\zz) , L_{t,s+t}  \right) \BB 1_{\mcl E_\sz} + \op{dist}\left( f_t(\zz) , R_t  \right) \BB 1_{\mcl E_\sz^c} .
\eqe
 
By the Koebe quarter theorem applied to $f_t$, as in~\eqref{eqn-koebe-t}, followed by~\eqref{eqn-dist-cases}, we get that on $\mcl E_\sz$, 
\eqb \label{eqn-koebe-union}
|f_t'(\zz)| 
\asymp \frac{\op{dist}\left( f_t(\zz) , L_{t,s+t} \cup R_t \right)}{\op{dist}(\zz , L_{s+t})}
= \frac{\op{dist}\left( f_t(\zz) , L_{t,s+t} \right)}{\op{dist}(\zz , L_{s+t})} .
\eqe
Dividing both sides of~\eqref{eqn-koebe-union} by $|f_t(\zz)|$ gives that on $\mcl E_\sz$, 
\eqb
 \frac{\dist(f_t(\zz),L_{t,s+t})}{|f_t(\zz)| \dist(\zz,L_{s+t})} \asymp \frac{|f_t'(\zz)|}{|f_t(\zz)|} . 
\label{eqn-dist-es}
\eqe 

We also need an upper bound for the left side of~\eqref{eqn-dist-es} on the complement of the event $\mcl{E}_{\sz}$. To get such a bound, we use~\eqref{eqn-two-dist-ratio}, followed by the fact that $0\in L_{t,s+t}$ and the formula~\eqref{eqn-dist-cases}, to get that on $\mcl E_\sz^c$, 
\eqb
\frac{\dist(f_{s+t}(\zz),R_{t,s+t})}{\dist(f_{s+t}(\zz),R_{s+t})} \asymp \frac{\dist(f_t(\zz),L_{t,s+t})}{\dist(f_t(\zz), L_{t,s+t} \cup R_t)} \leq \frac{|f_t(\zz)|}{\dist(f_t(\zz),R_t)}.
\label{eqn-dist-2}
\eqe
Combining~\eqref{eqn-dist-1} and~\eqref{eqn-dist-2}, we deduce that, on $\mcl{E}_{\sz}^c$, 
\eqb
\frac{\dist(f_t(\zz),L_{t,s+t})}{|f_t(\zz)| \dist(\zz,L_{s+t})} \lesssim \frac{|f_t'(\zz)|}{\dist(f_t(\zz),R_t)} \asymp \frac{1}{\dist(\zz,L_t)},
\label{eqn-dist-esc}
\eqe
where we obtain the last $\asymp$ relation by applying the Koebe quarter theorem.

Combining the bounds~\eqref{eqn-dist-es} and~\eqref{eqn-dist-esc}  yields~\eqref{eqn-dist-3}.
\end{proof}

We will now upgrade from a positive-probability statement for the random point $\zz$ to an a.s.\ statement for a general $z\in \BB C\setminus \{0\}$. 

\begin{proof}[Proof of Proposition~{\hyperref[item-ii-b]{\ref*{prop-ii}\ref*{item-ii-b}}}]
For $z \in \BB C\setminus \{0\}$, let 
\eqbn
q(z) = \BB P\left( \lim_{t \rta T_{z}^-} \dist(z,L_t) > 0 \right) .
\eqen
By scale invariance (Lemma~\ref{lem-scaling}), $q(z) = q( z/|z|) $. 

By Lemma~\ref{lem-converse-w-o-upgrade}, $\BB E[q(\zz)] > 0$. 
Since the law of $\arg\zz$ has a density with respect to Lebesgue measure, it follows that the set  
\eqb
\Lambda := \left\{z\in \BB C : |z| =1 , \: q(z)  > 0 \right\} 
\eqe
has positive one-dimensional Lebesgue measure.
We claim that in fact $\Lambda$ is the whole unit circle, which implies that $q(z) > 0$ for each $z\in\BB C\setminus \{0\}$. 
 
To see this, let $\Lambda'$ be a closed subset of $\Lambda$ with positive Lebesgue measure. For $z\in \BB C$ with $|z|=1$, let $\wh\theta_\tz(z) = \arg f_{\sigma_z(\tz)}(z)$ be the re-parametrized argument process, and let $\tau$ be the smallest $\tz > 0$ for which $e^{i\wh\theta_\tz(z)} \in \Lambda'$. 
From the SDE~\eqref{eqn-theta-t} for $\wh\theta_\tz(z) = \arg f_{\sigma_z(\tz)}(z)$ and Girsanov's theorem, we infer that $\BB P(\tau < \infty) > 0$. 

By scale invariance and the domain Markov property (Lemmas~\ref{lem-scaling} and~\ref{lem-markov}), on the event $\tau < \infty$, the conditional probability that $\lim_{t\rta T_z^-} \dist(z,L_t)  >  0$ given the Loewner evolution up to time $\sigma_z(\tau)$ is equal to $q( e^{i\wh\theta_\tau(z)})$. Since $e^{i\wh\theta_\tau(z)} \in \Lambda' \subset \Lambda$ and $\BB P(\tau < \infty) > 0$, we infer that $q(z) > 0$, i.e., $z\in\Lambda$. Thus, our claim is proven. 

By the Koebe distortion theorem, if $\lim_{t\rta T_z^-} \dist(z,L_t) \geq r > 0$, then for each $w \in B_{r/16}(z)$, we have $\lim_{t\rta T_z^-} \dist(f_t(w),R_t) = 0$. Therefore, $(L_t)_{t\geq0 }$ absorbs the ball $B_{r/16}(z)$ at time $T_z$, i.e., $\lim_{t\rta T_w^-} \dist(w,L_t) \geq 15 r/16 > 0$ for each $w \in B_{r/16}(z)$.  
Consequently, there exists a deterministic $\ep > 0$ such that $q(w) \geq q(z)/2$ for each $w \in B_\ep(z)$. 
By compactness, we can cover $\{|z|=1\}$ by finitely many balls of this form. Hence, there exists $q_* > 0$ such that $q(z) \geq q_*$ for each $z$ in the unit circle, and hence (by scale invariance) for each $z\in\BB C$. 

By strong Markov property (Lemma~\ref{lem-markov}), it follows that for each $z\in\BB C\setminus \{0\}$, 
\eqbn
 \BB{P}\left(\lim_{t \rta T_z^-} \dist(z,L_t) > 0 \,|\, \mcl F_{\sigma_z(\tz)} \right) \geq q_* ,\quad \forall \tz \geq 0 . 
 \eqen
By the martingale convergence theorem, as $\tz \rta \infty$ the left-hand side of this expression converges to the indicator of the event that $\lim_{t \rta T_z^-} \dist(z,L_t) > 0$.  Hence, the indicator of the event is a.s.\ at least $q_*$, which means that the indicator must a.s.\ equal $1$. This proves that a.s.\ $\lim_{t \rta T_z^-} \dist(z,L_t) > 0$. 
\end{proof}

\subsection{Explicit phase formulas when the correlation is zero}
 
The description of the phases for a general covariance matrix $\Sigma$ in terms of~\ref{item-I} and~\ref{item-II} follows from Propositions~\ref{prop-i} and~\ref{prop-ii}.  
To prove Theorem~\ref{thm-phases}, it remains to deduce the explicit formulas for the phase boundaries when the correlation $c$ is equal to zero. We first prove the formulas except for the behavior on the boundary of the swallowing and hitting phases (which will be treated separately). 

\begin{lem} \label{lem-c0-phases}
Assume that the covariance matrix $\Sigma$ is such that $c=0$. 
Then $\SLEG$ is in the thin phase if $a-b\leq 4$, the swallowing phase if $4 < a-b < 8$, and the hitting phase if $a-b > 8$. 
\end{lem}
\begin{proof}
When $c=0$, the formulas for $\mu$ and $D$ in~\eqref{eqn-nu-D} simplify, and we can explicitly solve~\eqref{eqn-p-ode} to get that the density of the stationary distribution for $\wh\theta_{\frk t}$ is given by
\eqb \label{eqn-c0-p}
p(u) = \begin{cases}
Z^{-1} \left( a + b + (b-a) \cos(2u) \right)^{   \frac{4}{b-a}   } ,\quad & a\not= b  \\
Z^{-1} \exp\left(  \frac{4}{a} \cos^2(u) \right) ,\quad & a = b  
\end{cases}
\eqe 
where $Z = Z(a,b) > 0$ is a normalizing constant. We remark that in the proof of Lemma~\ref{lem-p}, we have $r_* = 0$ and hence $p =p_0$ when $c = 0$. 

For $c=0$, we get from~\eqref{eqn-nu-D} that $\nu(u) = \left(2-\frac{a}{2} + \frac{b}{2}\right) \cos{2u}$, so (in the notation of Theorem~\ref{thm-phases}),
\eqb \label{eqn-c0-integrals}
 \ref{item-I}  =  \left(2-\frac{a}{2} + \frac{b}{2}\right)  \int_0^{2\pi}  \cos(2u)  p(u) \, du \quad \text{and} \quad
 \ref{item-II}  = \left(4-\frac{a}{2} + \frac{b}{2}\right)  \int_0^{2\pi}  \cos(2u) p(u) \, du  .
\eqe
We claim that 
\eqb \label{eqn-c0-phase-show}
\int_0^{2\pi}  \cos(2u)  p(u) \, du > 0
\eqe
 for any choice of $a,b > 0$. Once~\eqref{eqn-c0-phase-show} is established, we get from~\eqref{eqn-c0-integrals} that $\ref{item-I} \geq 0$ if and only if $a-b \leq 4$ and $\ref{item-II} \geq 0$ if and only if $a-b\leq 8$. Our desired description of the phases for $c=0$ then follows from our description of the phases in terms of the integrals~\ref{item-I} and~\ref{item-II} (Propositions~\ref{prop-i} and~\ref{prop-ii}).

The function $p$ of~\eqref{eqn-c0-p} is $\pi$-periodic, so to prove~\eqref{eqn-c0-phase-show} we only need to show that $\int_0^\pi \cos(2u) p(u)\,du > 0$. 
Furthermore, we have $\cos(2(\pi - u)) p(\pi-u) = \cos(2u) p(u)$, so it suffices to show that 
\eqb \label{eqn-c0-phase-show'}
\int_0^{\pi/2} \cos(2u)p(u) \,du > 0 . 
\eqe
 
Note that $\cos(2u)$ is positive on $[0,\pi/4]$ and negative on $[\pi/4,\pi/2]$. To compare the integrals over these two intervals, we reflect across $\pi/4$ to get that  
\eqb \label{eqn-cos-reflect} 
\cos(2(\pi/2-u)) p(\pi/2-u) 
=  \begin{cases}
- Z^{-1}   \cos(2u) \left( a + b -  (b-a) \cos(2u) \right)^{   \frac{4}{b-a}   }   ,\quad &a\not= b \\
- Z^{-1} \cos(2u) \exp\left(  \frac{4}{a} (1- \cos^2(u) ) \right) ,\quad &a = b 
\end{cases}
\eqe
By considering separately the cases when $a < b$, $a>b$, and $a=b$ we see that then~\eqref{eqn-cos-reflect} implies that 
\eqbn
|\cos(2(\pi/2-u)) p(\pi/2-u)| < \cos(2u) p(u) ,\quad\forall u \in [0,\pi/4). 
\eqen
This shows that
\eqbn
\int_0^{\pi/2} \cos(2u) p(u) \,du =\int_0^{\pi/4} \cos(2u) p(u) \,du - \int_{\pi/4}^{\pi/2} | \cos(2u) p(u)| \,du > 0 ,
\eqen
which is~\eqref{eqn-c0-phase-show'}. 
\end{proof}

It remains only to treat the boundary case when $~\ref{item-II} = 0$. 

\begin{lem} \label{lem-c0-bdy}
Suppose that our covariance matrix $\Sigma$ is such that $c=0$ and $a-b=8$. 
Then $\SLEG$ is in the hitting phase.
\end{lem}
\begin{proof}
Let $z\in\BB C\setminus\{0\}$, let $\sigma_z$ be the time change of Definition~\ref{defn-sigma}, and denote time changed processes by a hat. 
From the definition of $\nu$ in~\eqref{eqn-nu-D}, when $c=0$ and $a-b=8$, we have $\nu(\cdot) + 2\cos(2\cdot) \equiv 0$. 
Therefore, the SDE~\eqref{eqn-sde-ratio} (or just Lemma~\ref{lem-polar}) gives
\eqb
d\log\frac{|\wh f_{\tz}' (z)|}{|\wh f_{\tz} (z)|} =  -  \sqrt{a}\cos{\wh\theta_\tz}\, dW_\tz - \sqrt{b}\sin{\wh\theta_\tz}\,  d\wt W_\tz , 
\eqe
where $(W,\wt W) \eqD (B,\wt B)$ are independent standard linear Brownian motions. Hence, $  \log\frac{|\wh f_{\tz}' (z)|}{|\wh f_{\tz} (z)|} $ is a continuous martingale with quadratic variation $\int_0^\tz [ a \cos^2(\wh\theta_\sz)  + b \sin^2(\wh\theta_\sz) ]\,d\sz$. This quadratic variation tends to $\infty$ as $\tz \rta\infty$, so $\log\frac{|\wh f_{\tz}' (z)|}{|\wh f_{\tz} (z)|}$ is a re-parametrized Brownian motion under a time change which covers the whole interval $[0,\infty)$. In particular, a.s.\ 
\eqb
\limsup_{\tz\rta\infty}\log\frac{|\wh f_{\tz}' (z)|}{|\wh f_{\tz} (z)|} =  \infty ,\quad \text{i.e.} \quad 
\limsup_{t\rta T_z^-} \frac{|f_t'(z)|}{|f_t(z)|} = \infty. 
\eqe
The proof of Proposition~{\hyperref[item-ii-b]{\ref*{prop-ii}\ref*{item-ii-a}}} (see in particular~\eqref{eqn-hit-koebe}) then implies that $\SLEG$ is in the hitting phase.  
\end{proof}

\begin{proof}[Proof of Theorem~\ref{thm-phases}]
Combine Propositions~\ref{prop-i} and~\ref{prop-ii} with Lemmas~\ref{lem-c0-phases} and~\ref{lem-c0-bdy}. 
\end{proof}

\section{$\SLEG$ disconnects points}
\label{sec-disconnects}

In this section we prove Theorem~\ref{thm-disconnects}---namely that for each $z \in \BB{C}$, a.s.\ the left hull disconnects $z$ from infinity before absorbing $z$, and that the complements of the left and right hulls at each fixed time are not connected. 
Since replacing the covariance $c$ by $-c$ has the effect of conjugating the hulls (see the discussion just after~\eqref{eqn-def-ft}), we can assume without loss of generality that $c\geq 0$. This assumption is used in the proof of Lemma~\ref{lem-c-event} below.

Throughout the section, we use the following notation: for real numbers $x<y$ and $x'<y'$, we define  
\eqbn
[x,y] \times [x',y'] = \{z \in \BB{C} : \Re(z) \in [x,y], \Im(z) \in [x',y']\} \quad \text{and} \quad [x,y] = [x,y] \times \{0\} .
\eqen
To prove Theorem~\ref{thm-disconnects}, we will first show that the left hull disconnects \emph{some} fixed open subset of the plane on an event with positive probability.

\begin{lem}
\label{lem-disconnect-2}
There exists a deterministic bounded open set $U \subset \BB{C}$ and $t>0$ such that, with positive probability, $U\cap L_t = \emptyset$ and $L_t$ disconnects $U$ from infinity.
\end{lem}

We devote most of this section to proving Lemma~\ref{lem-disconnect-2}; we then deduce Theorem~\ref{thm-disconnects} from Lemma~\ref{lem-disconnect-2} at the end of this section. 
We now roughly outline our strategy for proving Lemma~\ref{lem-disconnect-2}.
We construct a positive-probability event that 
imposes enough constraints that the following two  properties hold for some large (fixed) times $T,T'>0$ and positive real numbers $x<y$. 
\begin{enumerate}[label={(\alph*)}]
\item
\label{item-behavior-a}
The right hull $R_T$ is contained in a small neighborhood of the real axis, and some connected subset  $S \subset R_T \setminus \{0\}$  crosses from a small neighborhood of $x$ to a small neighborhood of $y$.
\item
\label{item-behavior-b}
The hull $L_{T,T+T'}$ avoids a (slightly larger) small neighborhood of $x$ and intersects $S$ (recall the definition of $L_{T,T+T'}$ from~\eqref{eqn-conc}).
\end{enumerate}
Properties~\ref{item-behavior-a} and~\ref{item-behavior-b} together imply that $R_T \cup L_{T,T+T'}$ disconnects some small set near $x$ from infinity; see Figure~\ref{fig-disconnects-idea}. Due to the concatenation property~\eqref{eqn-conc}, this gives Lemma~\ref{lem-disconnect-2}. 

\begin{figure}[ht!]
\centering
\label{fig-disconnects-idea}
    \includegraphics[width=0.6\linewidth]{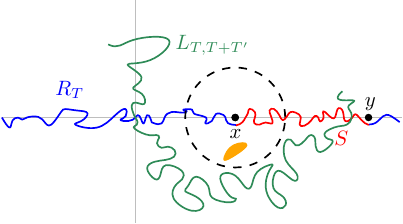}
\caption{An illustration of the rough sketch of the proof of Lemma~\ref{lem-disconnect-2}. To make the figure easier to parse, we show the left and right hulls as simple curves, but, as we show in this subsection, in actuality the complements of these hulls are typically not connected. Properties~\ref{item-behavior-a} and~\ref{item-behavior-b} imply that the union of $R_T$ (blue and red) and $L_{T,T+T'}$ (green) disconnects a set (orange) in a small neighborhood of $x$ (dashed) from infinity.}
\end{figure}

To show that the properties~\ref{item-behavior-a} and~\ref{item-behavior-b} hold with positive probability, we prove a pair of lemmas (Lemmas~\ref{lem-line-approx} and~\ref{lem-pos-hits-before-neg}).
Both lemmas are elementary consequences of the fact that we can make our driving Brownian motion have a given behavior with positive probability.
Let us begin with property~\ref{item-behavior-a}.

\begin{lem} \label{lem-line-approx}
Let $T >0$ and $\ep > 0$. It holds with positive probability that the following distances are each at most $\ep$: the Hausdorff distance between $L_T$ and the line segment $[-2T^{-1/2} i , 2T^{1/2} i]$, the Hausdorff distance between $R_T$ and the line segment $[-2T^{1/2} , 2T^{1/2}]$, and the uniform distance between $f_T$ and the map $z\mapsto \sqrt{z^2+4T}$ on the complement of the $\ep$-neighborhood of $ [-2T^{-1/2} i , 2T^{1/2} i] $.  
\end{lem} 

To prove Lemma~\ref{lem-line-approx}, we show that we can approximate a general forward Loewner chain (and its left and right hulls) by approximating its driving function.

\begin{lem}
\label{lem-haus}
Let $U_t^n,U_t$ be continuous functions from $[0,\infty)$ into $\BB{C}$, let $g_t^n,g_t$ be the corresponding forward Loewner maps, and let $L_t^n,L_t$ the corresponding left hulls.  If $U_t^n$ converges to $U_t$ uniformly on compact sub-intervals of $[0,\infty)$, then the following holds for each $T >0$ and each $\ep>0$:
\begin{itemize}
\item
The convergence $g_t^n(z) \rta g_t(z)$ holds uniformly in $(t,z) \in [0,T] \times \{z: \dist(z,K_T) \geq \ep\}$.
\item
For each $t$, we can choose $n$ large enough so that 
$L_t^n \subset B_\ep(L_t)$,  and every point disconnected from infinity by $L_t$ is disconnected from infinity by $B_\ep(L_t^n)$.\footnote{As in Definition~\ref{def-disconnect}, we say that a point $z$ is disconnected from infinity by a bounded set $S \subset \BB{C}$ if $z$ is not contained in the unbounded connected component of $\BB{C} \backslash S$. In particular, points in $S$ itself are disconnected from infinity by $S$.}
\end{itemize}
The result also holds with forward Loewner maps replaced by reverse Loewner maps (as defined after Lemma~\ref{lem-forward-reverse-ext}), or with left hulls replaced by right hulls.
\end{lem}
\begin{proof}
The convergence of the forward Loewner maps $g_t^n$ is proved in~\cite[Proposition 4.43]{lawler-book}.  (The result we have cited was stated for chordal Loewner chains with real driving functions, but the proof extends to our more general setting by simply replacing each instance of $\BB{H}$ in the proof by $\BB{C}$.)  The same proof applies to the reverse Loewner maps.

We now prove the second part of the lemma.  It suffices to consider left hulls, since the result for right hulls then follows from considering the left hulls of the reverse Loewner maps (see Lemma~\ref{lem-forward-reverse-ext}). 
Fix $t>0$.
By compactness, for every  subsequence of natural numbers, we can choose a further subsequence along which $L_t^n$ converges in the Hausdorff distance to some set $L'$.  To prove the second part of the lemma, it suffices to show that $L' \subset L_t$  and that  $L'$ and $L_t$ disconnect the same set of points from infinity.  

For each $\ep >0$, along the chosen subsequence, the maps $g_t^n$ converge to $g_t$ on $\BB{C} \backslash B_\ep(L_t)$, and therefore $L' \subset B_\ep(L_t)$. Thus, $L'  \subset L_t$.
Now, suppose for contradiction that the set of points disconnected from infinity by $L_t$ is not the same as the set of points disconnected from infinity by $L'$. Then there is a bounded open set $V$ which intersects the boundary of the unbounded connected component of $\BB C\setminus L_t$ but does not intersect $L'$, and such that $L'$ does not disconnect $V$ from $\infty$. 
Restricting to $n$ in the chosen subsequence, we note that $g_t^n$ is conformal on $V$ for large enough $n$, and that the $g_t^n$ are uniformly bounded on $V$ since $g_t^n \rightarrow g_t$ on some set which disconnects $V$ from infinity. It follows from the Cauchy integral formula that, after passing to a further subsequence, we can arrange that $g_t^n|_V$ converges uniformly on compact subsets of $V$ to some function $f$ on $V$ that is either conformal or constant. Now, since $V$ is not disconnected from infinity by $L'$, if $f$ were constant then by the uniqueness of analytic continuation the limit of the maps $g_t^n$ along the subsequence would have to be constant. Thus, $f$ is conformal, and hence the limit of the maps $g_t^n$ extends conformally to $V$.  Since the maps $g_t^n$ converge to $g_t$ on the complement of $L_t$, this implies that $g_t$ extends conformally to $V$, which cannot hold since $V$ intersects $L_t$.
\end{proof}

\begin{proof}[Proof of Lemma~\ref{lem-line-approx}]
The Loewner chain with driving function $0$ is given by the mappings $z \mapsto \sqrt{z^2 + 4t}$, and it has left hulls $L_t = [-2t^{1/2} i , 2 t^{1/2} i]$ and right hulls $R_t = [-2t^{1/2}, 2 t^{1/2}]$.  The driving Brownian motion of the $\SLEG$ chain has a positive chance to be arbitrarily close to $0$ on $[0,T]$.
By Lemma~\ref{lem-haus} (applied with $U \equiv 0$), it follows that with positive probability, $L_T \subset B_\ep(   [-2T^{1/2} i , 2 T^{1/2} i]  )$, and every point on $[-2T^{1/2} i , 2 T^{1/2} i]$ is disconnected from infinity by $B_\ep(L_T)$; and the analogous statement also holds for right hulls. This implies the desired bound on Hausdorff distances.
\end{proof}

We have now shown how to obtain property~\ref{item-behavior-a} in the outline above.  To obtain property~\ref{item-behavior-b}, roughly speaking, we impose the following pair of conditions.
\begin{itemize}
\item
First, we ``force'' $f_{T,T+2}(S)$ to cross the imaginary axis and $L_{T,T+2}$ to avoid a small neighborhood of $x$.
\item
Second, we let $K>0$ be large and we ``force'' $L_{T+2,T+2+K}$ to intersect $f_{T,T+2}(S)$ while staying close to the imaginary axis (the latter to ensure that $L_{T,T+2+K}$ avoids a small neighborhood of $x$).
\end{itemize}
These two conditions together yield property~\ref{item-behavior-b} with $T' = 2+K$.  We can impose the second condition by Lemma~\ref{lem-line-approx}.  To show that the first condition also holds with positive probability, we prove the following lemma.

\begin{lem}
\label{lem-pos-hits-before-neg} Assume that the parameters from~\eqref{eqn-cov-matrix} satisfy $a,b>0$. 
We can choose $0<x<y$ such that, on an event with positive probability, the left hull $L_2$ is disjoint from $[x,y]$ and $\Re f_2(x) < 0 < \Re f_2(y)$.
\end{lem}

To prove Lemma~\ref{lem-pos-hits-before-neg}, we construct a positive-probability event on which $|f_t(z)| > 1$ for all $t \in [0,2]$ and $z \in [x,y]$, and $\Re f_2(x) < 0 < \Re f_2(y)$.  Note that the first property implies that $[x,y]$ is disjoint from $L_2$ by Remark~\ref{remark-sigma}.  Before constructing the event, we derive simple inequalities which the real and imaginary parts of $f_t(z)$ satisfy as long as $|f_t(z)| > 1$.

\begin{lem}
Recall the pair of correlated standard linear Brownian motions from~\eqref{eqn-def-gt}. 
Let $z \in \BB{C}$ and 
\eqbn
\rho = \rho_z:= \inf\{t:|f_t(z)| \leq 1\} . 
\eqen
Then, for each $s,t \in [0,\rho]$, we have 
\eqb
\label{eqn-re-ft}
h^-(s,t;a) \leq \Re f_t(z) - \Re f_s(z) \leq h^+(s,t;a)
\eqe 
and
\eqb 
\label{eqn-im-ft}
h^-(s,t;b) \leq \Im f_t(z)  - \Im f_s(z) \leq h^+(s,t;b).
\eqe
where
\[
h^\pm(s,t;a) = \pm 2(t-s) - \sqrt{a} (B_t-B_s) \qquad \text{and} \qquad h^\pm(s,t;b) = \pm 2(t-s) - \sqrt{b} (\wt B_t- \wt B_s)
\]
\end{lem}

\begin{proof}
By~\eqref{eqn-def-ft}, the real and imaginary parts of $f_t(z)$ satisfy the pair of SDEs
\eqb
\label{eqn-re-im-ft}
d\Re f_t(z) = \frac{2 \Re f_t(z)}{|f_t(z)|^2} \, dt - \sqrt{a} \, dB_t, \qquad  \qquad
d\Im f_t(z) = -\frac{2 \Im f_t(z)}{|f_t(z)|^2} \, dt - \sqrt{b} \, d\wt B_t
\eqe
The lemma follows from applying these SDEs to the range $[s,t]$ and using the fact that $|f_r(z)| \geq 1$ for $r \in [s,t]$ to bound the drift terms.
\end{proof}

We now define the positive-probability event we use to prove Lemma~\ref{lem-pos-hits-before-neg}.

\begin{lem}
For all sufficiently large (deterministic) $C>0$, it holds with positive probability that 
\eqb
\label{cond-2}
-C \leq h^-(0,t;a) \leq h^+(0,t;a) \leq C,
\qquad \text{for $t \in [0,1]$,}
\eqe
\eqb
\label{cond-3}
-5C \leq h^-(1,2;a) \leq h^+(1,2;a) \leq -4C,
\eqe
\eqb
\label{cond-1}
h^+(0,1;b) \leq -2,
\eqe
and 
\eqb
\label{cond-4}
h^+(1,t;b) \leq 1 \qquad \forall  t \in [1,2] . 
\eqe
\label{lem-c-event}
\end{lem}

\begin{proof}
We can express~\eqref{cond-1} as $\sqrt{b} \wt B_1 \geq 4$ and~\eqref{cond-2} as $\sqrt{a} |B_t| \leq -2t+C$ for $t \in [0,1]$.  The event $\mcl E_1$ that these two conditions hold has positive probability as long as $C$ is large enough.  Next, 
we can express~\eqref{cond-3} as $2 + 4C \leq \sqrt{a} (B_2-B_1) \leq 5C - 2$ and~\eqref{cond-4} as $\sqrt{b} (\wt B_t- \wt B_1) \geq 2t-3$ for $t \in [1,2]$. Since the correlation $c/\sqrt{ab}$ of $B$ and $\wt B$ is assumed to be non-negative (recall the discussion at the beginning of this section), the event $\mcl E_2$ that these two conditions hold has positive probability as long as $C$ is large enough. Since $\mcl E_1$ and $\mcl E_2$ are independent, we deduce that all four conditions hold simultaneously with positive probability as long as $C$ is large enough.
\end{proof}

\begin{figure}[ht!]
\centering
    \includegraphics[width=.8\linewidth]{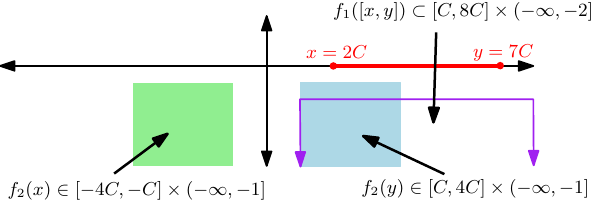}
\caption{An illustration of the proof of Lemma~\ref{lem-pos-hits-before-neg}. On the event of Lemma~\ref{lem-c-event}, the map $f_1$ takes $[x,y]$ into the purple-bordered region. Then, the map $f_{1,2}$ takes $f_1(x)$ into the light green region and $f_1(y)$ into the light blue region. This forces $f_2([x,y])$ to cross the imaginary axis.
    }
     \label{fig-c-event}
\end{figure} 

\begin{proof}[Proof of Lemma~\ref{lem-pos-hits-before-neg}]
Let $C \geq 1$ be chosen so that the event of Lemma~\ref{lem-c-event} has positive probability.
Throughout the proof, we work on this event. See Figure~\ref{fig-c-event} for an illustration.

Let $x=2C$ and $y = 7C$.  Let $z \in [x,y]$.
By~\eqref{cond-2} and~\eqref{eqn-re-ft} with $s=0$ and $t \in [0,1]$, we have $\Re f_t(z) - z \geq -C$ for each $t \in [0,1\wedge \rho_z]$, where $\rho_z = \inf\{t : |f_t(z)| \leq 1\}$. Therefore, $\Re f_t(z) \geq C$ for each $t \in [0,1\wedge \rho_z]$. 
Since $C\geq 1$, this implies that $\rho_z \geq 1$ and hence
\eqbn
\Re f_t(z) \geq C ,\quad\forall t \in [0,1] .
\eqen
Hence, we may apply~\eqref{eqn-im-ft} (with $s=0$ and $t=1$) and~\eqref{cond-1} to deduce that $\Im f_1(z) = \Im f_1(z) - \Im z \leq -2$.
By plugging~\eqref{cond-4} into~\eqref{eqn-im-ft} with $s=1$ and $t \in [1,2]$, we get that $\Im f_t(z) - \Im f_1(z) \leq 1$ for $t \in [1,2]$.  
Combining this with the bound $\Im f_1(z) \leq -2$, we deduce that $\Im f_t(z) \leq -1$ for all $t \in [1,2]$. In particular, $\rho_z \geq 2$.   

The fact that $\rho_z \geq 2$ for each $z \in [x,y]$ together with Lemma~\ref{lem-endpt-cont} shows that the segment $[x,y]$ is disjoint from $L_2$. 

Since $\rho_z \geq 2$, by plugging~\eqref{cond-3} into~\eqref{eqn-re-ft} with $s=1$ and $t=2$, we deduce that $-5C \leq \Re f_2(z) - \Re f_1(z) \leq -4C$ for each $z\in [x,y]$.  
Moreover, by combining~\eqref{cond-2} and~\eqref{eqn-re-ft} with $s=0$ and $t=1$, we get $-C \leq \Re f_1(z) - \Re z \leq C$.  
Combining the last two sets of inequalities yields $-6C \leq \Re f_2(z) - \Re z \leq -3C$.  In particular, $\Re f_1(x) \leq 2C - 3C = -C$ and  $\Re f_1(y) \geq 7C - 6C = C$.
\end{proof}

\begin{figure}[ht!]
\centering
    \includegraphics[width=\linewidth]{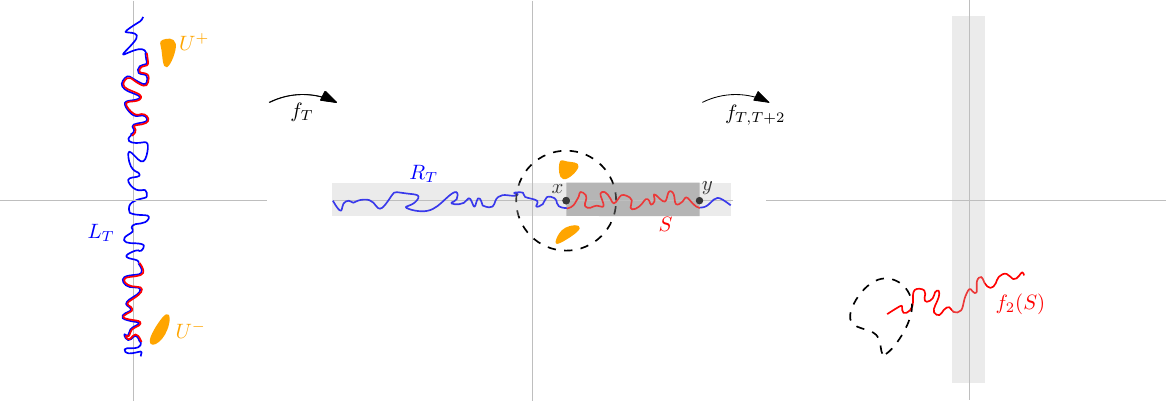}
\caption{An illustration of the proof of Lemma~\ref{lem-disconnect-2} (continued in Figure~\ref{fig-disconnects-part-2}). As in Figure~\ref{fig-disconnects-idea}, we show the left and right hulls as simple curves for clarity, but, as we prove in this section, in actuality they are not. By conditioning on a sequence of events $\mcl E_1,\ldots,\mcl E_4$ that occur simultaneously with positive probability, we enforce the following behavior of the Loewner maps for some small $\ep>0$ and large $T,K>0$.
\textbf{Left:} The left hull $L_T$ avoids a deterministic pair of sets $U^\pm$ (orange). \textbf{Middle:} The map $f_T$ maps the sets $U^\pm$ to a pair of sets (orange) contained in the ball $B_{3\ep}(x)$ (dashed).  The right hull $R_T$ (blue and red)  is contained  in the $\ep$-neighborhood of the real axis (light gray).  Moreover, we can choose a subset $S$ of $R_T$ (red) that crosses the rectangle $[x,y] \times [-\ep,\ep]$ (dark gray). \textbf{Right:} The image of $S$  under $f_{T,T+2}$ (red) crosses the rectangle $[-\ep,\ep] \times [-2K^{1/2}+\ep,2 K^{1/2}-\ep]$ (light gray).  Moreover, the image of $B_{3\ep}(x)$ under $f_{T,T+2}$ (dashed) does not intersect this rectangle.  
    }
     \label{fig-disconnects-part-1}
\end{figure}

\begin{figure}[ht!]
\centering
    \includegraphics[width=\linewidth]{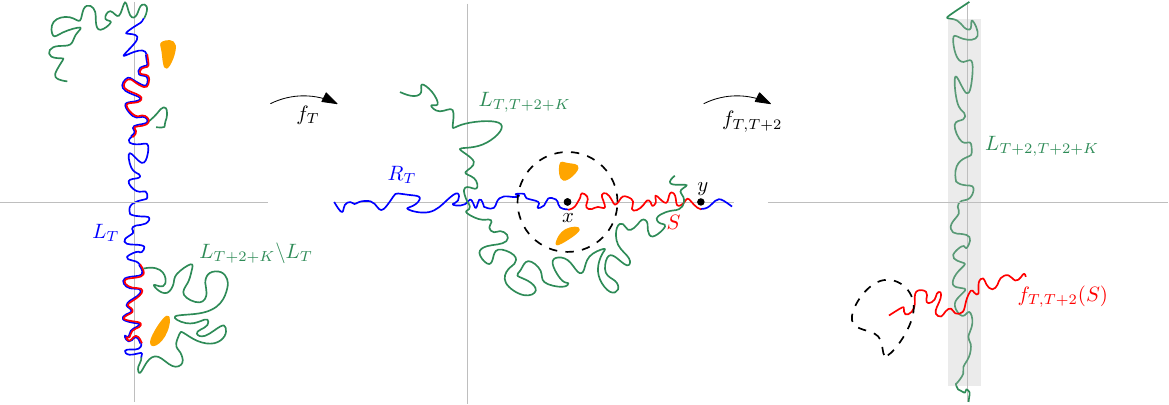}
\caption{An illustration of the proof of Lemma~\ref{lem-disconnect-2} (continued from Figure~{fig-disconnects-part-1}). \textbf{Right:} The hull $L_{T+2,T+2+K}$ (green) is contained in the $\ep$-neighborhood of the imaginary axis, and crosses the rectangle $[-\ep,\ep] \times [-2K^{1/2}+\ep,2 K^{1/2}-\ep]$ (light gray).  Therefore, $L_{T+2,T+2+K}$ intersects $f_{T,T+2}(S)$ (red) but does not intersect $f_{T,T+2}(B_{3\ep}(x))$ (dashed).  \textbf{Middle:} This means that $L_{T,T+2+K}$ (green) intersects $S$ (red) but not either of the sets $f_T(U^\pm)$.  Therefore, the union of $L_{T,T+2+K}$ and $R_T$ disconnects one of the sets $f_T(U^\pm)$ (orange) from infinity.  \textbf{Left:} We conclude that $L_{T+2+K}$ (blue and green) disconnects one of the sets $U^\pm$ (orange) from infinity.}
     \label{fig-disconnects-part-2}
\end{figure} 

We are now ready to prove Lemma~\ref{lem-disconnect-2}.  See Figures~\ref{fig-disconnects-part-1} and~\ref{fig-disconnects-part-2} for an illustration of the proof.

\begin{proof}[Proof of Lemma~\ref{lem-disconnect-2}]
Let $0<x<y$ be as in Lemma~\ref{lem-pos-hits-before-neg}. 
For a rectangle $R\subset \BB C$ with sides parallel to the coordinate axes, we say that a connected set $S\subset R$ is a \emph{left-right crossing} of $R$ if $S$ intersects both the left and right sides of $R$.

Fix parameters $T>y^2/2$ (so that $y < 2T^{1/2}$), $\ep > 0$, and $K>0$.  
Also let $U^-$ (resp.\ $U^+$) be a deterministic open subset of $\BB C$ which lies at distance at least $2\ep$ from the imaginary axis and which has the following property: the time $T$ Loewner map $z\mapsto \sqrt{z^2 + 4T}$ with driving function 0 maps $U^-$ (resp.\ $U^+$) into $B_{\ep/2}(x-2i\ep)$ (resp.\ $B_{\ep/2}(x+2i\ep)$).
We will argue that we can ``force'' the left hull $L_{T+2+K}$ to disconnect either $U^-$ or $U^+$ from infinity by conditioning on the following four events:
\begin{itemize}
\item
The event $\mcl E_1$ that $L_T$ is contained in the $\ep$-neighborhood of the imaginary axis, $R_T$ is contained in the $\ep$-neighborhood of the real axis, $f_T$ maps $U^\pm$ into $B_\ep(x \pm 2i\ep)$, and $R_T$ contains a left-right crossing of $[x,y]\times[-\ep,\ep]$. 
\item
The event $\mcl E_2$ that $L_{T,T+2}$ is disjoint from $[x,y]\times[-\ep,\ep]$ and, for any left-right crossing of the rectangle $[x,y] \times [-\ep,\ep]$, its image under $f_{T,T+2}$ contains a left-right crossing of $[-\ep,\ep] \times [-2K^{1/2}+\ep,2 K^{1/2}-\ep]$.
\item
The event $\mcl E_3$ that $f_{T,T+2}(B_{3\ep}(x))$ is at distance $>\ep$ from the imaginary axis.
\item
The event $\mcl E_4$ that $L_{T+2,T+2+K}$ is contained in the $\ep$-neighborhood of the imaginary axis, and hits every left-right crossing of the rectangle  $[-\ep,\ep] \times [-2K^{1/2}+\ep,2 K^{1/2}-\ep]$.
\end{itemize}
To prove the lemma, we must prove that the four events occur simultaneously with positive probability for some $T,\ep,K>0$, and that the events together imply that the left hull $L_{T+2+K}$ disconnects one of the sets $U^\pm$ from $\infty$.

The events $\mcl E_1,\mcl E_2 \cap \mcl E_3,$ and $\mcl E_4$ are measurable with respect to $(B_t)_{t \in [0,T]}$, $(B_t-B_T)_{t \in [T,T+2]}$ and $(B_t-B_{T+2})_{t \in [T+2,T+2+K]}$, respectively. Hence, the three events are independent. 
By Lemma~\ref{lem-haus} (applied once at time $T$ and once with $K$ in place of $T$), we have $\BB{P}(\mcl E_1), \BB{P}(\mcl E_4) > 0$ for each $T,\ep , K>0$.
By Lemma~\ref{lem-pos-hits-before-neg} (applied to $(f_{T,T+t})_{t\geq 0}$ instead of $(f_t)_{t\geq 0}$), we have $\BB{P}(\mcl E_2 \cap \mcl E_3) > 0$ for some $\ep,K>0$ and any $T>0$. Hence $\BB P(\mcl E_1\cap\mcl E_2\cap \mcl E_3\cap \mcl E_4) > 0$.

Now assume all four events hold for some choice of $\ep,K,T>0$.
\begin{itemize}
\item
Since $\mcl E_1$ holds, we can choose a connected subset $S$ of $R_T$ that avoids $f_T(U^\pm)$ and crosses the rectangle $[x,y] \times [-\ep,\ep]$ from left to right.
\item
Since $\mcl E_2$ holds, $f_{T,T+2}(S)$ crosses the rectangle $[-\ep,\ep] \times [-2K^{1/2}+\ep,2 K^{1/2}-\ep]$ from left to right.
\item
Since $\mcl E_3$ holds, the set $f_{T+2}(B_{3\ep}(x))$ is at distance $> \ep$ from the imaginary axis.
\item
Since $\mcl E_4$ holds, the hull $L_{T+2,T+2+K}$ intersects $f_{T,T+2}(S)$ but not  $f_{T,T+2}(B_{3\ep}(x))$.
\end{itemize}
Putting this together, we deduce that $L_{T,T+2+K}$ intersects $S$ but not $B_{3\ep}(x)$. Also, since $\mcl E_1$ holds, $R_T$ is contained in the $\ep$-neighborhood of the real axis and $f_T(U^\pm) \subset B_\ep(x\pm 2 i\ep) \subset B_{3\ep}(x) \setminus R_T$.  By topological considerations, we conclude that $L_{T,T+2+K} \cup R_T$ disconnects one of the sets $f_T(U^\pm)$ from infinity without hitting this set. By the concatenation property~\eqref{eqn-conc}, this gives the lemma statement for $t = T+2+K$. 
\end{proof}

We now use Lemma~\ref{lem-disconnect-2} to deduce the first part of Theorem~\ref{thm-disconnects}.

\begin{proof}[Proof of Theorem~\ref{thm-disconnects}]
The proof of assertion~\ref{item-disconnect-z} very similar to the proof of Proposition~{\hyperref[item-ii-b]{\ref*{prop-ii}\ref*{item-ii-b}}} at the end of Section~\ref{sec-prop-ii}, so we will be brief. 
For $z \in \BB C\setminus \{0\}$, let 
\eqbn
q(z) = \BB P\left( \text{$(L_t)_{t\geq 0}$ disconnects $z$ from $\infty$ before time $T_z$} \right) .
\eqen
By scale invariance (Lemma~\ref{lem-scaling}), $q(z) = q( z/|z|) $. 
By Lemma~\ref{lem-disconnect-2}, the set of $z\in\BB C$ such that $|z|=1$ and $q(z) > 0$ has positive one-dimensional Lebesgue measure. 
Exactly as in the proof of Proposition~{\hyperref[item-ii-b]{\ref*{prop-ii}\ref*{item-ii-b}}}, we deduce from this that $q(z) > 0$ for all $z\in\BB C\setminus\{0\}$.

If $(L_t)_{t\geq 0}$ disconnects $z$ from $\infty$ before time $T_z$, then there is a random $r > 0$ such that for each $w\in B_r(z)$, $(L_t)_{t\geq 0}$ disconnects $w$ from $\infty$ before time $T_w$. Therefore, there is a deterministic $\ep > 0$ such that $q(w) \geq q(z)/2$ for each $w\in B_\ep(z)$. 
By a compactness argument, we infer that there exists $q_* > 0$ such that $q(z) \geq q_*$ for each $z\in\BB C\setminus \{0\}$. 
 
By scale invariance and the strong Markov property (Lemmas~\ref{lem-scaling} and~\ref{lem-markov}), this implies that for each $z\in\BB C$, 
\eqb \label{eqn-disconnect-mart}
\BB{P}\left(    \text{$(L_t)_{t\geq 0}$ disconnects $z$ from $\infty$ before time $T_z$}    \,|\, \mcl F_{\sigma_z(\tz)} \right) \geq q_*   ,\quad \forall \tz \geq 0  ,
\eqe 
where $\sigma_z $ is the time change as in Definition~\ref{defn-sigma}. 
By the martingale convergence theorem, the left side of~\eqref{eqn-disconnect-mart} converges as $\tz\rta\infty$ to the indicator of the event that $(L_t)_{t\geq 0}$ disconnects $z$ from $\infty$ before time $T_z$. Hence, this indicator is a.s.\ at least $q_*$, which means that the indicator must a.s.\ be equal to $1$.

It remains to prove assertion~\ref{item-disconnect-t}. 
By scale invariance (Lemma~\ref{lem-scaling}), the probability that the set of points which are disconnected from $\infty$ by $L_t$ has non-empty interior does not depend on $t$. 
By Lemma~\ref{lem-disconnect-2}, this probability is positive for every $t>0$. 
By Blumenthal's zero-one law, a.s.\ there are arbitrarily small times $t>0$ for the set of points which are disconnected from $\infty$ by $L_t$ has non-empty interior.
Since the left hulls are non-decreasing, if $L_t$ disconnects an open subset of $\BB C$ from $\infty$, then so does $L_s$ for every $s > t$ (note that by Definition~\ref{def-disconnect} each subset of $L_t$ is considered to be disconnected from $\infty$ by $L_t$).  
\end{proof}

\section{Open problems}
\label{sec-problems}

Since there has been very little work so far on Loewner evolution driven by complex Brownian motion, there are a huge number of open problems concerning these objects. Here, we highlight a few open problems which we find particularly interesting. 

Schramm-Loewner evolution is famous for its relationship to many classes of discrete models.  A central question emerging from our work is whether $\SLEG$ with $a,b \neq 0$ is related to any class of discrete models.

\begin{ques}
Is there a class of natural discrete models which converges to $\SLEG$ for some $\Sigma$ with $a,b \neq 0$?
\end{ques}

SLE also has several connections to the Gaussian free field (GFF; see, e.g.,~\cite{bp-lqg-notes,shef-gff} for more background). More precisely, there is a coupling of SLE with the GFF produced using the \emph{forward} Loewner evolution, where the SLE is interpreted as a ``flow line" ($\kappa < 4$), a ``level line" ($\kappa=4$), or a ``counterflow line" ($\kappa > 4$) of the GFF~\cite{ss-contour,dubedat-coupling,ig1,ig4}. There are also couplings of SLE with the GFF using the \emph{reverse} Loewner evolution, which are important in the theory of Liouville quantum gravity (LQG)~\cite{shef-zipper,wedges}. As noted in Section~\ref{sec-intro}, various objects associated with LQG can be defined for complex parameter values. It would be extremely interesting to extend any of the SLE/GFF coupling results to the case of complex parameter values. 
 

\begin{ques}
Is there any interesting coupling of $\SLEG$ with the Gaussian free field or with Liouville quantum gravity for any covariance matrix $\Sigma$ with $a,b \neq 0$?  
\end{ques}

There are a number of interesting results about the behavior SLE$_\kappa$ as $\kappa \rta 0$ or $\kappa \rta \infty$ (especially in terms of large deviations). Sometimes, these results can be used to prove non-trivial theorems in complex analysis. See~\cite{wang-ld-survey} for a survey of this work. It is natural to ask whether there are any generalizations of this work to the case of Loewner evolution driven by complex Brownian motion. 

\begin{ques}
Can anything interesting be said about the large deviations behavior of $\SLEG$ as $\Sigma \rta 0$ or $\Sigma\rta\infty$ along various curves in the space of $2\times 2$ positive semidefinite symmetric matrices? What about as $\Sigma$ approaches the set of matrices of the form $\left(\begin{array}{cc}
a & 0 \\ 
0 & 0
\end{array} \right)$ or $\left(\begin{array}{cc}
0 & 0 \\ 
0 & b
\end{array} \right)$, 
which correspond to forward and reverse SLE with real driving function, respectively?
\end{ques}

Another class of open problems concerns the geometric properties of the left and right $\SLEG$ hulls. Probabilists have discovered many beautiful properties of $\SLE$ hulls, and it is possible that some of these properties extend in some form to our setting. For instance, it is known that, for $\kappa>4$, the outer boundary of an $\SLE$ hull (i.e., the boundary of the unbounded complementary connected component) is an SLE$_{16/\kappa}$-type curve~\cite{zhan-duality1,zhan-duality2,dubedat-duality,ig1,ig4}. We can ask whether $\SLEG$ possesses a similar property.

\begin{ques}
What can be said about the outer boundaries of the left and right hulls of $\SLEG$? Are these outer boundaries Jordan curves? Are they related to $\SLEGt$ for a different covariance matrix $\wt \Sigma$? 
\end{ques}

It was shown in~\cite{lawler-werness-multipoint} that SLE$_\kappa$ hulls for $\kappa > 0$ are transient. One way of stating this condition is that a.s.\ for each $R > 0$, there exists $T = T(R) > 0$ such that $L_t \cap B_R(0) = L_T \cap B_R(0)$ for each $t\geq T$. 

\begin{ques} \label{ques-transient}
Are the $\SLEG$ left hulls transient in the above sense?
\end{ques} 

We note that Question~\ref{ques-transient} is closely related to the problem of showing that the dense phase (Definition~\ref{defn-phases}) is empty, see the discussion just after Definition~\ref{defn-phases}. 

There are numerous papers which compute the Hausdorff dimensions of various sets associated with SLE$_\kappa$ for $\kappa > 0$; see, e.g.,~\cite{beffara-dim,alberts-shef-bdy-dim,miller-wu-dim}. We can also investigate the Hausdorff dimensions of sets associated to $\SLEG$. 

\begin{ques}
Compute the Hausdorff dimensions of various sets associated with $\SLEG$, e.g., the following. 
\begin{itemize}
\item The left and right hulls.
\item The set of points which are hit by the left hull before being disconnected from $\infty$ (this set has Lebesgue measure zero by Theorem~\ref{thm-disconnects}).
\item In the swallowing phase (Definition~\ref{defn-phases}), the set of points $z\in\BB C$ for which $\lim_{t\rta T_z^-} \op{dist}(z,L_t) = 0$ (this set has Lebesgue measure zero by the definition of the swallowing phase). 
\item The outer boundaries of the left and right hulls. 
\end{itemize} 
\end{ques}

Finally, Theorem~\ref{thm-phases} characterizes the phases of $\SLEG$ in terms of the signs of a pair of definite integrals, but it is not obvious how to compute these integrals.  It would be interesting to give explicit formulas for the phase boundaries for $\SLEG$ when $c \neq 0$, e.g., in the case when $c = \sqrt{ab}$.

\begin{ques}
Give an explicit description of the phases of $\SLEG$; i.e., explicitly write down the range of parameter values for which the integrals~\ref{item-I} and~\ref{item-II} in Theorem~\ref{thm-phases} are positive.
\end{ques}



\appendix

\section{Kolmogorov forward equation for periodic SDEs}
 
We prove the following version of the Kolmogorov forward equation for the stationary distribution of an SDE with periodic coefficients, which is used in the proof of Lemma~\ref{lem-stationary}. 

\begin{lem} \label{lem-kolmogorov}
Consider an SDE of the form
\eqb \label{eqn-kolmogorov-sde}
dX_t = \alpha(X_t) \,dt +  \beta(X_t)  \,dB_t ,
\eqe 
where $B$ is a standard linear Brownian motion and $\alpha  , \beta : \BB R \rta [0,\infty)$ are smooth functions such that both $\alpha$ and $\beta^2$ are $\pi$-periodic. Assume that
\eqb \label{eqn-hormander-cond}
\text{$\forall u \in \BB R$, at least one of $\beta(u)$ or $\beta'(u)\alpha(u)$ is non-zero}.  
\eqe
Let $p$ be a continuous, non-negative, $\pi$-periodic function and suppose that there exists $x\in\BB R$ such that the following is true. 
On $\BB R \setminus (x+\pi \BB Z)$, the function $p$ is twice continuously differentiable and satisfies the ODE
\eqb \label{eqn-kolmogorov}
- \frac{d}{du} [ \alpha(u) p (u) ] + \frac12  \frac{d^2}{du^2} [ \beta^2(u) p (u)]  = 0,\quad \int_0^{\pi} p(u) \,du = 1 .
\eqe
Furthermore, the function $\beta^2 p$ is continuously differentiable on all of $\BB R$. 
If the starting point $X_0$ is sampled from the probability measure $p(u) \BB 1_{[0,\pi]}(u)\,du$, then the law of $(e^{2 i X_t})_{t\geq 0}$ is stationary.
\end{lem}

The purpose of the condition~\eqref{eqn-hormander-cond} is to ensure that the semigroup associated to the SDE~\eqref{eqn-kolmogorov-sde} maps into the space of smooth functions. This is a consequence of H\"ormander's criterion, as we now explain.

\begin{lem} \label{lem-hormander} 
Assume that we are in the setting of Lemma~\ref{lem-kolmogorov}. 
Then for each $u \in\BB R$ and each $t\geq 0$, the law of $X_t$ under $\BB P_u$ has a smooth density with respect to Lebesgue measure. 
Furthermore, for each bounded function $f : \BB R\rta \BB R$ and each $t\geq 0$ the function $u\mapsto \BB E_u( f(X_t))$ is smooth. 
\end{lem}
\begin{proof}
In Stratonovich form, the SDE~\eqref{eqn-kolmogorov-sde} takes the form
\eqb \label{eqn-stratanovich}
dX_t = \gamma(X_t) \,dt + \beta(X_t) \circ dB_t , \quad \text{where} \quad \gamma(u) =  \alpha(u) - \frac12 \beta'(u) \beta(u)  
\eqe
and $\int_0^t \beta(X_s) \circ dB_s$ denotes the Stratonovich integral. By the one-dimensional case of H\"ormander's criterion (see, e.g.,~\cite[Theorem 1.3]{hairer-hormander}), the smoothness conditions in the lemma statement are satisfied if for each $u \in \BB R$, the numbers
\eqbn
\beta(u)  \quad \text{and} \quad \beta'(u) \gamma(u) - \gamma'(u) \beta(u)   
\eqen
span the vector space $\BB R$. This is equivalent to the condition that at least one of $\beta(u)$ or $\beta'(u) \alpha(u)$ is non-zero. 
\end{proof}

\begin{proof}[Proof of Lemma~\ref{lem-kolmogorov}]
Since both $p$ and the coefficients of our SDE are $\pi$-periodic, the law of $(e^{2 i X_t})_{t\geq 0}$ is unchanged if we sample $X_0$ from $ p(u) \BB 1_{[x, x +  \pi]}(u)\,du$ instead of $p(u) \BB 1_{[0,    \pi]}(u)\,du$. 
For $u \in\BB R$, write $\BB E_u$ for the law of $X$ started from $X_0 = u$ and write $\BB E_p$ for the law of $X$ started from $X_0 \sim  p(u) \BB 1_{[x, x + \pi]}(u)\,du$.  
Let $f : \BB R\rta \BB R$ be a twice continuously differentiable, $\pi$-periodic function.  

By It\'o's formula,  
\eqbn
df(X_t) 
 = \left( f'(X_t) \alpha(X_t) + \frac12 f''(X_t) \beta^2(X_t)  \right) \,dt   + f'(X_t) \beta(X_t)  \,dB_t .
\eqen
Therefore,
\alb
\BB E_p\left( f(X_t)  \right)  
=   \int_0^t \int_x^{x+\pi} \BB E_u\left(  f'(X_s) \alpha(X_s) + \frac12 f''(X_s) \beta^2(X_s)    \right) p(u) \, du \,ds  
\ale
Differentiating both sides with respect to $t$, then evaluating at $t = 0$ gives
\eqb \label{eqn-periodic-deriv0}
\left.\frac{\partial}{\partial t}\right|_{t=0} \BB E_p\left( f(X_t) \right)  
=    \int_x^{x + \pi} \left(  f'(u) \alpha(u) + \frac12 f''(u) \beta^2(u)     \right) p(u) \, du.
\eqe 

We now want to integrate by parts on the right side of~\eqref{eqn-periodic-deriv0}: in particular, we will integrate the term $f'(u) \alpha(u) p(u)$ by parts once, and the term $ f''(u) \beta^2(u)  p(u)$ by parts twice.
We explain why the boundary terms vanish. 
The functions $f,\alpha, \beta^2$, and $p$ are all $\pi$-periodic and twice continuously differentiable on $(x,x+\pi)$. 
Hence the boundary terms vanish when we integrate by parts once on the right side of~\eqref{eqn-periodic-deriv0}. 
The function $\beta^2  p$ is continuously differentiable and $\pi$-periodic, so its derivative $[\beta^2  p]'$ is also $\pi$-periodic. 
Hence the boundary term $f(u) [\beta^2  p]'(u)|_{u=x}^{u=x+\pi}$ which arises when integrating $f''(u) \beta^2(u)  p(u)$ by parts twice vanishes.
We may therefore integrate by parts on the right side of~\eqref{eqn-periodic-deriv0} to get
\eqb \label{eqn-periodic-deriv}
\left.\frac{\partial}{\partial t}\right|_{t=0} \BB E_p\left( f(X_t) \right) 
= \int_x^{x+\pi} f(u)  \left(  - \frac{d}{du} [ \alpha(u) p (u) ] + \frac12  \frac{d^2}{du^2} [ \beta^2(u)  p (u)]   \right)  \,du 
 = 0 ,
\eqe
where the last inequality is by~\eqref{eqn-kolmogorov}.  

For $s \geq 0$, let $T_s f (u) := \BB E_u(f(X_s))$. Since $f,\alpha$, and $\beta$ are $\pi$-periodic, so is $T_s f$. Furthermore, by Lemma~\ref{lem-hormander}, the function $T_s f$ is smooth. 
By the Markov property, 
\eqbn
\BB E_p\left( T_s f(X_t) \right)  = \BB E_p\left( \BB E_{X_t}(f(X_s) ) \right) = \BB E_p (f(X_{t+s}) ) .
\eqen
Therefore,~\eqref{eqn-periodic-deriv} applied with $T_s f$ instead of $f$ gives $\frac{d}{ds} \BB E_p(f(X_s)) = 0$. That is, $\BB E_p(f(X_s)) = \BB E_p(f(X_0)) = \int_x^{x+\pi} f(u) p(u) \,du$. This holds for every twice continuously differentiable $\pi$-periodic function $f$. Hence the law of $(e^{2 iX_t})_{t\geq 0}$ is stationary when $X_0$ is sampled from $  p(u) \BB 1_{[x,x+\pi]}(u)\,du$. 
\end{proof}

\bibliography{cibib,complex-sle-supplementary-bib}

\def\cprime{$'$}
\begin{thebibliography}{BCDM12}

\bibitem[AS08]{alberts-shef-bdy-dim}
T.~Alberts and S.~Sheffield.
\newblock Hausdorff dimension of the {SLE} curve intersected with the real
  line.
\newblock {\em Electron. J. Probab.}, 13:no. 40, 1166--1188, 2008,
  \arxiv{0711.4070}. \MR{2430703 (2009e:60025)}

\bibitem[BCDM12]{bcd-evolution-families1}
F.~Bracci, M.~D. Contreras, and S.~D\'{\i}az-Madrigal.
\newblock Evolution families and the {L}oewner equation {I}: the unit disc.
\newblock {\em J. Reine Angew. Math.}, 672:1--37, 2012, \arxiv{0807.1594}.
  \MR{2995431}

\bibitem[Bef08]{beffara-dim}
V.~Beffara.
\newblock The dimension of the {SLE} curves.
\newblock {\em Ann. Probab.}, 36(4):1421--1452, 2008, \arxiv{math/0211322}.
  \MR{2435854 (2009e:60026)}

\bibitem[BN]{bn-sle-notes}
N.~{B}erestycki and J.~{N}orris.
\newblock {L}ectures on {S}chramm-{L}oewner {E}volution.
\newblock Available at \url{http://www.statslab.cam.ac.uk/~james/Lectures/}.

\bibitem[BP]{bp-lqg-notes}
N.~{Berestycki} and E.~{Powell}.
\newblock {G}aussian {f}ree {f}ield, {L}iouville {q}uantum {g}ravity, and
  {G}aussian multiplicative chaos.
\newblock {A}vailable at
  \url{https://homepage.univie.ac.at/nathanael.berestycki/Articles/master.pdf}.

\bibitem[CN06]{camia-newman-sle6}
F.~Camia and C.~M. Newman.
\newblock Two-dimensional critical percolation: the full scaling limit.
\newblock {\em Comm. Math. Phys.}, 268(1):1--38, 2006, \arxiv{math/0605035}.
  \MR{2249794}

\bibitem[DG20]{dg-supercritical-lfpp}
J.~{Ding} and E.~{Gwynne}.
\newblock {Tightness of supercritical Liouville first passage percolation}.
\newblock {\em {J}ournal of the {E}uropean {M}athematical {S}ociety}, to
  appear, 2020, \arxiv{2005.13576}.

\bibitem[DG21]{dg-confluence}
J.~{Ding} and E.~{Gwynne}.
\newblock {Regularity and confluence of geodesics for the supercritical
  Liouville quantum gravity metric}.
\newblock {\em ArXiv e-prints}, April 2021, \arxiv{2104.06502}.

\bibitem[DG23]{dg-uniqueness}
J.~Ding and E.~Gwynne.
\newblock Uniqueness of the critical and supercritical {L}iouville quantum
  gravity metrics.
\newblock {\em Proc. Lond. Math. Soc. (3)}, 126(1):216--333, 2023,
  \arxiv{2110.00177}. \MR{4535021}

\bibitem[DMS21]{wedges}
B.~Duplantier, J.~Miller, and S.~Sheffield.
\newblock Liouville quantum gravity as a mating of trees.
\newblock {\em Ast\'{e}risque}, (427):viii+257, 2021, \arxiv{1409.7055}.
  \MR{4340069}

\bibitem[Dub09a]{dubedat-duality}
J.~Dub{\'e}dat.
\newblock Duality of {S}chramm-{L}oewner evolutions.
\newblock {\em Ann. Sci. \'Ec. Norm. Sup\'er. (4)}, 42(5):697--724, 2009,
  \arxiv{0711.1884}. \MR{2571956 (2011g:60151)}

\bibitem[Dub09b]{dubedat-coupling}
J.~Dub{\'e}dat.
\newblock S{LE} and the free field: partition functions and couplings.
\newblock {\em J. Amer. Math. Soc.}, 22(4):995--1054, 2009, \arxiv{0712.3018}.
  \MR{2525778 (2011d:60242)}

\bibitem[Dur10]{durrett}
R.~Durrett.
\newblock {\em Probability: theory and examples}.
\newblock Cambridge Series in Statistical and Probabilistic Mathematics.
  Cambridge University Press, Cambridge, fourth edition, 2010. \MR{2722836
  (2011e:60001)}

\bibitem[Hai11]{hairer-hormander}
M.~Hairer.
\newblock On {M}alliavin's proof of {H}\"{o}rmander's theorem.
\newblock {\em Bull. Sci. Math.}, 135(6-7):650--666, 2011, \arxiv{1103.1998}.
  \MR{2838095}

\bibitem[{Hua}18]{huang-complex-insertion}
Y.~{Huang}.
\newblock {Path integral approach to analytic continuation of Liouville theory:
  the pencil region}.
\newblock {\em ArXiv e-prints}, September 2018, 1809.08650.

\bibitem[JSW18]{jsw-imaginary-gmc}
J.~{Junnila}, E.~{Saksman}, and C.~{Webb}.
\newblock {Imaginary multiplicative chaos: Moments, regularity and connections
  to the Ising model}.
\newblock {\em ArXiv e-prints}, June 2018, \arxiv{1806.02118}.

\bibitem[JSW19]{jsw-decompositions}
J.~Junnila, E.~Saksman, and C.~Webb.
\newblock Decompositions of log-correlated fields with applications.
\newblock {\em Ann. Appl. Probab.}, 29(6):3786--3820, 2019, \arxiv{1808.06838}.
  \MR{4047992}

\bibitem[Ken09]{kennedy-sle-sim}
T.~Kennedy.
\newblock Numerical computations for the {S}chramm-{L}oewner evolution.
\newblock {\em J. Stat. Phys.}, 137(5-6):839--856, 2009, \arxiv{0909.2438}.
  \MR{2570752}

\bibitem[{Lac}20]{lacoin-complex-gmc}
H.~{Lacoin}.
\newblock {Convergence in law for Complex Gaussian Multiplicative Chaos in
  phase III}.
\newblock {\em ArXiv e-prints}, November 2020, \arxiv{2011.08033}.

\bibitem[Law05]{lawler-book}
G.~F. Lawler.
\newblock {\em Conformally invariant processes in the plane}, volume 114 of
  {\em Mathematical Surveys and Monographs}.
\newblock American Mathematical Society, Providence, RI, 2005. \MR{2129588
  (2006i:60003)}

\bibitem[LRV15]{rhodes-vargas-complex-gmt}
H.~Lacoin, R.~Rhodes, and V.~Vargas.
\newblock Complex {G}aussian multiplicative chaos.
\newblock {\em Comm. Math. Phys.}, 337(2):569--632, 2015, \arxiv{1307.6117}.
  \MR{3339158}

\bibitem[LSW04]{lsw-lerw-ust}
G.~F. Lawler, O.~Schramm, and W.~Werner.
\newblock Conformal invariance of planar loop-erased random walks and uniform
  spanning trees.
\newblock {\em Ann. Probab.}, 32(1B):939--995, 2004, \arxiv{math/0112234}.
  \MR{2044671 (2005f:82043)}

\bibitem[LU21]{lu-complex-loewner}
J.~{Lind} and J.~{Utley}.
\newblock {Phase Transition for a Family of Complex-driven Loewner Hulls}.
\newblock {\em ArXiv e-prints}, 2021, \arxiv{2106.14940}.

\bibitem[LW13]{lawler-werness-multipoint}
G.~F. Lawler and B.~M. Werness.
\newblock Multi-point {G}reen's functions for {SLE} and an estimate of
  {B}effara.
\newblock {\em Ann. Probab.}, 41(3A):1513--1555, 2013, \arxiv{1011.3551}.
  \MR{3098683}

\bibitem[MS16]{ig1}
J.~Miller and S.~Sheffield.
\newblock Imaginary geometry {I}: interacting {SLE}s.
\newblock {\em Probab. Theory Related Fields}, 164(3-4):553--705, 2016,
  \arxiv{1201.1496}. \MR{3477777}

\bibitem[MS17]{ig4}
J.~Miller and S.~Sheffield.
\newblock Imaginary geometry {IV}: interior rays, whole-plane reversibility,
  and space-filling trees.
\newblock {\em Probab. Theory Related Fields}, 169(3-4):729--869, 2017,
  \arxiv{1302.4738}. \MR{3719057}

\bibitem[MW17]{miller-wu-dim}
J.~Miller and H.~Wu.
\newblock Intersections of {SLE} {P}aths: the double and cut point dimension of
  {SLE}.
\newblock {\em Probab. Theory Related Fields}, 167(1-2):45--105, 2017,
  \arxiv{1303.4725}. \MR{3602842}

\bibitem[{Pfe}21]{pfeffer-supercritical-lqg}
J.~{Pfeffer}.
\newblock {Weak Liouville quantum gravity metrics with matter central charge
  $\mathbf{c} \in (-\infty, 25)$}.
\newblock {\em ArXiv e-prints}, April 2021, \arxiv{2104.04020}.

\bibitem[Pom92]{pom-book}
C.~Pommerenke.
\newblock {\em Boundary behaviour of conformal maps}, volume 299 of {\em
  Grundlehren der Mathematischen Wissenschaften [Fundamental Principles of
  Mathematical Sciences]}.
\newblock Springer-Verlag, Berlin, 1992. \MR{1217706 (95b:30008)}

\bibitem[Pro05]{protter}
P.~E. Protter.
\newblock {\em Stochastic integration and differential equations}, volume~21 of
  {\em Stochastic Modelling and Applied Probability}.
\newblock Springer-Verlag, Berlin, 2005.
\newblock Second edition. Version 2.1, Corrected third printing. \MR{2273672}

\bibitem[RS]{rs-complex-sle}
S.~Rohde and S.~Schramm.
\newblock Unpublished.

\bibitem[RS05]{schramm-sle}
S.~Rohde and O.~Schramm.
\newblock Basic properties of {SLE}.
\newblock {\em Ann. of Math. (2)}, 161(2):883--924, 2005, \arxiv{math/0106036}.
  \MR{2153402 (2006f:60093)}

\bibitem[Sch00]{schramm0}
O.~Schramm.
\newblock Scaling limits of loop-erased random walks and uniform spanning
  trees.
\newblock {\em Israel J. Math.}, 118:221--288, 2000, \arxiv{math/9904022}.
  \MR{1776084 (2001m:60227)}

\bibitem[She07]{shef-gff}
S.~Sheffield.
\newblock Gaussian free fields for mathematicians.
\newblock {\em Probab. Theory Related Fields}, 139(3-4):521--541, 2007,
  \arxiv{math/0312099}. \MR{2322706 (2008d:60120)}

\bibitem[She16]{shef-zipper}
S.~Sheffield.
\newblock Conformal weldings of random surfaces: {SLE} and the quantum gravity
  zipper.
\newblock {\em Ann. Probab.}, 44(5):3474--3545, 2016, \arxiv{1012.4797}.
  \MR{3551203}

\bibitem[Smi01]{smirnov-cardy}
S.~Smirnov.
\newblock Critical percolation in the plane: conformal invariance, {C}ardy's
  formula, scaling limits.
\newblock {\em C. R. Acad. Sci. Paris S\'er. I Math.}, 333(3):239--244, 2001,
  \arxiv{0909.4499}. \MR{1851632 (2002f:60193)}

\bibitem[Smi10]{smirnov-ising}
S.~Smirnov.
\newblock Conformal invariance in random cluster models. {I}. {H}olomorphic
  fermions in the {I}sing model.
\newblock {\em Ann. of Math. (2)}, 172(2):1435--1467, 2010, \arxiv{0708.0039}.
  \MR{2680496 (2011m:60302)}

\bibitem[SS13]{ss-contour}
O.~Schramm and S.~Sheffield.
\newblock A contour line of the continuum {G}aussian free field.
\newblock {\em Probab. Theory Related Fields}, 157(1-2):47--80, 2013,
  \arxiv{1008.2447}. \MR{3101840}

\bibitem[{Tra}17]{tran-complex-loewner}
H.~{Tran}.
\newblock {Loewner Equation driven by complex-valued functions}.
\newblock {\em ArXiv e-prints}, July 2017, \arxiv{1707.01023}.

\bibitem[{Wan}21]{wang-ld-survey}
Y.~{Wang}.
\newblock {Large deviations of Schramm-Loewner evolutions: A survey}.
\newblock {\em ArXiv e-prints}, February 2021, \arxiv{2102.07032}.

\bibitem[Wer04]{werner-notes}
W.~Werner.
\newblock Random planar curves and {S}chramm-{L}oewner evolutions.
\newblock In {\em Lectures on probability theory and statistics}, volume 1840
  of {\em Lecture Notes in Math.}, pages 107--195. Springer, Berlin, 2004,
  \arxiv{math/030335}. \MR{2079672 (2005m:60020)}

\bibitem[Zha08]{zhan-duality1}
D.~Zhan.
\newblock Duality of chordal {SLE}.
\newblock {\em Invent. Math.}, 174(2):309--353, 2008, \arxiv{0712.0332}.
  \MR{2439609 (2010f:60239)}

\bibitem[Zha10]{zhan-duality2}
D.~Zhan.
\newblock Duality of chordal {SLE}, {II}.
\newblock {\em Ann. Inst. Henri Poincar\'e Probab. Stat.}, 46(3):740--759,
  2010, \arxiv{0803.2223}. \MR{2682265 (2011i:60155)}

\end{thebibliography}
\bibliographystyle{hmralphaabbrv}

\end{document}